\documentclass{amsart}
\usepackage{amssymb}
\usepackage{bbm}
\usepackage{enumerate}
\usepackage{graphicx}
\usepackage{color}
\usepackage{type1cm}
\usepackage{hyperref}
\usepackage{subfig}
\usepackage{url}

\usepackage[colorinlistoftodos,bordercolor=orange,backgroundcolor=orange!20,linecolor=orange,textsize=scriptsize]{todonotes}

\newcommand{\myquote}{{\tt "}}

\newcommand{\nk}{K} %
\newcommand{\nkn}{\!\frac{\nk}{n}\!} %
\newcommand{\baralpha}{\bar{\alpha}} %
\newcommand{\ent}[1]{\left\lceil #1 \right\rceil}
\renewcommand{\geq}{\geqslant}
\renewcommand{\leq}{\leqslant}

\newcommand{\R}{\mathbb{R}}
\newcommand{\N}{\mathbb{N}}

\newcommand{\NEW}[1]{{\em #1}\index{#1}}

\newcommand{\rmax}{\R_{\max}}
\newcommand{\C}{\mathbb{C}}
\newcommand{\trop}{\operatorname{\mathsf{t}}}
\newcommand{\mon}{\operatorname{mon}}

\newcommand{\cF}{c}
\newcommand{\set}[2]{\{#1\mid #2\}}

\newcommand{\binomial}[2]{{#1 \choose #2}}
\newcommand{\detp}{\tilde{p}}
\newcommand{\E}{\mathcal{E}}
\newcommand{\U}{\mathcal{U}}
\newcommand{\col}[1]{{\sc{#1}}}
\newcommand{\nf}{\mathrm{nF}}
\newcommand{\diag}{\operatorname{diag}}

\theoremstyle{plain}%
\newtheorem{theorem}{Theorem}[section]
\newtheorem{lemma}[theorem]{Lemma}
\newtheorem{proposition}[theorem]{Proposition}
\newtheorem{corollary}[theorem]{Corollary}

\theoremstyle{definition}

\newtheorem{example}[theorem]{Example}

\theoremstyle{remark}
\newtheorem{remark}[theorem]{Remark}
\newtheorem{property}[theorem]{Property}

\catcode`\@=11
\def\refcounter#1{\protected@edef\@currentlabel
       {\csname the#1\endcsname}
}
\catcode`\@=12
\newcounter{assume}
\def\theassume{{\rm (A\arabic{assume})}}
\def\myitem{\refstepcounter{assume}\item[\theassume\hskip 0.72ex]}

\title{Log-majorization of the moduli of the eigenvalues of a matrix polynomial by tropical roots}

\author{Marianne Akian}
\thanks{The first and second authors were partially supported by PGMO, a joint program of EDF and FMJH (Fondation Math\'ematique Jacques Hadamard).}
\address{Marianne Akian,
INRIA and CMAP, \'Ecole 
polytechnique, UMR 7641, CNRS. Address:
CMAP, \'Ecole polytechnique,
Route de Saclay,
91128 Palaiseau Cedex, France.}
\email{Marianne.Akian@inria.fr}

\author{St\'ephane Gaubert}
\thanks{The last two authors were partly supported by the Arpege programme of the French National Agency of Research (ANR), project ``ASOPT'', number ANR-08-SEGI-005.}
\address{St\'ephane Gaubert,
INRIA and CMAP, \'Ecole 
polytechnique, UMR 7641, CNRS. Address:
CMAP, \'Ecole polytechnique,
Route de Saclay,
91128 Palaiseau Cedex, France.}
\email{Stephane.Gaubert@inria.fr}
\author{Meisam Sharify}
\thanks{This work was performed when the third author was with INRIA Saclay--\^Ile-de-France and CMAP, Ecole polytechnique.}
\address{Meisam Sharify,
School of Mathematics, Institute for Research in Fundamental Sciences (IPM), P.O. Box:19395-5746, Tehran, Iran.}
\email{meisam.sharify@ipm.ir}



\begin{document}

\begin{abstract}
We show that the sequence of moduli of the eigenvalues of a matrix polynomial is log-majorized, up to universal constants, by a sequence of ``tropical roots'' depending only on the norms of the matrix coefficients. These tropical roots are the non-differentiability points of an auxiliary tropical polynomial, or equivalently, the opposites of the slopes of its Newton polygon. This extends to the case of matrix polynomials some bounds obtained by Hadamard, Ostrowski and P\'olya for the roots of scalar polynomials. We also obtain new bounds in the scalar case, which are accurate for ``fewnomials'' or when the tropical roots are well separated.
\end{abstract}

\maketitle

\pagestyle{myheadings}
\thispagestyle{plain}
\markboth{Marianne Akian, St\'ephane Gaubert and Meisam Sharify }{Log-majorization of the moduli of the eigenvalues of a matrix polynomial by tropical roots}

\section{Introduction}
Let $p(x)=\sum_{j=0}^n{a_jx^j},\quad a_j\in \C$
be a polynomial of degree $n$ in a complex variable $x$.
Let  $\zeta_1,\dots,\zeta_n$ denote the roots of $p(x)$ arranged 
by non-decreasing modulus
(i.e., $|\zeta_1|\leq\dots\leq|\zeta_n|$).
We associate with $p$ the {\em tropical polynomial} $\trop p(x)$, 
defined for all nonnegative numbers $x$ by
\[ \trop p(x):=\max_{0\leq j\leq n}|a_j|x^j
\enspace .
\]
The {\em tropical roots} of $\trop p$, $\alpha_1,\dots,\alpha_n$,
ordered by non-decreasing value
(i.e., $\alpha_1\leq \dots\leq \alpha_n$), are defined as the non-differentiability
points of the function $\trop p$, counted with certain multiplicities. They
coincide with the exponential of the opposite of the slopes of
the edges of a Newton polygon, defined by Hadamard~\cite{hadamard1893}
and Ostrowski~\cite{Ostrowski1,Ostrowski2}
as the upper boundary of the convex hull of the set of points $\{(j,\log |a_j|)\mid 0\leq j\leq n\}$. The logarithms of these roots were called the {\em inclinaisons num\'eriques} by Ostrowski. One interest of these roots is that
they can be easily computed (linear number of arithmetic operations
and comparisons).
See Section~\ref{sec:tropical_polynomials} below for details. 

Hadamard was probably the first to prove a log-majorization type inequality for the modulus of the roots of a scalar polynomial by 
using what we call today the tropical roots. His result (page 201
of~\cite{hadamard1893}, third inequality) can be restated
as follows in tropical terms:
\begin{equation}
\label{chapter3:eq:hadamard_bound}
\frac{|\zeta_1\zeta_2\ldots\zeta_k|}{\alpha_1\cdots\alpha_k}\geq \frac{1}{k+1}\enspace.
\end{equation}
This bound, proved in passing in a memoir devoted to
the Riemann zeta function, remained apparently not so well known.
In particular, the special case $|\zeta_1|/\alpha_1 \geq 1/2$
is equivalent to the homogeneous form of the classical bound
of Cauchy, established later on by Fujiwara~\cite{fujiwara},
and a weaker inequality, with $\alpha_1^k$ at the denominator
instead of $\alpha_1\cdots\alpha_k$, appeared later on in the work
of Specht~\cite{specht1938}.

Ostrowski proved several bounds on the
roots of a polynomial in his work on the method
of Graeffe~\cite{Ostrowski1,Ostrowski2},
in which he used again the Newton polygon considered by Hadamard.
In particular, he obtained the following upper bound
(see~\cite[\S 7]{Ostrowski1}),
\begin{equation}
\label{eq:ostrowski}
\frac{|\zeta_1\zeta_2\cdots\zeta_k|}{\alpha_1\cdots\alpha_k }
\leq
\binomial{n}{k} \enspace,
\end{equation}
which can be thought of as a generalization of a
``reverse'' of the Cauchy inequality due
to Birkhoff~\cite{birkhoff} (corresponding to the case $k=1$ in~\eqref{eq:ostrowski}).
He also gave a different proof of a variant of~\eqref{chapter3:eq:hadamard_bound}, 
with the constant $1/(2k)$ instead of $1/(k+1)$, and 
reported a private communication of P\'olya, 
leading to a tighter constant
\begin{equation}
\label{eq:polyabound}
\frac{|\zeta_1\zeta_2\cdots\zeta_k|}{\alpha_1\cdots\alpha_k }
\geq
\frac{1}{\sqrt{\E(k) (k+1)}}\enspace,
\end{equation}
with
\begin{equation}\label{defEc}
\E(k):=\left(\frac{k+1}{k}\right)^{k} <e\enspace.\end{equation}

In this paper, we generalize the bounds of Hadamard, Ostrowski,
and P\'olya, to the case of a matrix polynomial
\begin{equation}\label{matrixpoly}
P(\lambda)=A_0+A_1\lambda+\dots+A_d\lambda^d, \quad A_j\in\C^{n\times n},\quad 0\leq j\leq d\enspace.
\end{equation}
We now associate
with the matrix polynomial $P$ the tropical polynomial 
\begin{equation}
\label{eq:tropdefini}
\trop p(x):=\max_{0\leq j\leq d}\|A_j\|x^j \enspace,
\end{equation}
where $\|\cdot \|$ is a norm on the space of matrices,
and show that the moduli of the roots $\zeta_1,\dots,\zeta_{nd}$
of $P$ can still be controlled in terms of the tropical roots 
$\alpha_1,\dots,\alpha_n$ of $\trop p$.
Our results give in particular bounds on the ratios
$|\zeta_1\dots\zeta_{nk}|/(\alpha_1\dots\alpha_{k})^n$,
which extend and refine the above bounds.
In particular, in Theorem~\ref{thm:mainthmmatrox}, we extend 
the lower bound~\eqref{eq:polyabound} of P\'olya 
to the matrix polynomial case, and in Theorem~\ref{thm:newthem}, 
we extend the upper bound~\eqref{eq:ostrowski} of Ostrowski.
Note that, for the lower bounds to be similar to
the one of P\'olya and in some sense independent of the dimension
of the matrices, the norm $\|\cdot \|$ must satisfy 
assumptions described in Section~\ref{subsec-norm},
which hold in particular for the normalized Frobenius norm.

We also obtain other lower bounds that are new 
even in the case of scalar polynomials.
In particular, in Theorem~\ref{thm:mainthmnnz}, we obtain a lower bound which 
may be tighter for ``fewnomials''. In Theorems~\ref{thm:mainthmmatrox3}
and~\ref{thm:mainthmpoly3}, we obtain general lower bounds, which extend 
the bound of P\'olya and its extension to the matrix case,
and which may be much tighter when the
tropical roots are sufficiently separated.
Then, all together our results show that the
tropical roots give tight estimates of the moduli of the
eigenvalues if the tropical roots are sufficiently
separated and if certain matrices are sufficiently well
conditioned. 

The results of the present paper combine ideas from
max-plus algebra and tropical geometry, and numerical linear algebra.
In~\cite{ABG96,abg04,abg04b,akianbapatgaubert2016}, Akian, Bapat, and Gaubert 
studied the eigenvalues and eigenvectors of matrices and matrix polynomials
whose entries (the $A_j$) are functions, for instance Puiseux series, 
of a (perturbation) parameter. It is shown there that the 
leading exponents of the Puiseux series
representing the different eigenvalues (resp.\ eigenvectors) coincide, 
under some genericity conditions, 
with the ``tropical eigenvalues'' (resp.\ eigenvectors)
of the tropical  matrix polynomial with entries  equal to
the leading exponents of the entries of the initial  matrix polynomial.
This can be interpreted in the light of tropical
geometry (see~\cite{viro,itenberg,sturmfels2005}),
using the notion of non-archimedean amoeba~\cite{kapranov}
 with respect to the usual 
non-archimedean valuation on the field of Puiseux series
(taking the leading exponent).
Since amoebas with respect to the archimedean
valuation $z\mapsto \log |z|$ on the field of complex numbers $\C$
can be approximated by non-archimedean amoebas~\cite{passare,purbhoo},
tropical eigenvalues are expected to 
provide approximations of the log of moduli of the classical eigenvalues.
The above bounds of Hadamard, Ostrowski, and P\'olya
can be interpreted as an approximation result of this nature,
for roots of univariate scalar valued polynomials.
In the case of a matrix with non-negative coefficients,
Friedland~\cite{shmuel} established a bound for its spectral radius (or its
Perron eigenvalue), which can be interpreted
as a bound of the maximal eigenvalue of $A$ by the maximal
tropical eigenvalue of the valuation of $A$~\cite{abg05}.
Similar bounds for the other eigenvalues
have been established by Akian, Gaubert and Marchesini in~\cite{marche}.

Here,  we replaced the valuation on $\C$ by the 
``valuation'' $A \mapsto \log \|A\|$ on the ring of $n\times n$ matrices
over $\C$, which leads to~\eqref{eq:tropdefini}. 
The tropical roots obtained in this way are easy to compute
and we shall see they provide a good approximation of the moduli 
of the eigenvalues under reasonable assumptions.
The idea of using the norm instead of the usual valuation was inspired by 
several works in numerical linear algebra, suggesting that the
information of the norms is relevant.
For instance, Higham and Tisseur~\cite{highamtisseur03} extended
to matrix polynomials the bound of Cauchy (related to the
special case $k=1$ in the Hadamard-Ostrowski-P\'olya inequality),
by using the norms of $A_0^{-1} A_j$ and $A_d^{-1} A_j$.
Fan, Lin and  Van Dooren~\cite{vandooren00} introduced a scaling based
on the norms of the matrix coefficients of a matrix quadratic polynomial.

The interest of tropical roots is not limited to theoretical bounds,
they 
can be used to perform scalings allowing one to improve the 
accuracy of the numerical 
computation of the eigenvalues of a matrix polynomial.
In~\cite{posta09}, the tropical polynomial of~\eqref{eq:tropdefini}
was initially introduced to refine  the results of~\cite{vandooren00}.
It was shown there (see also~\cite{meisam2011})
that the tropical roots of $\trop p$
can be used to perform scalings allowing one to improve the backward stability
of eigenvalue computations for a matrix polynomial $P$. 
In the special case of quadratic pencils,
this was confirmed by a work of Hamarling,
Munro, and Tisseur~\cite{hamarling}, who
implemented the tropical scaling as an option of {\tt quadeig},
and gave theoretical estimates showing that the tropical
scaling does reduce the backward error under some conditions.
We shall explain this scaling in more details 
in Section~\ref{sec-scaling}. We also explain
why the tropical scaling helps to
compute the eigenvalues of a matrix polynomial, as soon as all the ratios
$\zeta_{n(k-1)+l}/\alpha_k$ with $k=1,\ldots, d$ and
$\ell=1,\ldots, n$ are close to $1$.
We also illustrate this behavior numerically.

Some bounds on the modulus of the eigenvalues of $P$, 
involving the tropical roots of $\trop p$, also appeared in~\cite{posta09}
in the case in which $d=2$, and in~\cite{meisam2011} for the 
smallest and largest tropical roots when $d\geq 2$,
which may be seen as particular cases of the lower bounds of the present
paper.

Let us finally point out some further related works. 
Bini used in~\cite{bini96} what we call the tropical roots
(from the Newton polygon technique) to initialize
the Aberth method of computation of the roots of a scalar polynomial.
Also, Malajovich and Zubelli applied Ostrowski's analysis
to effective root solving~\cite{zubelli}. 
Finally, Bini, Noferini, and Sharify recently
proved in~\cite{binisharify} some location results 
by using the tropical roots of~\eqref{eq:tropdefini}
for the eigenvalues 
of a specific class of matrix polynomials 
such that the coefficient matrices are unitary up to some constants.
They also generalized the classical Pellet's theorem to 
the case of matrix polynomials, which involve
the norms of the $A_k^{-1} A_j$, except in the specific above class,
where they involve the norms of the $A_j$.
They also relate it to the Newton polygon method, thus to the tropical roots.
Moreover, Melman~\cite{melman} proved a different generalization
and some variations of Pellet's theorem, which are in general 
less costly since they involve the norms of the $A_j$ and $A_k^{-1}$
only (we need to compute $O(d)$ norms instead of $O(d^2)$, each of
these computations requiring at least an inversion or a matrix product).
 \label{pelletpage}
Also a recent work by Noferini, Sharify, Tisseur~\cite{SharifyTisseur}
improved these results. 
Note that the information obtained by Pellet type methods
is complementary and incomparable to the one obtained by the present 
log-majorization inequalities, see Remark~\ref{rem-pellet}.

The paper is organized as follows. 
We first recall a variation on Jensen formula due to Landau (Section~\ref{subsec-landau}) and extend it to the matrix case (Section~\ref{subsec-genlandau}), exploiting some technical results on matrix norms (Section~\ref{subsec-norm}).
In Section~\ref{sec:newtonpoly}, we recall the construction of the Newton polygon and of tropical roots. Then,
we derive in Section~\ref{subsec-generaltropical} a general estimate
for the modulus of the product of $\nk$ smallest eigenvalues of
a matrix polynomial in terms of tropical roots.
This leads to various explicit lower bounds,
stated in Section~\ref{sec:Main_results}, which extend
the Hadamard-Ostrowski-Poly\'a inequality to the matrix case.
These bounds
are proved in Section~\ref{sec:scalar_poly} in the scalar
case, and then in Section~\ref{sec:matrix_poly} in general.
An upper bound (reverse inequality) is stated in Section~\ref{sec-upperbound}
and proved in Section~\ref{sec-proof-upperbound}.
In Section~\ref{sec:examples}
we provide examples showing the tightness of the lower bounds.
In Section~\ref{sec-scaling}, we explain the tropical scaling 
and the interest of the above bounds in this context.

\section{Estimates of the eigenvalues of matrix polynomials using norms of matrices}\label{sec-generalinterest}

\subsection{An inequality of Landau}\label{subsec-landau}
An ingredient of our results is a bound on the 
modulus of the product of the $k$ smallest roots of a
polynomial in terms of its coefficients. It
is a consequence of Jensen's formula, 
derived by Landau in~\cite{landau}, building on an 
earlier observation of Lindel\"of~\cite{lindelof}.
\begin{lemma}[\cite{landau}] 
\label{lem:polynomialineq1new}
Let $\zeta_1, \dots, \zeta_d$ be the roots of a polynomial
$p(x)=\sum_{j=0}^d a_j x^j$, arranged by non-decreasing modulus, 
and assume that $a_0\neq 0$. For all $1\leq k\leq d$, we have
\begin{equation}
\label{eq:ineqmodulrootsnew}
\log{|\zeta_1\dots\zeta_k|}
\geq-\inf_{r>0}\frac{1}{2}\log\left(\sum_{j=0}^d \frac{|a_j|^2}{|a_0|^2} r^{2(j-k)}\right)\enspace.
\end{equation}
\end{lemma}
We include the short proof, as its idea will be used in the
extension to the matrix case.
\begin{proof}
The formula of Jensen~\cite{Jensen1899} shows that if $\zeta_1,\dots,\zeta_k$ are the roots of $p(z)$ in
the closed disk of $\C$ of radius $r$, counted with multiplicities, then
\[
\log|\zeta_1\dots\zeta_{k}|= k\log r+\log |p(0)|-\frac{1}{2\pi}\int_0^{2\pi} \log{|p(re^{i\theta})|}d\theta \enspace .
\]
It follows that, for all $r>0$, and $k=1,\ldots, d$:
\begin{equation}
\label{eq:principeq11}  
\log|\zeta_1\dots\zeta_{k}|\geq k\log r+\log |p(0)|-\frac{1}{2\pi}\int_0^{2\pi} \log{|p(re^{i\theta})|}d\theta \enspace.
\end{equation}
Using the comparison between the geometric and the $L^2$ mean,
together with Parseval's identity, we get
\begin{eqnarray*}
\lefteqn{\frac{1}{2\pi}\int_0^{2\pi}\log|p(re^{i\theta})|d\theta \leq 
\frac{1}{2}\log\big( \frac{1}{2\pi}\int_0^{2\pi}|p(re^{i\theta})|^2d\theta\big)}\\
&=& \frac{1}{2}\log\big( \frac{1}{2\pi}\int_0^{2\pi}|\sum_{j=0}^d a_jr^je^{ij\theta}|^2d\theta\big)
= \frac{1}{2}\log\big( \sum_{j=0}^d|a_j|^2r^{2j}\big) \enspace .
\end{eqnarray*}
Gathering this inequality with~\eqref{eq:principeq11} yields
\[ \log|\zeta_1\dots\zeta_{k}|\geq k\log r+\log |a_0|-
\frac{1}{2}\log\big( \sum_{j=0}^d|a_j|^2r^{2j}\big) \enspace .
\]
Since this holds for all $r>0$, this shows the inequality of the lemma.
\end{proof}
\subsection{Preliminary results on matrix norms}\label{subsec-norm}
In order to generalize Lemma~\ref{lem:polynomialineq1new}
to matrix polynomials, and to derive
effective lower bounds for eigenvalues, we
need to introduce technical assumptions on matrix norms,
like the following one:
\begin{itemize}
\myitem\label{assump1}
 $|\det A|\leq \|A\|^n$, for all $A\in\C^{n\times n}$.
\end{itemize}
The
\NEW{normalized Frobenius norm},
\[ \|A\|_\nf:= \left(\frac{1}{n}\sum_{i,j=1}^n |A_{ij}|^2\right)^{\frac{1}{2}},\quad
\forall A\in \C^{n\times n}\enspace ,\]
will be of special interest, as some of our estimates rely on $L_2$
methods. 
Therefore, the next assumption will 
also be considered:
\begin{itemize}
\stepcounter{enumi}
\myitem\label{assump2}
There exists $Q,Q'\in\C^{n\times n}$, such that 
$\det Q=\det Q'=1$ and $\|QAQ'\|_\nf\leq \|A\|$, for all $A\in\C^{n\times n}$.
\end{itemize}
We next show that Assumption~\ref{assump2} implies Assumption~\ref{assump1},
and that a number of commonly used norms satisfy Assumption~\ref{assump1}
or Assumption~\ref{assump2}.

For all $p\in[1,\infty]$, and $n\geq 1$, we shall denote by $\|\cdot\|_p$ the
$\ell^p$ norm of $\C^n$:
$\|v\|_p:=(\sum_{i=1}^n |v_i|^p)^{1/p}$ for $p<\infty$ and
$\|v\|_\infty:=\max_{i=1,\ldots, n} |v_i|$.
In particular, 
$\|\cdot\|_2$ is the Euclidean norm.
The norm on the space of matrices $\C^{n\times n}$ induced by the
norm $\|\cdot\|_p$ on $\C^n$ 
will also be denoted by $\|\cdot\|_p$ (the same norm is used for the domain and the range of $A$). In particular the norm 
$\|\cdot\|_2$  on $\C^{n\times n}$ is the spectral norm.
Moreover, for all $p\in[1,\infty)$, and $n\geq 1$, we shall 
denote by $\|\cdot\|_{*p}$ the following normalized Schatten $p$-norm 
on the space of matrices $\C^{n\times n}$:
$\|A\|_{*p}:= (\frac{1}{n}\sum_{i=1}^n \sigma_i^p)^{1/p}$,
where $\sigma_1,\ldots,\sigma_n$ are the singular values of $A$
(the eigenvalues of $\sqrt{A^*A}$,
where $A^*$ denotes the conjugate transpose of $A$).
Then, the normalized Schatten $2$-norm coincides with the normalized 
Frobenius norm. Recall that the (unnormalized) Schatten $1$-norm is
also called the trace norm or the Ky Fan $n$-norm.
We shall also denote by $\|A\|_{*\infty}:=\max_{i=1,\ldots, n} \sigma_i$ the
Schatten $\infty$-norm, which coincides with the spectral norm.

\begin{property}\label{propertyfrob}
The normalized Frobenius norm satisfies Assumption~\ref{assump1}.
Therefore, Assumption~\ref{assump2} implies Assumption~\ref{assump1}.
\end{property}
\begin{proof}
From Hadamard's inequality, we get that 
$|\det A|\leq \|A^{(1)}\|_2\cdots\|A^{(n)}\|_2$,
for all $A\in \C^{n\times n}$,
where $A^{(j)}$ denotes the $j$-th column of $A$.
Then, since the geometric mean is less than or equal to the arithmetic mean (or 
simply by the concavity of the logarithm), we get that 
$|\det A|^{2/n}\leq \frac{1}{n} (\|A^{(1)}\|_2^2+\dots+\|A^{(n)}\|_2^2)=\|A\|_\nf^2$,
which implies that the normalized Frobenius norm 
satisfies Assumption~\ref{assump1}.
Since for all $Q,Q'\in\C^{n\times n}$ such that $\det Q=\det Q'=1$,
we get $|\det A|=|\det (QA Q')|\leq \|Q A Q'\|_\nf^n$ for all $A\in \C^{n\times n}$,
Assumption~\ref{assump2} implies Assumption~\ref{assump1}.
\end{proof}

\begin{property}
For all $p\in [1,\infty ]$, the norm $\|\cdot\|_p$ on $\C^{n\times n}$
satisfies~\ref{assump2}.
\end{property}
\begin{proof}
Indeed, let us denote by $\|A\|_{p,q}$ the norm induced by
the norms $\|\cdot\|_p$ and $\|\cdot\|_q$ on the range and 
domain $\C^n$ of $A$ respectively, which means that 
$\|A\|_{p,q}=\max\{ \|A v\|_q/\|v\|_p,\; v\in\C^n, v\neq 0\}$.
Since (for all $v\in \C^n$) $p\in [1,\infty]\mapsto \|v\|_p\in\R^+$
is a nonincreasing map, we get that 
$\|A\|_{p,q}\leq \|A\|_{p',q'}$ when $p\leq p'$ and $q\geq q'$.
In particular $\|A\|_p=\|A\|_{p,p}\geq \|A\|_{2,\infty}$ when
$2\leq p$, and $\|A\|_p=\|A\|_{p,p}\geq \|A\|_{1,2}$  when $p\leq 2$.
In addition,  it is easy to show that
$\|A\|_{2,\infty}= \max\{\|(A^*)^{(i)}\|_2, i=1,\dots, n\}$.
Since the last expression is greater or equal to
$\|A\|_\nf$, we deduce that $\|A\|_p\geq \|A\|_{2,\infty}\geq \|A\|_\nf$
for all $A\in\C^{n\times n}$ and $2\leq p\leq \infty$.
Similarly, $\|A\|_{1,2}= \max\{\|A^{(i)}\|_2, i=1,\ldots n\}
\geq \|A\|_\nf$, hence $\|A\|_p\geq \|A\|_{1,2}\geq \|A\|_\nf$
for all $A\in\C^{n\times n}$ and $1\leq p\leq 2$.
This shows that $\|\cdot\|_p$ satisfies~\ref{assump2} with $Q=Q'=I$ the identity 
matrix.
\end{proof}

This result implies that all norms induced 
by the norm $\|v\|_d=\|Q v\|_p$ in the domain of $A$ and
the norm $\|v\|_r=\|Q' v\|_p$ in the range of $A$ with the same $p$,
but possibly different matrices $Q,Q'$ such that $\det Q=\det Q'=1$,
satisfy~\ref{assump2}.

\begin{property}
If $\|\cdot\|$ is the norm on $\C^{n\times n}$ induced by a norm
on $\C^n$ (the same for the domain and the range of matrices), then it 
satisfies~\ref{assump1}.
\end{property}
\begin{proof}
Since $\det A$ is the product of the eigenvalues of $A$ 
counted with multiplicities, we get that
$|\det A|\leq \rho(A)^n$, where $\rho(A)$ is the 
spectral radius of $A$. Since $\rho(A)\leq \|A\|$ for all $A\in\C^{n\times n}$
and all induced norms $\|\cdot\|$ on the space of matrices, we get 
the result.
\end{proof}

\begin{property}\label{prop-schatten}
The normalized Schatten $p$-norm $\|\cdot\|_{*p}$
on $\C^{n\times n}$ satisfies~\ref{assump1}, for all $p\in [1,\infty ]$.
It  satisfies~\ref{assump2} if and only if $p\geq 2$.
Moreover, for $p\in [1,2)$, the least constant $\eta$ such that 
$\eta \|\cdot\|_{*p}$ satisfies~\ref{assump2} is given by $\eta=n^{1/p-1/2}>1$.
\end{property}
\begin{proof}
Since $|\det A|$ is the product of the singular values of $A$ 
counted with multiplicities, we get that
$|\det A|^p= \sigma_1^p\cdots \sigma_n^p$, and using
that the geometric mean is less than or equal to the arithmetic mean,
we obtain that 
$|\det A|^{p/n}\leq \frac{1}{n} (\sigma_1^p+\cdots + \sigma_n^p)=\|A\|_{*p}^p$,
which implies that the normalized Schatten $p$-norm 
satisfies Assumption~\ref{assump1}.

We have that $p\mapsto \|A\|_{*p}$ is a nondecreasing map. Hence,
when $p\in [2,\infty ]$, $\|A\|_{*p}\geq \|A\|_{*2}=\|A\|_\nf$, 
for all $A\in\C^{n\times n}$, which implies that $\|\cdot\|_{*p}$
satisfies~\ref{assump2} when $p\geq 2$.
Let $\eta>0$ be the least constant such that 
$\eta \|\cdot\|_{*p}$ satisfies~\ref{assump2} and let us show that
$\eta=n^{1/p-1/2}$ when $p<2$.
This will implies in particular that  $\eta>1$ hence $\|\cdot\|_{*p}$ 
does not satisfy~\ref{assump2} for $p<2$, which will finishes the proof
of the equivalence ``$\|\cdot\|_{*p}$ satisfies~\ref{assump2} 
if and only if $p\geq 2$''.
Since $n^{1/p} \|\cdot\|_{*p}$ is the unnormalized 
Schatten norm, which is nonincreasing with respect to $p$,
we get that, for all $p\in [1,2)$, and $A\in\C^{n\times n}$,
$n^{1/2} \|A\|_\nf\leq n^{1/p} \|A\|_{*p}$, which implies that
$\eta\leq n^{1/p-1/2}$.
Let us fix $p$, $Q, Q'\in\C^{n\times n}$
such that $\det Q=\det Q'=1$ and 
$\|Q A Q'\|_\nf \leq \eta  \|A\|_{*p}$,  for all $A\in\C^{n\times n}$.
For all $i=1,\ldots, n$, let us consider the matrix $A$ whose 
entries are all zero but the entry $ii$ equal to $1$.
Then, the singular values of $A$ are all zero except one which is equal to $1$,
so that $\|A\|_{*p}=(1/n)^{1/p}$.
Moreover $Q AQ'=Q^{(i)}Q'_{(i)}$, where $Q'_{(i)}$ denotes the $i$-th raw of 
$Q'$. Hence $\|Q AQ'\|_\nf=(1/n)^{1/2} \|Q^{(i)}\|_2 \|(Q'^*)^{(i)}\|_2$.
From $\|Q A Q'\|_\nf \leq \eta  \|A\|_{*p}$, we deduce that
$\|Q^{(i)}\|_2 \|(Q'^*)^{(i)}\|_2\leq \eta (1/n)^{1/p-1/2}$.
Since $\det Q=\det Q'=1$, we get using Hadamard's inequality, 
$1\leq \|Q^{(1)}\|_2 \|(Q'^*)^{(1)}\|_2\cdots\|Q^{(n)}\|_2 \|(Q'^*)^{(n)}\|_2
\leq ( \eta (1/n)^{1/p-1/2})^n$, which shows that
$\eta \geq n^{1/p-1/2}$.
We deduce that $\eta=n^{1/p-1/2}>1$ for $p>2$, which completes the proof.
\end{proof}

\subsection{Generalization of the inequality of Landau to matrix polynomials}
\label{subsec-genlandau}
The following generalization of Lemma~\ref{lem:polynomialineq1new} will be a key tool to establish
lower bounds for the eigenvalues of matrix polynomials.
\begin{lemma}
\label{lem:newmain}
Consider the matrix polynomial $P$ with degree $d$ defined 
in~\eqref{matrixpoly},
and let $\zeta_1,\dots,\zeta_{nd}$ denote its eigenvalues,
arranged by non-decreasing modulus.
Assume that $\|\cdot\|$ is any norm on the space of matrices 
satisfying~\ref{assump1}, that $\det A_0\neq0$ and let $c=\frac{|\det A_0|}{\|A_0\|^n}$.
Then, for $\nk\in\{1,\ldots, nd\}$, we have
\begin{equation}
\label{eq:matrixineq1}
\log |\zeta_1\dots\zeta_{\nk}|
\geq
\log c-n \inf_{r>0} \log\left(\sum_{j=0}^d \frac{\|A_j\|}{\|A_0\|}r^{j-\nkn } \right)\enspace.
\end{equation}
When $\|\cdot\|$ satisfies~\ref{assump2},
the previous bound can be improved as follows
\begin{equation}
\label{eq:matrixineq2}
\log |\zeta_1\dots\zeta_{\nk}|
\geq 
\log c -n \inf_{r>0} \frac{1}{2}\log\left(\sum_{j=0}^d \frac{\|A_j\|^2}{\|A_0\|^2}r^{2(j-\nkn )} \right)\enspace.
\end{equation}
\end{lemma}
\begin{proof}
From Inequality~\eqref{eq:principeq11}
applied to $\detp(\lambda)=\det P(\lambda)$, we get, for all $r>0$,
\begin{equation}
\label{thm:jensenmatx}
\log|\zeta_1\dots\zeta_{\nk}|\geq \nk\log r+\log |\det A_0|-\frac{1}{2\pi}\int_0^{2\pi} \log{|\det P(re^{i\theta})|}d\theta \enspace .
\end{equation}
Since $\|\cdot\|$ satisfies~\ref{assump1}, we have 
\begin{equation}\label{firstineq-nonfrob}
|\det P(re^{i\theta})| \leq
\|P(re^{i\theta})\|^n =
\|\sum_{j=0}^dA_j(re^{i\theta})^j\|^n 
\leq (\sum_{j=0}^d\|A_j\|r^j)^n\enspace,
\end{equation}
for all $\theta\in [0,2\pi]$.
Gathering~\eqref{firstineq-nonfrob} with~\eqref{thm:jensenmatx}, we
obtain~\eqref{eq:matrixineq1}.

Assume now that $\|\cdot\|$ satisfies~\ref{assump2} with some matrices
$Q,Q'$ such that $\det Q=\det Q'=1$, and let us prove \eqref{eq:matrixineq2}.
Since from Property~\ref{propertyfrob},
the normalized Frobenius norm satisfies~\ref{assump1}, 
we get that $|\det P(re^{i\theta})|= |\det (QP (re^{i\theta}) Q')|\leq
\|QP(re^{i\theta})Q'\|_\nf^n$.
Now using the comparison between geometric and $L^2$ means, we deduce
\begin{eqnarray}
\frac{1}{2\pi}\int_0^{2\pi}\log |\det P(re^{i\theta})|d\theta &\leq &
n\frac{1}{2 \pi}\int_0^{2\pi}\log \| QP(re^{i\theta})Q'\|_\nf d\theta \nonumber\\
&\leq & 
\frac{n}{2}\log\left(\frac{1}{2\pi}\int_0^{2\pi}\|QP(re^{i\theta})Q'\|_\nf^{2}
d\theta\right)\enspace.\label{firstineq-frob}
\end{eqnarray}
From the formula of the normalized Frobenius norm, we get by 
applying Parseval's identity to each coordinate $(QP(re^{i\theta})Q')_{\ell m}$
\begin{eqnarray*}
\lefteqn{\frac{1}{2\pi}\int_0^{2\pi}\|QP(re^{i\theta})Q'\|_\nf^{2}d\theta
= \frac{1}{n}\sum_{\ell,m=1}^n \left( \frac{1}{2\pi}\int_0^{2\pi}
|(QP(re^{i\theta})Q')_{\ell,m}|^2d\theta \right)} \\
&=& \frac{1}{n}\sum_{\ell,m=1}^n\left(\sum_{j=0}^d (| (QA_jQ')_{\ell,m} | r^j)^2 \right)
=\sum_{j=0}^d (\|QA_jQ'\|_\nf^2r^{2j})
\leq\sum_{j=0}^d (\|A_j\|^2r^{2j})\enspace.
\end{eqnarray*}
Gathering this inequality with Inequalities~\eqref{firstineq-frob} 
and~\eqref{thm:jensenmatx}, we obtain~\eqref{eq:matrixineq2}.
\end{proof}

\section{Tropical polynomials and numerical Newton polygons}
\subsection{Preliminary results on tropical roots}
\label{sec:newtonpoly}
\label{sec:tropical_polynomials}
We recall here basic results on tropical polynomials of one variable.
See for instance~\cite{bcoq,viro,itenberg} for more 
background on tropical polynomials from different perspectives.

Let $\rmax$ denotes the set $\R\cup\{-\infty\}$. 
A (max-plus) {\em tropical polynomial} $f$ is a function
of a variable $x\in \rmax$ of the form
\begin{equation}
\label{e-def-f}
f(x) = \max_{0\leq j\leq d} (f_j +j x)  \enspace,
\end{equation}
where $d$ is an integer, and $f_0,\dots,f_d$ are given elements of 
$\rmax$. 
We say that $f$ is of {\em degree $d$} if $f_d\neq -\infty$.
We shall assume that at least one of the coefficients $f_0,\dots,f_d$ is finite (i.e., that $f$ is not the tropical ``zero polynomial'').
Then, $f$ is a real valued convex function, piecewise affine, with integer slopes.

Cuninghame-Green and Meijer showed~\cite{cuning80} that the analogue of the fundamental theorem of algebra holds in the tropical setting, i.e., $f(x)$ can be written uniquely as 
\[ 
f(x)= f_{d}+\sum_{j=1}^{d} \max(x,\alpha_j)\enspace,
\]
where $\alpha_1\leq\cdots\leq \alpha_d\in\rmax$.
The numbers  $\alpha_1,\ldots,\alpha_d$ are called the \NEW{tropical roots}. The finite tropical roots can be checked to be the
points at which the maximum in the expression~\eqref{e-def-f} of $f(x)$ is attained at least twice, whereas $-\infty$ arises as a tropical root if 
$f_0=-\infty$. 
The \NEW{multiplicity} of a root $\alpha$ is defined
as the cardinality of the set 
$\set{j\in \{1,\dots, d\}}{\alpha_{j} = \alpha }$. 

The multiplicity of a finite root $\alpha$ can be checked to coincide
with the variation of the derivative of the map $f$ at point $\alpha$,
whereas the multiplicity of the root $-\infty$ is given by
$\inf\{j\mid f_j \neq -\infty\}$ or by the slope of the map $f$ at $-\infty$.
The notion of tropical roots is an elementary special case of the notion of tropical variety which has arisen recently in tropical geometry~\cite{itenberg}.

The tropical roots can be computed by the following variant
of the classical Newton polygon construction.
Define the {\em Newton polygon} $\Delta(f)$ 
of $f$ to be the upper boundary of the convex hull of
the region
\[\set{(j,\lambda)\in\N\times\R}
{\lambda \leq f_j ,\;  0\leq j\leq d}  \subset \R^2  \enspace .
\]
The latter region is the hypograph of the map $\tilde{f}:\R\to\R\cup\{-\infty\}$
such that $\tilde{f}(j)=f_j$ for $j\in\{0,\ldots, d\}$ and 
$\tilde{f}(x)=-\infty$ otherwise.
Hence, the Newton polygon coincides with the graph of the
concave hull of $\tilde{f}$. Let us denote by $\hat{f}$ this
concave hull and by $\hat{f}_j$ its value at an integer point $j$. 
Then the Newton polygon of $f$
consists of (linear) segments relying (or passing through)
the points $(j,\hat{f}_j)$, $j\in\{0,\ldots, d\}$.

The following result relies on standard Legendre-Fenchel duality.
A proof can be found in \cite[Proposition 2.10]{abg04b}
(for min-plus polynomials instead of max-plus polynomials).

\begin{proposition}\label{correspondance}
There is a bijection between the set of finite tropical roots 
of $f$ and the set of segments of the Newton polygon $\Delta(f)$:
the tropical root corresponding to a segment is the opposite
of its slope, and the multiplicity of this root is the
length of this segment (measured by the difference of
the absciss\ae\ of its endpoints).
\end{proposition}

\begin{remark}\label{tropical_compute}\rm
The Graham scan algorithm~\cite{Graham72} allows one to compute the 
convex hull of a finite set of points by making $O(n)$ arithmetical operations 
and comparisons, provided that the given set of points is already arranged by increasing absciss\ae, which is the case in the present setting.
It follows that the tropical roots, counted with
multiplicities, can be computed in linear time (see also~\cite[Proposition 1]{posta09}). 
\end{remark}

In particular, the maximal tropical root is given by
\[
\alpha_d = \max_{0\leq j\leq d} \frac{f_j-f_d}{d-j} \enspace .
\]

\begin{example}\label{ex1} \rm 
Consider
\[
f(x)= \max(0,1+x,6+2x,4+4x,9+8x,5+10x,1+16x)  \enspace.
\]
The graph of $f$ and the Newton polygon of $f$ are shown
in Figure~\ref{fig:newtongraph}.  The tropical roots are $-3$, $-1/2$, and $1$, 
with respective multiplicities $2$, $6$ and $8$.
\begin{figure}[htbp]
\begin{minipage}{0.49\linewidth}
\begin{flushleft}
\includegraphics[scale=0.33]{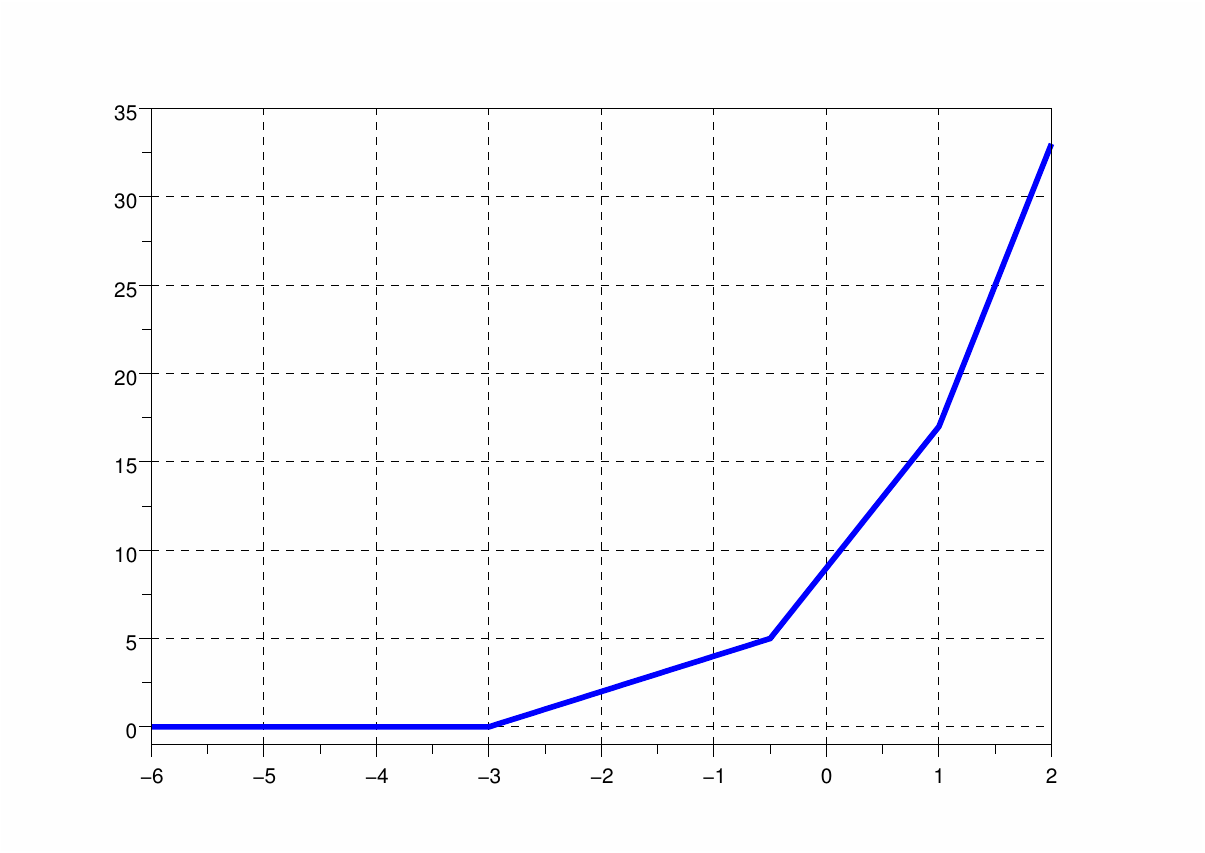}
\end{flushleft}
\end{minipage}
\begin{minipage}{0.49\linewidth}
\centering
\includegraphics[scale=0.33]{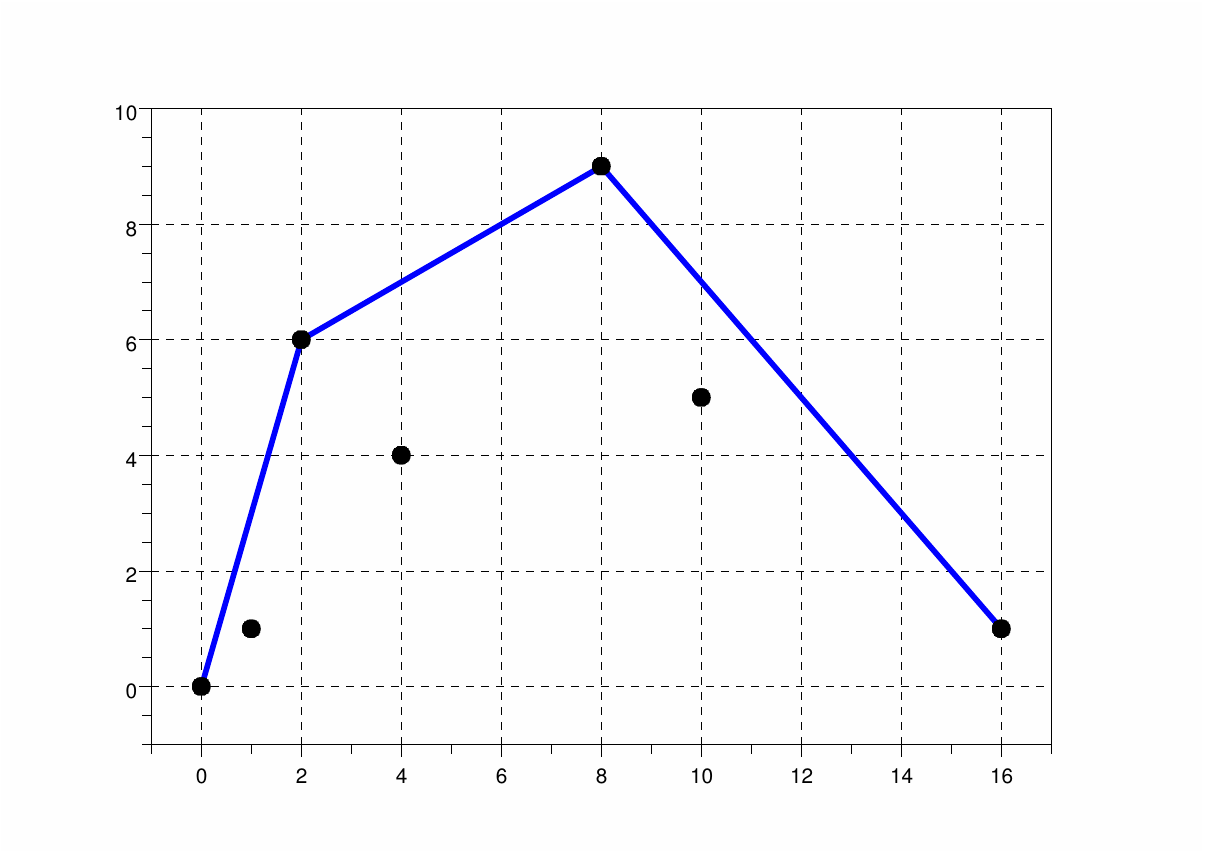}
\end{minipage}
\caption{Graph of the tropical polynomial function $f$ on left, and of its Newton polygon on right.}
\label{fig:newtongraph}
\end{figure}

\end{example}

The notion of root also applies with trivial changes to the ``max-times'' model
of the tropical structure, in which polynomial functions now have
the form
\[
\trop p(x)= \max_{0\leq j\leq d} a_j x^j \enspace ,
\]
where $a_0,\dots,a_d$ are nonnegative numbers, and the variable
$x$ now takes nonnegative values.
Then, the {\em tropical roots} of $\trop p(x)$ are, by definition, the exponentials of the tropical roots of its log-exp transformation
$f(x):= \log \trop p(\exp(x))=\max_{0\leq j\leq d}(\log a_j+jx)$. 

In the sequel we shall consider max-times polynomials associated with
usual scalar or matrix polynomials.
We shall need the following result which follows from the above definitions
and properties.
\begin{proposition}
\label{prop:coefbnd}
Let $\trop p(x)= \max_{0\leq j\leq d} a_j x^j$ be a max-times 
polynomial of degree $d$, with $a_j\geq 0$, $j=0,\ldots d$.
Assume that $a_0\neq 0$.  Let
 $0< \alpha_1\leq \cdots\leq \alpha_d$ denote the tropical roots of $\trop p$
arranged in non-decreasing order. Then,
\[
a_j\leq a_0\prod_{\ell=1}^j\alpha_\ell^{-1},\quad j=0,\ldots d \enspace.
\]
Moreover, let $k$ be the abscissa of a vertex of the Newton polygon of $p$,
then $a_k> 0$ and 
\[
 \begin{array}{ll}
a_j\leq a_k\prod_{\ell=k+1}^j\alpha_\ell^{-1},& \text{for} \; j=k,\ldots d \enspace,\\
a_j\leq a_k\prod_{\ell=j+1}^k\alpha_\ell,& \text{for} \; j=0,\ldots k \enspace.
\end{array}
\]
\end{proposition}
\begin{proof}
Let $f(x):= \max_{0\leq j\leq d}(\log a_j+jx)$. 
By definition
$\log \alpha_j$, $j=1,\ldots d$, are the tropical roots of $f$
and by the assumption $a_0\neq 0$, they are all finite.
Let $f_j=\log a_j$, for $j=0,\ldots, d$ be the coefficients of the
tropical polynomial $f$. From the above observations,
the Newton polygon of $f$ coincides with the 
graph of the concave hull $\hat{f}$
of the map $\tilde{f}: \R\to\R\cup\{-\infty\}$ 
such that $\tilde{f}(j)=f_j$ for $j\in\{0,\ldots, d\}$ and 
$\tilde{f}(x)=-\infty$ otherwise.
Then the Newton polygon of $f$
consists of (linear) segments relying (or passing through)
the points $(j,\hat{f}_j)$, $j\in\{0,\ldots, d\}$.
By Proposition~\ref{correspondance}, the finite tropical roots are the
 opposites of the slopes of $\hat{f}$ 
and their multiplicities are the lengths of the segments where $\hat{f}$ has
this slope. This means that
$\log\alpha_j=\hat{f}_{j-1}-\hat{f}_j$, hence
$\hat{f}_{j}=\hat{f}_0-\sum_{\ell=1}^j \log\alpha_\ell$, and using that
$\hat{f}$ is above the map $\tilde{f}$, and that both maps
coincide at the boundary point $j=0$, we get the first inequality
of the proposition.
If now $k$ is the abscissa of a vertex of the Newton polygon of $p$,
then $(k,\hat{f}_{k})$ is an exposed point of the hypograph of $\tilde{f}$, which
implies  that $\hat{f}_k=\tilde{f}(k)=f_k=\log a_k$.
Since $\hat{f}_{j}=\hat{f}_{i}-\sum_{\ell=i+1}^j \log\alpha_\ell$,
for all $j\geq i$,  we get the two last inequalities of the proposition.
\end{proof}
\subsection{A general lower bound using tropical roots}
\label{subsec-generaltropical}
Using the same method as in the 
proof of P\'olya's inequality reproduced in~\cite{Ostrowski1},
we derive from Landau's inequality the following result which involves now the 
tropical roots $\alpha_k$ instead of the moduli $|a_k|$ of the coefficients
of $p$.
\begin{proposition}\label{genlowercor}
Let $\zeta_1, \dots, \zeta_d$ be the roots of a univariate scalar polynomial
$p$ of degree $d$, $p(x)=\sum_{j=0}^d a_j x^j$,
arranged by non-decreasing modulus, 
and assume that $a_0\neq 0$. Let $\alpha_1,\dots,\alpha_d$ denote the 
 tropical roots of the associated tropical polynomial
$\trop p(x)=\max_{0\leq j\leq d}|a_j|x^j$, 
arranged in non-decreasing order. Then, for all $k\in\{1,\ldots, d\}$, 
the inequality
\begin{align}\label{i_edge_bound-scalar-reordered}
\frac{|\zeta_1\dots\zeta_{k}|}{\alpha_1\dots\alpha_k}\geq L_k \enspace ,
\end{align}
holds with
\begin{equation}\label{generalLk}
(L_k)^{-2}=\inf_{\xi>0} \left(
\sum_{j=0, a_j\neq 0}^d \beta_{k,j}^2\xi^{j-k} \right)\enspace,
\end{equation}
and
\begin{equation}\label{generalbetakl}
\beta_{k,j}:=
\begin{cases}
\prod_{\ell=j+1}^{k-1}\left(\frac{\alpha_\ell}{\alpha_k}\right)
&\text{if}\; j< k-1\enspace,\\
\prod_{\ell=k+1}^j\left(\frac{\alpha_k}{\alpha_\ell}\right)
&\text{if}\; j> k\enspace,\\
1&\text{if}\; j= k-1\;\text{or}\; k\enspace,
\end{cases}
\end{equation}
so that
in particular $\beta_{k,j}\leq 1$ for all $k,j=0,\ldots, d$. 
\end{proposition}
\begin{proof}
By applying the change of variable $r=\alpha_k\sqrt{\xi}$ in the inequality of
Lemma~\ref{lem:polynomialineq1new}, we get
\begin{eqnarray*}
\log{|\zeta_1\dots\zeta_k|}
&\geq & \sup_{\xi>0}\left(-\frac{1}{2}\log\left(\sum_{j=0}^d \frac{|a_j|^2}{|a_0|^2} (\alpha_k^2\xi)^{j-k}\right)\right)\\
&= & \log(\alpha_1\cdots \alpha_k)+\sup_{\xi>0}\left(-\frac{1}{2}\log \left(\prod_{\ell=1}^k\alpha_\ell^2\left(\sum_{j=0}^d \frac{|a_j|^2}{|a_0|^2} (\alpha_k^2\xi)^{j-k}\right)\right)\right)\enspace.
\end{eqnarray*}
Applying Proposition~\ref{prop:coefbnd} to the max-times polynomial $\trop p$,
we get $|a_j|\leq |a_0|\prod_{\ell=1}^j\alpha_\ell^{-1}$
for all $j=0,\ldots , d$, which with the above inequality 
yields~\eqref{i_edge_bound-scalar-reordered} with
\[ (L_k)^{-2}:= 
\inf_{\xi>0}\left(\sum_{j=0, a_j\neq 0}^d \big(\frac{(\prod_{\ell=1}^k\alpha_\ell)(\prod_{\ell=1}^j\alpha_\ell^{-1})}{\alpha_k^{k-j}}\big)^2\xi^{j-k} \right)
\] 
which can be written in the form~\eqref{generalLk} 
with $\beta_{k,j}$ as in~\eqref{generalbetakl}.
Moreover, since $\alpha_j$ is nondecreasing with respect to $j$, we get that
all the $\beta_{k,j}$ are less than or equal to $1$.
\end{proof}

The following is a matrix version of Proposition~\ref{genlowercor}.
It is proved along the same lines.
For any real number $x$, $\ent{x}$ will denote the least integer $\geq x$:
$\ent{x}-1< x\leq \ent{x}$.
\begin{proposition}\label{genlowercormatrix}
Let $A_0$,\ldots, $A_d$, $P$, $\zeta_1,\dots,\zeta_{nd}$, $\|\cdot\|$,  $c$, 
be as in the first part of 
Lemma~\ref{lem:newmain},
and let $\alpha_1,\dots,\alpha_d$ be the tropical roots
of the tropical polynomial of~\eqref{eq:tropdefini},
arranged in non-decreasing order.
For $\nk\in\{1,\ldots, nd\}$, denote $\baralpha_\nk=\alpha_{\ent{\nkn}}$.
Then, for all $\nk\in\{1,\ldots, nd\}$,
the inequality
\begin{equation}
\label{i_edge_bound-reordered}
\frac{|\zeta_1\dots\zeta_{\nk}|}{\baralpha_1\dots\baralpha_{\nk}}\geq
\cF (L_{\nkn})^n \enspace ,
\end{equation}
holds 
with $L_k$, $0<k\leq d$, such that:
\begin{equation}\label{generalLkmatrix}
(L_k)^{-1}=\inf_{\xi>0} \left(
\sum_{j=0, A_j\neq 0}^d \beta_{\ent{k},j}\;\xi^{j-k} \right)\enspace,
\end{equation}
with $\beta$ defined as in~\eqref{generalbetakl}.

Moreover, when $\|\cdot\|$ satisfies~\ref{assump2},
the constant $L_k$ in~\eqref{i_edge_bound-reordered} can be replaced by the greater
one:
\begin{equation}\label{generalLkmatrixfrob}
(L_k)^{-2}=\inf_{\xi>0} \left(
\sum_{j=0, A_j\neq 0}^d (\beta_{\ent{k},j})^2\;\xi^{j-k} \right)\enspace.
\end{equation}
\end{proposition}

Note that $\baralpha_1,\ldots, \baralpha_{\nk}$ are the tropical
roots of the tropical polynomial $(\trop p)^n$
arranged in non-decreasing order.
Hence the above result compares the eigenvalues of $P$ with the
tropical roots of $(\trop p)^n$.

\begin{remark}\rm
In the above proof, the only ingredient to deduce the result of
Proposition~\ref{genlowercor} from that of Lemma~\ref{lem:polynomialineq1new}
is the inequality  
 $|a_j|\leq  |a_0|\prod_{\ell=1}^j\alpha_\ell^{-1}$ for all $j=0,\ldots , d$.
By Proposition~\ref{prop:coefbnd}, this inequality is an equality
for any abscissa $j$  of a vertex  of the Newton polygon of $p$.
Hence, if all the tropical roots of the tropical polynomial
of~\eqref{eq:tropdefini} are simple (which means that all $k=0,\ldots, d$
are absciss\ae\ of a vertex of the Newton polygon of $p$),
the inequalities of Proposition~\ref{genlowercor} and 
Lemma~\ref{lem:polynomialineq1new} are equivalent.
The same is true for the inequalities of Proposition~{genlowercormatrix}
and Lemma~\ref{lem:newmain}.
\end{remark}

\section{Explicit lower bounds in terms of tropical roots}
\label{sec:Main_results}
We now derive
from Proposition~\ref{genlowercormatrix}
explicit lower bounds for the product
of the $\nk$ smallest eigenvalues of a matrix polynomial.
These new bounds can be easily computed (in $O(d)$ time,
either in floating point or exact arithmetics,
except for the ones of Corollary~\ref{cor-strongest} 
which can be computed in $O(d^2)$ time), as soon
as $c$ and the tropical roots are given.
Recall that the tropical roots themselves
can be computed in $O(d)$ time as soon as the coefficients of
the tropical polynomial,
that is the norms of the coefficient matrices, 
are given (see Remark~\ref{tropical_compute}).
Different bounds can be given, depending
on the information available on the  matrix polynomial.
Our first bound is useful for ``fewnomials''
(polynomials with few non-zero coefficients).
\begin{theorem}[Bounds involving the number of nonzero coefficients]
\label{thm:mainthmnnz}
Consider the matrix polynomial $P$ with degree $d$ defined 
in~\eqref{matrixpoly},
and let $\zeta_1,\dots,\zeta_{nd}$ denote its eigenvalues,
arranged by non-decreasing modulus.
Assume that $\|\cdot\|$ is any norm on the space of matrices 
satisfying~\ref{assump1}, that $\det A_0\neq0$ and let $c=\frac{|\det A_0|}{\|A_0\|^n}$.
Let $\alpha_1,\dots,\alpha_d$ be the tropical roots of 
the tropical polynomial of~\eqref{eq:tropdefini},
arranged in non-decreasing order, and 
for $\nk\in\{1,\ldots, nd\}$, denote $\baralpha_\nk=\alpha_{\ent{\nkn}}$.
Also let $\mon P$ denote the number of nonzero monomials of $P$.
Then, for all $\nk\in\{1,\ldots, nd\}$, we have
\begin{equation}
\label{i_edge_bound}
\frac{|\zeta_1\dots\zeta_{\nk}|}{\baralpha_1\dots\baralpha_{\nk}}\geq
\cF (L_{\nkn})^n \enspace ,
\end{equation}
in particular, for all $1\leq k\leq d$, we have
\begin{equation}
\label{i_edge_bound_n}
\frac{|\zeta_1\dots\zeta_{nk}|}{(\alpha_1\dots\alpha_{k})^n}\geq
\cF (L_{k})^n \enspace .
\end{equation}
with, for $0<k\leq d$, 
\begin{align}
\label{nnz_bound}
 L_{k}=\frac{1}{\mon P}\enspace .
\end{align}
Moreover, when $\|\cdot\|$ satisfies~\ref{assump2},
the constant $L_k$ can be replaced by the greater one:
\begin{equation}
\label{nnz_bound_fro}
 L_{k}=\frac{1}{\sqrt{\mon P}}\enspace.
\end{equation}
\end{theorem}
The proof
of all the results of this section will be given in Section~\ref{sec-prooflb}.
\begin{remark}\label{rk-scalar}\rm
When $n=1$ (thus the matrices $A_0,\dots,A_d$ are scalars),
any norm is proportional to the normalized Frobenius norm which 
is nothing but the modulus map $|\cdot|$, satisfies~\ref{assump2},
and for which $c=1$. 
Therefore the tropical roots of $\trop p$ are the same for all norms,
the best possible inequality~\eqref{i_edge_bound} is
\begin{align}\label{i_edge_bound-scalar}
\frac{|\zeta_1\dots\zeta_{k}|}{\alpha_1\dots\alpha_k}\geq L_k \enspace ,
\end{align}
and this inequality holds with $L_k$ as in~\eqref{nnz_bound_fro}.
\end{remark}
\begin{remark}\rm \label{rem23}
If a norm $\|\cdot\|$ on the space of matrices
satisfies~\ref{assump1} but not \ref{assump2}, we can
still obtain a bound of the form~{\rm (\ref{i_edge_bound},\ref{nnz_bound_fro})}
by changing the
constant $c$. Indeed, by the equivalence between norms on $\C^{n\times n}$,
for any norm $\|\cdot\|$ on $\C^{n\times n}$,
there exists a constant $\eta>0$ (which depends on $n$) such that 
$\|A\|_\nf\leq \eta \|A\|$ for all $A\in\C^{n\times n}$.
Then, the norm $N$ obtained by multiplying $\|\cdot\|$ by $\eta$
satisfies~\ref{assump2}, hence~{\rm (\ref{i_edge_bound},\ref{nnz_bound_fro})}.
There, the constant $L_k$ is given by~\eqref{nnz_bound_fro}, thus
independent of the norm, whereas the constant $c$ and the
tropical roots $\alpha_j$  are computed with respect to the norm $N$.
In particular $c=\frac{|\det A_0|}{(\eta \|A_0\|)^n}=c_0/\eta^n$,
where $c_0$ is the constant $c$ associated with the norm $\|\cdot\|$.
Moreover, if $\trop p$ denotes the tropical polynomial
associated with the norm $\|\cdot\|$, then
the tropical polynomial associated with the norm $N$ 
is equal to $\eta \trop p$, and its tropical roots 
are equal to the ones of $\trop p$.
Hence, we deduce that~{\rm (\ref{i_edge_bound},\ref{nnz_bound_fro})}
holds for $\|\cdot\|$ with $c/\eta^n$ instead of $c$.
However, if $\|\cdot\|$ satisfies~\ref{assump1}
but not~\ref{assump2}, then
$\eta>1$, so that the inequality derived 
from~{\rm (\ref{i_edge_bound},\ref{nnz_bound_fro})} with
$c/\eta^n$ instead of $c$
may be weaker than~{\rm (\ref{i_edge_bound},\ref{nnz_bound})}:
this is indeed the case if and only if $\eta>\sqrt{\mon{P}}$.
The same type of conclusions can be obtained for the 
lower bounds that are stated in the next theorems.
\end{remark}

The following theorem provides a lower bound generalizing the lower bound of P\'olya to the matrix case (since, as said in Remark~\ref{rk-scalar}, when $n=1$,
the modulus map is a norm satisfying~\ref{assump2}, and for which 
Inequality~\eqref{i_edge_bound} reduces to~\eqref{i_edge_bound-scalar}). Up to the constant $c$, the following bounds are independent of the coefficients of the matrix polynomial $P$.
\begin{theorem}[Universal bound]
\label{thm:mainthmmatrox}
Let $A_0$,\ldots, $A_d$, $P$, $\zeta_1,\dots,\zeta_{nd}$, $\|\cdot\|$,  $c$, and
$\alpha_1,\dots,\alpha_d$, $\baralpha_1$,\ldots, $\baralpha_{nd}$ 
be as in the first part of Theorem~\ref{thm:mainthmnnz}.
Then, for all $\nk\in\{1,\ldots, nd\}$, Inequality~\eqref{i_edge_bound} holds with $L_{k}$, $0<k\leq d$, defined as follows:
\begin{equation}
\label{gen_polya}
L_k= \frac{ 1}{\E(k) (k+1)} 
\enspace,
\end{equation}
where $\E$ is defined as in~\eqref{defEc}.
Moreover, when $\|\cdot\|$ satisfies~\ref{assump2},
Inequality~\eqref{i_edge_bound} holds with the greater constant:
\begin{equation}
\label{gen_polya_fro}
L_k= \frac{ 1}{\sqrt{\E(k) (k+1)}}
\enspace.
\end{equation}
\end{theorem}

The lower bound of P\'olya, and its matrix version given above,
are tight only for $k$ small. By symmetry, one can obtain a tight
lower bound when $k$ is close to $d$ (this
was already noted by Ostrowski in the scalar case~\cite{Ostrowski1}).
Theorem~\ref{thm:mainthmmatrox3} below will allow us to 
obtain a tight lower bound when $k$ lies ``in the middle''
of the interval $[0,d]$, by using the
comparison between the tropical roots. 
Unlike the lower bound of P\'olya, this theorem gives a bound which is not anymore 
independent of the coefficients of the polynomial $P$, although it 
depends only on a small information, namely a ratio measuring
the {\em separation} between some tropical roots. 
Thus, the bound will involve a coefficient $\U(k,\delta)$ depending
both on the index $k$ and on the parameter $\delta$ (the ratio).
It will be specially useful in situations in which $\delta$ is small.
This coefficient $\U(k,\delta)$ is defined as follows:
\begin{equation}\label{defU}
\U(k,\delta):=\E(k)\left(\frac{1+\delta+\sqrt{k^2(1-\delta)^2+4\delta}}{1-\delta}\right)\enspace,\quad \text{for}\; k\geq 0\;\text{and}\; 0\leq \delta \leq 1\enspace,
\end{equation}
where $\E(k)$ is defined by~\eqref{defEc} for $k>0$, 
and $\E(0):=1$, and with the convention that $1/0=+\infty$, so that
$\U(k,1)=+\infty$.
It is easy to check that
\begin{equation}\label{propU}
 k+\frac{1+\delta}{1-\delta} \leq \U(k,\delta) <e \left( k+\frac{1+\sqrt{\delta}}{1-\sqrt{\delta}}\right) \enspace .\end{equation}
The following asymptotic regime should
be kept in mind:
\[ \U(k,\delta) \sim \E(k)(k+1) < e(k+1), \qquad \delta \to 0 \enspace .
\]
\begin{theorem}[Master lower bound]
\label{thm:mainthmmatrox3}
Let $A_0$,\ldots, $A_d$, $P$, $\zeta_1,\dots,\zeta_{nd}$, $\|\cdot\|$,  $c$, and
$\alpha_1,\dots,\alpha_d$, $\baralpha_1$,\ldots, $\baralpha_{nd}$ 
be as in the first part of Theorem~\ref{thm:mainthmnnz},
and denote $\alpha_0=0$ and $\alpha_{d+1}=+\infty$.

Let $\nk\in\{1,\ldots, nd\}$, $k^-,k^+$ be integers such that
$0\leq k^-<\ent{\nkn} \leq k^+-1\leq d$, 
and denote 
$\delta_-:=  \alpha_{k^-}/\alpha_{\ent{\nkn}}\leq 1$
and $\delta_+:= \alpha_{\ent{\nkn}}/\alpha_{k^+}\leq 1$.
Then, 
Inequality~\eqref{i_edge_bound} holds with $L_k$, $k^-<k \leq k^+-1$,
defined as follows:
\begin{align}
\label{e-def-Lk}
 L_{k}:=\max(L_{k}^\pm), \quad
L_{k}^+:=\frac{1}{\U(k^+-k-1,\delta_+)}, \;
L_{k}^-:=\frac{1}{\U(k-k^-,\delta_-)}\enspace ,
\end{align}
with the convention that $1/\infty=0$.
Moreover, when $\|\cdot\|$ satisfies~\ref{assump2}, the constant $L_k$ 
of~\eqref{i_edge_bound} can be replaced by the greater constant $L_k^*$:
\begin{align}
\label{L_k_j_fro3}
L_{k}^*:= &\max(L_{k}^{\pm*}),\quad
L_{k}^{+*}:=\frac{1}{\sqrt{\U(k^+-k-1,\delta_+^2)}},\;
L_{k}^{-*}:=\frac{1}{\sqrt{\U(k-k^-,\delta_-^2)}}\enspace .
\end{align}
\end{theorem}

Theorem~\ref{thm:mainthmmatrox3} generalizes Theorem~\ref{thm:mainthmmatrox}, and thus the lower bound of P\'olya.
Indeed, taking $k^-=0$, $k^+=d+1$, and using that $\U(k,0)=\E(k) (k+1)$,
we get that the constants $L_k^-\leq L_k$ and $L_k^{-*}\leq L_k^*$
of Theorem~\ref{thm:mainthmmatrox3} are exactly the constants $L_k$ 
of~\eqref{gen_polya} and~\eqref{gen_polya_fro} of
Theorem~\ref{thm:mainthmmatrox} respectively.
Moreover, Theorem~\ref{thm:mainthmmatrox3} 
is already new in the scalar case ($n=1$). 

Note that when $\delta_-=1$ in
Theorem~\ref{thm:mainthmmatrox3}, $\U(k-k^-,\delta_-)=+\infty$, so that 
$L_k^-=0$ and $L_k=L_k^+$. 
Similarly when $\delta_+=1$, we get $L_k^+=0$.
Hence, if $\delta_-=1=\delta_+$, we get $L_k=0$ so
that~\eqref{i_edge_bound} does not provide any information, although it is true.
However, if for instance $\delta_-=1$ and $\delta_+<1$, 
we get that $L_k=L_k^+>0$, which gives a positive 
lower bound in~\eqref{i_edge_bound}.

Applying Theorem~\ref{thm:mainthmmatrox3} to the particular case 
when $\nkn=k=k^-+1=k^+-1$ leads to the following formula for the
constants of~\eqref{e-def-Lk} and~\eqref{L_k_j_fro3}:
\begin{align}
\label{e-def-Lk1}
 L_{k}:=\max(L_{k}^{\pm}), \quad
L_{k}^{+}:=\frac{1-\sqrt{\delta_+}}{1+\sqrt{\delta_+}}, \;
L_{k}^{-}:=\frac{1-\delta_-}{4(1+\delta_-)}\enspace ,
\end{align}
and 
\begin{align}
\label{L_k_j_fro31}
L_{k}^*:= &\max(L_{k}^{\pm*}),\quad
L_{k}^{+*}:=\sqrt{\frac{1-\delta_+}{1+\delta_+}},\;
L_{k}^{-*}:=\frac{1}{2}\sqrt{\frac{1-\delta_-^2}{1+\delta_-^2}}\enspace .
\end{align}
In this special case, we obtain the following stronger lower bounds.

\begin{theorem}
\label{thm:mainthmpoly3}
Let us use the notations of  Theorem~\ref{thm:mainthmmatrox3},
and assume that $\nkn=k=k^-+1=k^+-1$.
Then, the statements of Theorem~\ref{thm:mainthmmatrox3} hold with 
Inequality~\eqref{i_edge_bound_n} instead of~\eqref{i_edge_bound} and the
constants $L_k$ and $L_k^*$ (given in~\eqref{e-def-Lk} and~\eqref{L_k_j_fro3},
or~\eqref{e-def-Lk1} and~\eqref{L_k_j_fro31}) replaced 
respectively by the greater constants $L_k^\sharp$ and $L_k^{*\sharp}$ given by:
\begin{equation}
L_{k}^\sharp:= \frac{1-\delta_{-}\delta_+}{(1+\sqrt{\delta_+})^2}
\enspace, \label{deltak-k+}
\end{equation}
and 
\begin{equation}
L_k^{*\sharp}:= \frac{\sqrt{1-\delta_{-}^2\delta_+^2}}{1+\delta_+}
\enspace. \label{deltak-k+-fro}
\end{equation}
\end{theorem}

We next indicate how the indices $k^+$ and $k^-$
should be chosen, for each $\nk\in\{1,\ldots, nd\}$,
in order to get the best lower bound $L_{\nkn}$. 

Let us use the notations of Theorem~\ref{thm:mainthmmatrox3},
so that $\alpha_1,\dots,\alpha_d$ are the tropical roots of 
the tropical polynomial $\trop p$ of~\eqref{eq:tropdefini},
$\alpha_0=0$ and $\alpha_{d+1}=+\infty$.
Let $k_0=0,k_1,\dots,k_q=d$ be the sequence of absciss\ae\ of the vertices
of the Newton polygon of $\trop p(x)$, as shown in Figure~\ref{newtonpolygon}.
For $j=1,\ldots,q$, we have,
$\alpha_{k_{j-1}+1}=\dots=\alpha_{k_{j}}<\alpha_{{k_j}+1}$. 
We also denote by
\begin{align}
\label{e-def-delta}
 \delta_j=\frac{\alpha_{k_j}}{\alpha_{k_{j}+1}}\enspace,
\end{align}
for $j=0,\ldots,q$, the parameters measuring the {\em separation} between the tropical roots,
in particular $\delta_0=\delta_q=0$.

\begin{figure}[htbp]
\begin{center}
\begin{picture}(0,0)%
\includegraphics{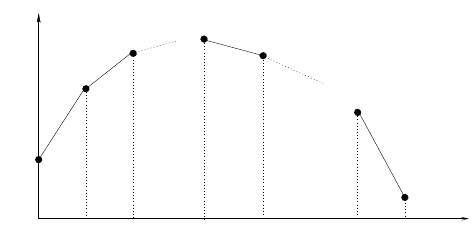}%
\end{picture}%
\setlength{\unitlength}{1657sp}%
\begingroup\makeatletter\ifx\SetFigFont\undefined%
\gdef\SetFigFont#1#2#3#4#5{%
  \reset@font\fontsize{#1}{#2pt}%
  \fontfamily{#3}\fontseries{#4}\fontshape{#5}%
  \selectfont}%
\fi\endgroup%
\begin{picture}(8937,4566)(1741,-6574)
\put(1756,-6436){\makebox(0,0)[lb]{\smash{{\SetFigFont{9}{10.8}{\rmdefault}{\mddefault}{\updefault}{\color[rgb]{0,0,0}$k_0=0$}%
}}}}
\put(3241,-6451){\makebox(0,0)[lb]{\smash{{\SetFigFont{9}{10.8}{\rmdefault}{\mddefault}{\updefault}{\color[rgb]{0,0,0}$k_1$}%
}}}}
\put(2836,-4696){\makebox(0,0)[lb]{\smash{{\SetFigFont{9}{10.8}{\rmdefault}{\mddefault}{\updefault}{\color[rgb]{0,0,0}$-\log\alpha_1$}%
}}}}
\put(4231,-6451){\makebox(0,0)[lb]{\smash{{\SetFigFont{9}{10.8}{\rmdefault}{\mddefault}{\updefault}{\color[rgb]{0,0,0}$k_2$}%
}}}}
\put(5536,-6451){\makebox(0,0)[lb]{\smash{{\SetFigFont{9}{10.8}{\rmdefault}{\mddefault}{\updefault}{\color[rgb]{0,0,0}$k_{j-1}$}%
}}}}
\put(6751,-6451){\makebox(0,0)[lb]{\smash{{\SetFigFont{9}{10.8}{\rmdefault}{\mddefault}{\updefault}{\color[rgb]{0,0,0}$k_{j}$}%
}}}}
\put(8461,-6451){\makebox(0,0)[lb]{\smash{{\SetFigFont{9}{10.8}{\rmdefault}{\mddefault}{\updefault}{\color[rgb]{0,0,0}$k_{q-1}$}%
}}}}
\put(9091,-4741){\makebox(0,0)[lb]{\smash{{\SetFigFont{9}{10.8}{\rmdefault}{\mddefault}{\updefault}{\color[rgb]{0,0,0}$-\log\alpha_{d}$}%
}}}}
\put(6211,-2761){\makebox(0,0)[lb]{\smash{{\SetFigFont{9}{10.8}{\rmdefault}{\mddefault}{\updefault}{\color[rgb]{0,0,0}$-\log\alpha_{k_j}$}%
}}}}
\put(2701,-2311){\makebox(0,0)[lb]{\smash{{\SetFigFont{9}{10.8}{\rmdefault}{\mddefault}{\updefault}{\color[rgb]{0,0,0}$\log\|A_k\|$}%
}}}}
\put(9541,-6451){\makebox(0,0)[lb]{\smash{{\SetFigFont{9}{10.8}{\rmdefault}{\mddefault}{\updefault}{\color[rgb]{0,0,0}$k_{q}=d$}%
}}}}
\put(10216,-5866){\makebox(0,0)[lb]{\smash{{\SetFigFont{9}{10.8}{\rmdefault}{\mddefault}{\updefault}{\color[rgb]{0,0,0}$k$}%
}}}}
\end{picture}%

\end{center}
\caption{Newton polygon corresponding to $\trop p(x)$.}
\label{newtonpolygon}
\end{figure}

\begin{proposition}\label{prop-optim}
Each of the constants $L_k$ and $L_k^{*}$ appearing in Theorem~\ref{thm:mainthmmatrox3},
is maximized by choosing 
$k^-=k_r$ for some $0\leq r\leq q$ such that $k_r<k$
and $k^+=k_s+1$ for  some $0\leq s\leq q$ such that $k_s\geq k$.
\end{proposition}

This proposition shows that, to apply Theorem~\ref{thm:mainthmmatrox3},
we may always require $k^-$ and $k^+-1$ to be absciss\ae\ of
vertices of the Newton polygon.
Optimizing the choice of $k^\pm$ ,
we readily arrive at the following corollary.

\begin{corollary}\label{cor-strongest}
Let us use the notations of  Theorem~\ref{thm:mainthmmatrox3}, and
let $k_0=0,k_1,\dots,k_q=d$ be the sequence of absciss\ae\ of \sloppy
the vertices of the Newton polygon of $\trop p(x)$, 
as shown in Figure~\ref{newtonpolygon}.
Then, the statements of  Theorem~\ref{thm:mainthmmatrox3} hold with the
constants $L_k$, $L_k^\pm$, $L_k^*$ and $L_k^{\pm *}$
 (given in~\eqref{e-def-Lk} and~\eqref{L_k_j_fro3}) replaced 
respectively by the following optimal ones:
\begin{subequations}
\label{L_k_j}
\begin{align}
L_{k}^{\text{opt}}:=& \max(L_{k}^{\pm ,\text{opt}}), \\
L_{k}^{+,\text{opt}}:=&\max_{j:\; k_j \geq k} \frac{1}{\U(k_j-k,\frac{\alpha_{\ent{k}}}{\alpha_{k_j+1}})},\\
L_{k}^{-,\text{opt}}:=&\max_{j:\; k_{j-1}< k}\frac{1}{\U(k-k_{j-1},\frac{\alpha_{k_{j-1}}}{\alpha_{\ent{k}}})}\enspace ,
\end{align}
\end{subequations}
and
\begin{subequations}
\label{L_k_j_fro}
\begin{align}
L_{k}^{*,\text{opt}}:=&\max(L_{k}^{\pm *,\text{opt}}), \\
L_{k}^{+*,\text{opt}}:=&\max_{j:\; k_j\geq k}\frac{1}{\sqrt{\U(k_j-k,(\frac{\alpha_{\ent{k}}}{\alpha_{k_j+1}})^2)}},\\
L_{k}^{-*,\text{opt}}:=&\max_{j:\; k_{j-1}< k}\frac{1}{\sqrt{\U(k-k_{j-1},
(\frac{\alpha_{k_{j-1}}}{\alpha_{\ent{k}}})^2)}}\enspace .
\end{align}
\end{subequations}
\end{corollary}

However, a simpler choice of $k^\pm$ consists in taking the nearest 
vertices of the Newton polygon in~\eqref{L_k_j} and~\eqref{L_k_j_fro},
which lead to the following corollary.

\begin{corollary}\label{cor-prox}
Let us use the notations of  Theorem~\ref{thm:mainthmmatrox3},
let  $k_0=0,k_1,\dots,k_q=d$ be the sequence of absciss\ae\ of 
the vertices of the Newton polygon of $\trop p(x)$, 
as shown in Figure~\ref{newtonpolygon}, and 
let $\delta_0,\ldots, \delta_q$ be defined by~\eqref{e-def-delta}.
For $\nk\in\{1,\ldots, nd\}$ and $k=\nkn$,
let us consider the unique $j\in\{1,\ldots,q\}$ such that
$k_{j-1}< k\leq k_j $,
so that $\alpha_{\ent{k}}=\alpha_{k_{j-1}+1}=\alpha_{k_j}$.
Then, the statements of  Theorem~\ref{thm:mainthmmatrox3} hold with the
constants $L_k$, $L_k^\pm$, $L_k^*$ and $L_k^{\pm *}$
 (given in~\eqref{e-def-Lk} and~\eqref{L_k_j_fro3}) replaced 
respectively by the following ones:
\begin{subequations}
\label{L_k_j_prox}
\begin{gather}
L_{k}^{\text{prox}}:= \max(L_{k}^{\pm ,\text{prox}}), \\
L_{k}^{+,\text{prox}}:=\frac{1}{\U(k_j-k,\delta_j)},\quad
L_{k}^{-,\text{prox}}:=\frac{1}{\U(k-k_{j-1},\delta_{j-1})}\enspace ,
\end{gather}
\end{subequations}
and
\begin{subequations}
\label{L_k_j_fro_prox}
\begin{gather}
L_{k}^{*,\text{prox}}:=\max(L_{k}^{\pm *,\text{prox}}), \\
L_{k}^{+*,\text{prox}}:=\frac{1}{\sqrt{\U(k_j-k,\delta_j^2)}},\quad
L_{k}^{-*,\text{prox}}:=\frac{1}{\sqrt{\U(k-k_{j-1},\delta_{j-1}^2)}}\enspace .
\end{gather}
\end{subequations}
Moreover, in the particular case where $k=k_j$, we have 
\begin{align*}
L_{k}^{\text{prox}} \geq L_{k}^{+,\text{prox}}=\frac{1-\sqrt{\delta_{j}}}{1+\sqrt{\delta_{j}}},
\qquad
L_{k}^{*,\text{prox}} \geq L_{k}^{+*,\text{prox}}=\sqrt{\frac{1-\delta_{j}}{1+\delta_{j}}}\enspace.
\end{align*}
\end{corollary}

\begin{remark}\rm
When all the ratios $\delta_1,\dots,\delta_{q-1}$ are small, the maxima in~\eqref{L_k_j} and~\eqref{L_k_j_fro} are attained by taking $j$ as in
Corollary~\ref{cor-prox}, that is $L_{k}^{\text{opt}}=L_{k}^{\text{prox}}$
and $L_{k}^{*,\text{opt}}=L_{k}^{*,\text{prox}}$ for all $0< k\leq d$.
\end{remark}

\begin{remark}\rm
Since $\U(k,\delta)$ is increasing in $k$ and $\delta$,
the maximizing $j$ in the definition of $L_{k}^{+*,\text{opt}}$ 
arises from a compromise between keeping $k_j\geq k$ close to $k$
and $\delta_j$ small. In particular, when $k$ belongs
to an edge of the Newton polygon such that several consecutive edges
have almost the same slope than this edge, the maximizing
$j$ may be the one corresponding to the first vertex
at which the slope changes significantly, i.e., the first
one such that $\delta_j$ is small. Similar considerations
apply to $L_{k}^{-*,\text{opt}}$ 
\end{remark}

\section{Proof of the lower bounds}\label{sec-prooflb}
\subsection{Proof of the lower bounds for the roots of scalar polynomials}
\label{sec:scalar_poly}
In this section, we prove the main 
lower bounds of Section~\ref{sec:Main_results}
for scalar polynomials, since the arguments are more transparent in this case.
The generalization to the matrix case
will be given in the next section.

From Remark~\ref{rk-scalar}, in the scalar case, all lower bounds
reduce to~\eqref{i_edge_bound-scalar} with the constant $L_k$
obtained under Assumption~\ref{assump2},
that is~\eqref{nnz_bound_fro} in Theorem~\ref{thm:mainthmnnz},
\eqref{gen_polya_fro} in Theorem~\ref{thm:mainthmmatrox},
\eqref{L_k_j_fro3} in Theorem~\ref{thm:mainthmmatrox3},
and~\eqref{deltak-k+-fro} in Theorem~\ref{thm:mainthmpoly3}.

\par{\it Proof of the scalar version of Theorem~\ref{thm:mainthmmatrox}}. \ignorespaces
We note first that P\'olya's inequality is an immediate consequence of
Proposition~\ref{genlowercor}. Indeed, using the property that
all the $\beta_{k,j}$ are less than or equal to $1$, we obtain:
\[(L_k)^{-2}\leq \sum_{j=0}^d \xi^{j-k} 
\leq \sum_{j=0}^\infty \xi^{j-k} 
= \frac{1}{\xi^{k}(1-\xi)}\enspace,
\]
for all $\xi>0$.
The minimum of the right hand side of the previous inequality
for $0<\xi<1$, is attained for
$\xi=k/(k+1)$, from which we deduce~\eqref{eq:polyabound},
which is also the scalar version of Theorem~\ref{thm:mainthmmatrox}.

We can now deduce similarly, from Proposition~\ref{genlowercor},
the scalar versions of the lower bounds of Theorems~\ref{thm:mainthmnnz},
\ref{thm:mainthmmatrox3}, and~\ref{thm:mainthmpoly3}.
Since in the scalar case, we are reduced
to show~\eqref{i_edge_bound-scalar} for some constants $L_k$,
and that this inequality is precisely the statement of 
Proposition~\ref{genlowercor},
we only need to show that these constants $L_k$ are lower bounds of
the constant $L_k$ of~\eqref{generalLk}.

\par{\it Proof of the scalar version of Theorem~\ref{thm:mainthmnnz}}. \ignorespaces
Using the property that $\beta_{k,j}\leq 1$ for all
$k,j$, we obtain that the constant $L_k$ of~\eqref{generalLk}
satisfies $(L_k)^{-2}\leq \sum_{j=0, a_j\neq 0}^d \xi^{j-k} $ for all $\xi>0$.
When $\xi=1$, the right hand side of this inequality is equal to
the number of non-zero coefficients of $p$, which shows that 
the constant $L_k$ of~\eqref{generalLk} is
lower bounded by the constant $L_k$ of~\eqref{nnz_bound_fro},
which implies~\eqref{i_edge_bound-scalar} with this lower bound $L_k$.
\endproof

\par{\it Proof of the scalar version of Theorem~\ref{thm:mainthmmatrox3}}. \ignorespaces
Assume that $0\leq k^-<k\leq k^+-1\leq d$ are integers.
Denote $\delta_-=\alpha_{k^-}/\alpha_k$ and $\delta_+=\alpha_k/\alpha_{k^+}$.
Then, for all $1\leq \ell\leq  k$,  $\alpha_{\ell}\leq \alpha_k$,
and for all $1\leq \ell\leq  k^-$, 
$\alpha_{\ell}\leq \alpha_{k^-}= \alpha_k\delta_-$.
This implies that
$\beta_{k,j}\leq \delta_-^{k^--j}$,   for all $j\leq k^-$.
Similarly,
for all $\ell\geq k$, $\alpha_{\ell}\geq \alpha_{k}$,
and for all $\ell\geq k^+$, $\alpha_{\ell}\geq \alpha_{k^+}= \alpha_k/\delta_+$,
hence, for all $j\geq k^+$,
$\beta_{k,j}\leq \delta_+^{j-k^++1}$.
Since we also have $\beta_{k,j}\leq 1$ for all $j,k$,
we obtain that 
the constant $L_k$ of~\eqref{generalLk}
satisfies, for all $\delta_-^2<\xi<\delta_+^{-2}$, $\xi\neq 1$,
\begin{eqnarray*}
 (L_k)^{-2}&\leq &\left(\sum_{j=0}^{k^-} \delta_-^{2(k^--j)} \xi^{j-k}\right) +
\left( \sum_{j=k^-+1}^{k^+-1} \xi^{j-k} \right)
+\left( \sum_{j=k^+}^d \delta_+^{2(j-k^++1)}\xi^{j-k}\right)
\\
&\leq & \xi^{k^--k+1} \frac{1}{\xi-\delta_-^2} 
+ \frac{\xi^{k^--k+1}-\xi^{k^+-k}}{1-\xi}
+\xi^{k^+-k-1}\frac{\delta_+^2\xi}{1-\delta_+^2\xi} \enspace.
\end{eqnarray*}
The later inequality can be written as
\begin{equation}\label{ineq_gen}
 (L_k)^{-2}\leq g(\xi),\quad \forall \delta_-^2<\xi<\delta_+^{-2},\; \xi\neq 1,
\end{equation} with 
\begin{eqnarray*}
g(\xi)&:=& g_-(\xi)+g_+(\xi)\enspace,\\
g_-(\xi)&:=& \xi^{k^--k+1} \left(\frac{1}{\xi-\delta_-^2}
+\frac{1}{1-\xi}\right)\enspace,\\
g_+(\xi)&:=& 
\xi^{k^+-k-1} \left( \frac{1}{\xi-1} +\frac{1}{1-\delta_+^2\xi}\right)\enspace.
\end{eqnarray*}
Note that~\eqref{ineq_gen} also holds when $k^-=0$ or $k^+=d+1$,
since then $\delta_-=0$ or $\delta_+=0$ respectively.
When $\delta_-=1$ and $\delta_+<1$, the conditions on $\xi$
in~\eqref{ineq_gen} are equivalent to $1<\xi< \delta_+^{-2}$, 
whereas when $\delta_-=\delta_+=1$, these conditions are never 
satisfied, but in this case the
constant $L_k$ of~\eqref{L_k_j_fro3} is equal to $0$
so there is nothing to prove.

The functions $g_-$ and $g_+$ satisfy 
$g_-(\xi)= g_{k-k^-,\delta_-^2}(\xi^{-1})$ and
$g_+(\xi)=g_{k^+-k-1,\delta_+^2}(\xi)$
where for $k\geq 0$ and $0\leq \delta\leq 1$,
 $g_{k,\delta}$ is defined as a function of 
$\xi\in (0, \delta^{-1})\setminus \{1\}$ by:
\[ g_{k,\delta}(\xi):=\xi^{k} \left( \frac{1}{\xi-1} +\frac{1}{1-\delta\xi}\right)= \xi^{k+1} \frac{1-\delta}{(\xi-1)(1-\delta\xi)} \enspace.\]
We have $g_{k,\delta}(\xi)\leq 0$ for $0<\xi<1$, hence 
$g_-(\xi)\leq 0$ for all
$\xi> 1$, and $g_+(\xi)\leq 0$ for all $\xi<1$.
When $0\leq \delta<1$, the minimum of $g_{k,\delta}$ on $(1,\delta^{-1})$ is attained at 
\[ \xi_{k,\delta}:= 
\begin{cases}
\frac{k(\delta+1)-\sqrt{k^2(\delta+1)^2-4\delta(k^2-1)}}{2 (k-1) \delta}
&\text{when}\; k\neq 1,\; \delta\neq 0\enspace,\\
\frac{k+1}{k(1+\delta)}&\text{otherwisewhen}\; \delta=0\enspace.
\end{cases}\]
The last formula gives $\xi_{k,\delta}=\infty$ when $\delta=0$ and
$k=0$, which is the point of infimum of $g_{k,\delta}$ for
 $1<\xi<\infty=\delta^{-1}$, since  $g_{k,\delta}$ is decreasing.
It also gives $\xi_{k,1}=1$ for all $k\geq 0$. Hence extending $g_{k,\delta}$,
$g_-$ and $g_+$ by $+\infty$ at point $1$, we get that the
infimum of $g_{k,\delta}$  on $(1,\delta^{-1})$ equals $g_{k,\delta}(\xi_{k,\delta})$,
and denoting $\xi_-=(\xi_{k-k^-,\delta_-^2})^{-1}$ and $\xi_+=\xi_{k^+-k-1,\delta_+^2}$,
we obtain 
\begin{equation}\label{firstlk}
 (L_k)^{-2}\leq \min( g(\xi_-),g(\xi_+))\leq \min( g_-(\xi_-),g_+(\xi_+))
\enspace.\end{equation}
In order to simplify this bound, we need to find good estimates of 
$\xi_{k,\delta}$.

When $k\neq 1$ and $\delta\neq 0$, we have:
\[ %
\xi_{k,\delta} =
\frac{k(\delta+1)-\sqrt{k^2(1-\delta)^2+4\delta}}{2 (k-1) \delta}%
= \frac{2(k+1)}{k(\delta+1)+\sqrt{k^2(1-\delta)^2+4\delta}}\enspace.
\]%
Using the property that $\delta\geq 0$ in the last formula, we get that
$\xi_{k,\delta} \leq \frac{k+1}{k}=\xi_{k,0}$.
Moreover, this inequality also holds for $k=1$.
In particular $(\xi_{k,\delta})^k\leq \E(k)$, for all $k\geq 0$ and $\delta\geq 0$ (taking the convention that $\xi^0=1$ for all $\xi\in (1,\infty]$).
We also have, for $k\geq 0$ and $0\leq \delta<1$,
\[ \frac{1}{\xi_{k,\delta}-1} +\frac{1}{1-\delta\xi_{k,\delta}}
=\frac{1+\delta+\sqrt{k^2(1-\delta)^2+4\delta}}{1-\delta}
\leq k+\frac{1+\sqrt{\delta}}{1-\sqrt{\delta}}\enspace.\]
This yields
\[ g_{k,\delta}(\xi_{k,\delta})\leq \U(k,\delta)
\quad\forall k\geq 0,\;0\leq \delta\leq 1,\]
with $\U$  as in~\eqref{defU}.
Indeed, this inequality holds for $\delta<1$ by the above arguments.
It also holds for $\delta=1$ since $\U(k,1)=+\infty$.
This implies in particular that 
$ g_+(\xi_+) \leq  \U(k^+-k-1,\delta_+^2) $ and 
$g_-(\xi_-)  \leq \U(k-k^-, \delta_-^2)$.
Combining these two inequalities with~\eqref{firstlk}, we 
obtain that the constant $L_k$ of~\eqref{generalLk} is
lower bounded by the constant $L_k^*$ of~\eqref{L_k_j_fro3},
which implies~\eqref{i_edge_bound-scalar} with this lower bound $L_k^*$
instead of $L_k$.
\endproof

\begin{remark}\rm 
In the previous proof, 
obtaining the minimum of $g$ instead of $g_-$ and $g_+$
would have led to a better 
 lower bound for $L_k$ than the one  of Theorem~\ref{thm:mainthmmatrox3}.
However, such a bound is more difficult to estimate
for general values of $k^+-k^-$.
We can use for instance the first inequality in~\eqref{firstlk},
which will give better estimates in some particular cases.
For instance, 
when $k=0$, we get that $g_{k,\delta}(\xi_{k,\delta})= \U(k,\delta)=
(1+\sqrt{\delta})/(1-\sqrt{\delta})$.
In particular, when $k=k^+-1$, the bound
$ (L_k)^{-2}\leq g_+(\xi_+)$ gives the lower bound $L_k^+$ 
of~\eqref{L_k_j_fro31}, with $\delta=\delta_+$.
However~\eqref{firstlk} gives the slightly better bound:
\[ (L_k)^{-2}\leq g(\xi_+)=\frac{1+\delta_+}{1-\delta_+}
-\frac{1-\delta_-^2}{(1-\delta_+\delta_-^2)(1-\delta_+)}
\delta_+^{k^+-k^-}\enspace.\]
\end{remark}

\par{\it Proof of the scalar version of Theorem~\ref{thm:mainthmpoly3}}. \ignorespaces
Let us use the notation of the previous proof.
When $k=k^+-1=k^-+1$, Inequality~\eqref{ineq_gen} is true for all
$\xi\in (\delta_-^2,\delta_+^{-2})$ since $g$ can be well defined 
at $1$ (by continuity for instance).
When $\delta_-\delta_+<1$, the interval $(\delta_-^2,\delta_+^{-2})$
is nonempty, the map $g$ is convex there 
and its minimum is achieved at the 
point $\xi=(1+\delta_{-}^2\delta_+)/(\delta_+(1+\delta_+))$.
Then~\eqref{ineq_gen} yields
$(L_k)^{-2}\leq  (1+\delta_+)^2/(1-\delta_{-}^2\delta_+^2)$,
which implies that the constant $L_k$ of~\eqref{generalLk} is
lower bounded by the constant $L_k^{*\sharp}$ of~\eqref{deltak-k+-fro},
which implies~\eqref{i_edge_bound-scalar} with this lower bound $L_k^{*\sharp}$
instead of $L_k$.
Moreover, since the minimum of $g$ is less or equal to the right hand side 
of~\eqref{firstlk}, the proof of the scalar version of
 Theorem~\ref{thm:mainthmmatrox3} implies that 
$L_k^{*\sharp}\leq L_k^{*}$.
\endproof

\subsection{Proofs of the lower bounds for the eigenvalues
of matrix polynomials}
\label{sec:matrix_poly}
In this section, we prove the lower bounds which were stated in Section~\ref{sec:Main_results} for matrix polynomials, and proved in Section~\ref{sec:scalar_poly} in the special scalar case. 

\par{\it Proof of Theorems~\ref{thm:mainthmnnz}, \ref{thm:mainthmmatrox},
\ref{thm:mainthmmatrox3}, and~\ref{thm:mainthmpoly3}}. \ignorespaces
In Proposition~\ref{genlowercormatrix},
the formula~\eqref{generalLkmatrixfrob} of  $L_k$
coincides with~\eqref{generalLk},
except that the condition $a_j\neq 0$ is replaced by $A_j\neq 0$
and $\beta_{k,j}$ is replaced by $\beta_{\ent{k},j}$.
For $0<k\leq d$, and for integers $k^-$ and
$k^+$, the inequality $0\leq k^-<k\leq k^+-1\leq d$ 
is equivalent to $0\leq k^-<\ent{k}\leq k^+-1\leq d$.
Then, for all $0<k\leq d$, the inequalities obtained for $\beta_{\ent{k},j}$
in the proofs of  Theorems~\ref{thm:mainthmnnz}, 
\ref{thm:mainthmmatrox3}, and~\ref{thm:mainthmpoly3} remain true
and depend only on  $k^-$ and $k^+$ and not on $\ent{k}$.
Hence, the lower bounds of the constant $L_k$ of~\eqref{generalLk}
obtained in these proofs remain true for the constant  $L_k$ 
of~\eqref{generalLkmatrixfrob}. Combined with the assertion that
Inequality~\eqref{i_edge_bound} holds with this constant,
they now provide the lower bounds of Theorems~\ref{thm:mainthmnnz},
\ref{thm:mainthmmatrox3}, and~\ref{thm:mainthmpoly3} respectively,
in the case where $\|\cdot\|$ satisfies~\ref{assump2},
that is Inequality~\eqref{i_edge_bound} with $L_k$ replaced by 
the constant $L_k$ of~\eqref{nnz_bound_fro},
the constant $L_k^*$ of~\eqref{L_k_j_fro3}, and the constant $L_k^{*\sharp}$ 
of~\eqref{deltak-k+-fro}, respectively.

In the case of a norm satisfying only~\ref{assump1},
the formula~\eqref{generalLkmatrixfrob} of  $L_k$ is replaced 
by~\eqref{generalLkmatrix}, which means that everything behave as if 
$L_k$ were replaced by its square and the numbers $\alpha_k$ were replaced by 
their square roots in~\eqref{generalLkmatrixfrob}.
Hence the assertions of Theorems~\ref{thm:mainthmnnz}, 
\ref{thm:mainthmmatrox3}, and~\ref{thm:mainthmpoly3} 
in the case of Assumption~\ref{assump2}, are still true 
under Assumption~\ref{assump1}, up to this transformation of
the constant $L_k$ of~\eqref{nnz_bound_fro},
the constant $L_k^*$ of~\eqref{L_k_j_fro3}, and the constant $L_k^{*\sharp}$ 
of~\eqref{deltak-k+-fro}, respectively. 

The constant $L_k$ of~\eqref{nnz_bound_fro}
does not depend on the values of the constants $\alpha_k$,
hence the above transformation leads to the square of $L_k$, which is
exactly the constant $L_k$ of~\eqref{nnz_bound}, which finishes the proof of
Theorem~\ref{thm:mainthmnnz}.
The constant  $L_k^*$ of~\eqref{L_k_j_fro3}
is a function of $\delta_-=\frac{\alpha_{k^-}}{\alpha_{\ent{k}}}$
and $\delta_+=\frac{\alpha_{\ent{k}}}{\alpha_{k^+}}$.
Since replacing all numbers $\alpha_i$, $i=0,\ldots, d+1$,
by their square roots, reduces to replace the
numbers $\delta_-$ and $\delta_+$ by their square roots too,
the above transformation of the constant  $L_k^*$ of~\eqref{L_k_j_fro3}
consists in taking its square and replacing 
$\delta_-$ and $\delta_+$ by their square roots, which leads to the
constant  $L_k$ of~\eqref{e-def-Lk}, which finishes the the proof of
Theorem~\ref{thm:mainthmmatrox3}.
Similarly, the above transformation of the constant $L_k^{*\sharp}$ 
of~\eqref{deltak-k+-fro} leads to the constant $L_k^{\sharp}$ 
of~\eqref{deltak-k+}, which finishes the proof of
Theorem~\ref{thm:mainthmpoly3}.

Now Theorem~\ref{thm:mainthmmatrox} is an
immediate corollary of Theorem~\ref{thm:mainthmmatrox3}.
Indeed as said in Section~\ref{sec:Main_results}, taking $k^-=0$, $k^+=d+1$
in Theorem~\ref{thm:mainthmmatrox3}, and using $L_k^-\leq L_k$, 
$L_k^{-*}\leq L_k^*$
and $\U(k,0)=\E(k) (k+1)$, we get Theorem~\ref{thm:mainthmmatrox}.
\endproof

Corollaries~\ref{cor-strongest}
and \ref{cor-prox} are straightforward consequences
of Theorem~\ref{thm:mainthmmatrox3} and Proposition~\ref{prop-optim}.
We only need to prove the latter proposition.

\par{\it Proof of Proposition~\ref{prop-optim}}. \ignorespaces
Let us use the notations of  Theorem~\ref{thm:mainthmmatrox3}, and
let $k_0=0,k_1,\dots,k_q=d$ be the sequence of absciss\ae\ of
the vertices of the Newton polygon of $\trop p(x)$, 
as shown in Figure~\ref{newtonpolygon}.
Then, the ratio $\delta_-$ does not change when $k^-$
moves inside an interval $[k_{r-1}+1,k_{r}]$.
Since $\U(\cdot,\delta)$ is a nondecreasing function,
in order to maximize $L_k^-$ or $L_k^{-*}$ and thus $L_k$,
with $\delta_-$ constant (and $\delta_+$ and $k^+$ constant),
we need to minimize $k-k^-$, which implies that
$k^-=k_r$ for some $0\leq r\leq q$ such that $k_r<k$.
Similarly, $k^+=k_s+1$ for  some $0\leq s\leq q$ such that $k_s\geq k$.
\endproof

\begin{remark}\rm\label{comparenorms}
All the lower bounds in Theorems~\ref{thm:mainthmnnz},~\ref{thm:mainthmmatrox},
\ref{thm:mainthmmatrox3}, and~\ref{thm:mainthmpoly3} for
$\nk=nk$ are equivalent to inequalities of the form
\begin{equation}\label{generallower}
(|\zeta_1\dots \zeta_{nk}|^{-1} |\det A_0|)^{1/n} L_k
\leq  \|A_0\|\alpha_1^{-1}\dots\alpha_k^{-1}\enspace.
\end{equation} 
By definition of the tropical roots, $\|A_0\|\alpha_1^{-1}\dots\alpha_k^{-1}=\exp \hat{p}_k$, where $\hat{p}_k$ is the value in $k$ of the concave hull 
of the map $j\in\{0,\ldots, d\}\mapsto \log \|A_j\|$.
Hence, if $\|\cdot\|$ and $\|\cdot\|'$ are two norms on the space of matrices
such that $\|A\|\leq \|A\|'$ for all $A\in\C^{n\times n}$, then 
the right hand side in~\eqref{generallower} is smaller for $\|\cdot\|$
than for $\|\cdot\|'$. 
Moreover, in~\eqref{nnz_bound},~\eqref{nnz_bound_fro},~\eqref{gen_polya},
and~\eqref{gen_polya_fro}, $L_k$ depends only on $P$ or $k$ and $n$, but not
on the norm $\|\cdot\|$, so that~\eqref{generallower}
is necessarily a tighter inequality for  $\|\cdot\|$ than for $\|\cdot\|'$. 
In particular if~\ref{assump2} holds with $Q$ and $Q'$, then
the lower bounds~\eqref{nnz_bound_fro} and~\eqref{gen_polya_fro}
for $\|\cdot\|$ are weaker 
than the corresponding ones for $A\mapsto \|Q A Q'\|_\nf$.
 For the 
lower bounds~\eqref{e-def-Lk}, \eqref{L_k_j_fro3}, \eqref{deltak-k+},
and~\eqref{deltak-k+-fro} 
of Theorems~\ref{thm:mainthmmatrox3} and~\ref{thm:mainthmpoly3},
and the ones of Corollaries~\ref{cor-strongest} and \ref{cor-prox},
 based on the separation between 
the tropical roots, the comparison is not so simple, because
changing the norm changes the separation between the tropical
roots.
\end{remark}
\section{Upper bound}\label{sec-newupperbound}
\subsection{Statement of the upper bound and corollary}\label{sec-upperbound}
The previous results provide lower bounds
for the product of the $\nk$ smallest eigenvalues,
in terms of tropical roots.
We next state a reverse inequality.
\begin{theorem}[Upper bound]
\label{thm:newthem}
Let $\|\cdot\|$ be any norm on the space of matrices $\C^{n\times n}$.
For all $i=1,\dots, n$, we denote 
by $A^{(i)}$ the $i$-th column of $A$,  
by $\eta_i$ the least positive constant
such that $\|A^{(i)}\|_2\leq \eta_i\|A\|$ for all $A\in \C^{n\times n}$
where $\|\cdot\|_2$ is the Euclidean norm of $\C^n$,
and we set $\eta:=\eta_1\cdots\eta_n$.
Let $A_0$,\ldots, $A_d$, $P$, $\zeta_1,\dots,\zeta_{nd}$, 
$c$, $\trop p$, 
$\alpha_1,\dots,\alpha_{d}$, be 
as in Theorem~\ref{thm:mainthmmatrox}, 
let $k_0=0,k_1,\dots,k_q=d$ be the sequence of all
absciss\ae\ of the vertices of the Newton polygon of $\trop p(x)$,
and define $\delta_0,\ldots,\delta_q$ as in~\eqref{e-def-delta}.
Denote by $C_{n,d,k}$ the number of maps $\phi:\{1,\ldots, n\}\to \{0,\ldots, d\}$ such that $\sum_{j=1}^n \phi(j)=k$, so that $C_{n,d,k}\leq (\min(d,k,nd-k)+1)^{n-1}$.

Then, for every $j=1,\ldots , q$, if
\begin{align}
c_j:=\frac{|\det A_{k_j}|}{\|A_{k_j}\|^n}-(C_{n,d,nk_j}-1)\eta \delta_{j}>0\enspace,
\label{assum-nonsig}
\end{align}
we have
\[
\frac{|\zeta_1\dots\zeta_{n{k_j}}|}{(\alpha_1\dots\alpha_{k_j})^n}
\leq
 \frac{c}{c_j}\binomial{nd}{n{k_j}}\enspace.
\]
\end{theorem}
Using the results of Section~\ref{sec:Main_results}
and Theorem~\ref{thm:newthem}, we are able to bound  from above and below 
the ratio between the modulus of an eigenvalue
and its corresponding tropical root, as follows.
\begin{corollary}\label{cor-bounds}
Let us use the notations of Theorems~\ref{thm:mainthmnnz}
and~\ref{thm:newthem}, and
assume that~\eqref{i_edge_bound} holds for all $\nk\in\{1,\ldots, nd\}$.
Then, under the assumptions of Theorem~\ref{thm:newthem},
we have, for all $1\leq j \leq q$,
such that $c_j>0$ and $c_{j-1}>0$, and all integer $\nk$ such that
$k_{j-1}<\nkn\leq k_{j}$
\begin{align}
\frac{c_{j-1}}{\binomial{nd}{n{k_{j-1}}}} (L_{k_{j-1}+\frac{1}{n}})^n
\leq \frac{|\zeta_{nk_{j-1}+1}|}{\alpha_{k_j}}
\leq \frac{|\zeta_{\nk}|}{\alpha_{k_j}}
\leq \frac{|\zeta_{nk_j}|}{\alpha_{k_j}}
\leq \frac{\binomial{nd}{n{k_{j}}}}{c_{j}} (L_{k_{j}-\frac{1}{n}})^{-n} \enspace.
\label{e-location}
\end{align}
In particular, if $n\geq 2$ and~\ref{assump2} holds, we have:
\begin{align}
\frac{c_{j-1}}{\binomial{nd}{n{k_{j-1}}}
\U(\frac{1}{n},\delta_{j-1}^2)^{n/2}}
\leq \frac{|\zeta_{\nk}|}{\alpha_{k_j}}
\leq \frac{\binomial{nd}{n{k_{j}}}\U(\frac{1}{n},\delta_j^2)^{n/2}}{c_{j}} 
\enspace.
\end{align}
\end{corollary}
Note that the bounds in the above corollary
do not depend on the constant $c$,
except if $j=1$ (in which case $c_{j-1}=c_0=c$).
Hence, by an argument of continuity, the assumption that $\det A_0\neq 0$,
which is present there (since the notations and assumptions of 
Theorem~\ref{thm:newthem} include the ones of the first 
part of Theorem~\ref{thm:mainthmnnz}), can be dispensed with, except when $j=1$.

\begin{remark}\rm \label{rem-tight-cond}
Since $1\geq \frac{|\det A_{k_j}|}{\|A_{k_j}\|^n}\geq (\kappa(A_{k_j}))^{-n}$,
where $\kappa$ denotes the condition number of a matrix 
with respect to the norm $\|\cdot\|$,
we get that Condition~\eqref{assum-nonsig} of Theorem~\ref{thm:newthem} holds
if the matrix $A_{k_j}$ is nonsingular and has a sufficiently small
condition number, and if the tropical
roots $\alpha_{k_j}$ and $\alpha_{k_{j+1}}$ are sufficiently
separated, so that $\delta_j = \alpha_{k_j}/\alpha_{k_{j+1}}\ll 1$.
Similarly, 
the tightness of the bounds in Corollary~\ref{cor-bounds} depends
on the parameters 
$\delta_{{j-1}},\delta_{j}$ and on the condition number of the matrices
$A_{k_{j-1}},A_{k_j}$. 
\end{remark}

\begin{remark}\label{rem-pellet}\rm
The bounds of Corollary~\ref{cor-bounds} require the $\delta_j$ to be 
very small. In this case, a different approach would be to use
the extensions of Pellet's theorem to the case of matrix 
polynomials~\cite{binisharify,melman,SharifyTisseur}.

Recall that in the scalar case, Pellet's theorem shows that for $k\not\in\{0,d\}$, if 
the polynomial $q(x)=|a_k|x^k - \sum_{i\neq k,\; 0\leq i\leq d}|a_i|x^i$ 
 has precisely two real positive roots $r(a)<R(a)$
(we know by Descartes' rule of sign it has $0$ or $2$ real positive roots),
then $p$ has precisely $k$ roots in the region $|x|\leq r(a)$ and $d-k$ roots 
in the region $|x|\geq R(a)$. 
Moreover, in~\cite{bini96}, it is shown that in  that case,
$k$ is the abscissa of a vertex of the Newton polygon of $p$,
and that tropical roots of $p$ are outside the interval $(\log r(a),\log R(a))$.
If Pellet's theorem can be applied to every abscissa $k$ of a vertex
of the Newton polygon of $p$, then each moduli of a root of $p$ can
be estimated using the corresponding tropical root.
This is in particular possible if $\delta_j<1/9$ for all
$j=1,\ldots, q$, by a result of Gaubert and Sharify (see~\cite[Theorem~3.3.3]{meisam2011}).

Using an extension of Pellet's theorem to the case of matrix polynomials, one may obtain a similar localisation of the eigenvalues when all the $\delta_j$ 
and condition numbers of the matrices $A_{k_j}$ are small enough,
 see for instance~\cite{binisharify} for the case where the coefficient matrices are unitary up to some constants,
and~\cite{SharifyTisseur} for the general case.

However,  it may happen that Pellet's theorem cannot be applied for any
vertex of the Newton polygon except the first and last one.
In this case, the log-majorization type inequalities
of Section~\ref{sec:Main_results} do provide 
a better information. Note that there is some duality
between Pellet type estimates and log-majorization
estimates: the former hold conditionnally, and give exclusion
annuli, whereas the latter hold inconditionnally, and allow
one to infer inclusion regions for the eigenvalues. 
For instance, it does not seem that tight constants
like the one in~\eqref{eq:polyabound} are achievable by Pellet type approaches.
\end{remark}

\subsection{Proof of the upper bound}\label{sec-proof-upperbound}
To prove Theorem~\ref{thm:newthem}, we first prove the following lemma which 
provides a lower bound for the modulus of the coefficients of 
the polynomial $\det(P(\lambda))$.
\begin{lemma}
\label{lem:uppbound2}
Let $A_0$,\ldots, $A_d$, $P$, $\|\cdot\|$, $\eta$,
$\alpha_1,\dots,\alpha_d$,
$q$, $\delta_0,\ldots,\delta_q$, and
 $C_{n,d,k}$ be as in Theorem~\ref{thm:newthem}, and denote by 
$\detp$ the polynomial: 
\begin{equation}
\label{detpolynom}
\detp(\lambda)=\det P(\lambda)=\sum_{l=0}^{nd}\detp_l\lambda^l \enspace.
\end{equation}
Then, for $j=0,\ldots, q$, we have $\|A_{k_j}\|>0$ and 
\[
|\detp_{nk_j}|\geq |\det A_{k_j}|-(C_{n,d,nk_j}-1)\eta \|A_{k_j}\|^n\delta_{j}\enspace.
\]
\end{lemma}
\begin{proof}
Let $k=0,\ldots , d$. Using the multilinearity of the determinant we get
\[
\detp(\lambda)=\det \sum_{j=0}^dA_j\lambda^j
=\sum_{\phi}\det(A_{\phi(1)}^{(1)},A_{\phi(2)}^{(2)},\dots,A_{\phi(n)}^{(n)})\lambda^{\sum_{m=1}^n\phi(m)}\enspace,
\]
where the sum is taken over all maps $\phi:\{1,\dots,n\}\to\{0,\dots,d\}$.
Denoting by $\Phi_k$ the set of all such maps that satisfy
$\sum_{m=1}^n\phi(m)=nk$, we obtain that the $nk$-th
 coefficient of the polynomial $\detp$ is equal to:
\begin{eqnarray*}
\detp_{nk}
&=&
\sum_{\phi\in \Phi_k}\det(A_{\phi(1)}^{(1)},A_{\phi(2)}^{(2)},\dots,A_{\phi(n)}^{(n)})\\
&=& \det A_k +\sum_{\phi\in\Phi_k,\; \phi\not\equiv k}\det(A_{\phi(1)}^{(1)},A_{\phi(2)}^{(2)},\dots,A_{\phi(n)}^{(n)})\enspace.
\label{eq:uppbrprv}
\end{eqnarray*}
Using Hadamard's inequality together with the definition of
$\eta$ yields
\begin{eqnarray*}
|\det (A_{\phi(1)}^{(1)},A_{\phi(2)}^{(2)},\dots,A_{\phi(n)}^{(n)})|
&\leq& \|A_{\phi(1)}^{(1)}\|_2\dots\|A_{\phi(n)}^{(n)}\|_2\leq \eta \|A_{\phi(1)}\|\dots\|A_{\phi(n)}\|\enspace.
\end{eqnarray*}
Assume now that $k=k_j$ for some $j=0,\ldots, q$. Then, $k$ is the 
abscissa of a vertex of the Newton polygon of the tropical polynomial $\trop p$
defined  in~\eqref{eq:tropdefini}.
By Proposition~\ref{prop:coefbnd}
applied to $\trop p$, we get that $\|A_k\|>0$,
$\|A_m\|\leq \|A_k\|\prod_{\ell=m+1}^k\alpha_{\ell}\leq
\|A_k\| \alpha_k^{k-m}$ for all $m\leq k$
and $\|A_m\|\leq \|A_k\|\prod_{\ell=k+1}^m\alpha_{\ell}^{-1}\leq\|A_k\|
\alpha_{k+1}^{k-m}\leq  \|A_k\| \alpha_k^{k-m}\delta_j^{m-k}$ for all $ m> k$.
Hence
\begin{eqnarray*}
\|A_{\phi(1)}\|\dots\|A_{\phi(n)}\| &
\leq &\|A_{k}\|^n\alpha_k^{\sum_{m=1}^n(k-\phi(m))}
\delta_j^{\sum_{m=1,\phi(m)>k}^n(\phi(m)-k)}\\
&=&\|A_{k}\|^n\delta_j^{\sum_{m=1,\phi(m)>k}^n(\phi(m)-k)}\enspace,\end{eqnarray*}
 when $\sum_{m=1}^n\phi(m)=nk$.
When in addition $\phi\not\equiv k$, there exists $m=1,\ldots, n$ such that
$\phi(m)> k$, thus $\sum_{m=1,\phi(m)>k}^n(\phi(m)-k)\geq 1$, which yields
\[
\|A_{\phi(1)}\|\dots\|A_{\phi(n)}\|\leq \|A_{k}\|^n\delta_j\enspace.
\]
From all the above inequalities, we deduce that
\[
|\detp_{nk}|\geq |\det A_k|-\sum_{\phi\in\Phi_k,\; \phi\not\equiv k}\eta \|A_k\|^n\delta_j
\enspace.
\]
Since  by definition $C_{n,d,nk}$ is the cardinality of $\Phi_k$,
and there exists exactly one element of $\Phi_k$ such that
$\phi\equiv k$, we obtain the inequality of the lemma for $k=k_j$.
\end{proof}

\par{\it Proof of Theorem~\ref{thm:newthem}}. \ignorespaces
Consider the polynomial $\detp$ of~\eqref{detpolynom}
and let $\gamma_1,\dots,\gamma_{nd}$ be the tropical roots of the tropical 
polynomial $\trop \detp(x)=\max_{0\leq l\leq nd}|\detp_l|x^l$
arranged in non-decreasing order. Let $j=0,\ldots, q$ and  denote 
$k=k_j$. 
From~\eqref{eq:ostrowski}, we have
\begin{equation}\label{ostrovn}
|\zeta_1\dots\zeta_{n{k}}|\leq \binomial{nd}{n{k}} \gamma_1\dots\gamma_{n{k}}\enspace,
\end{equation}
By the first part of Proposition~\ref{prop:coefbnd} applied to the 
tropical polynomial $\trop \detp$, we get 
$|\detp_{nk}| \gamma_1\dots\gamma_{n{k}} \leq |\detp_0|$.
Applying the result of Lemma~\ref{lem:uppbound2},
and using the assumption of Theorem~\ref{thm:newthem} on $c_j$,
we get that $\detp_{nk}\geq c_j \|A_{k}\|^n>0$.
Since $\detp_0=|\det A_0|$, we deduce that
$\gamma_1\dots\gamma_{n{k}}\leq |\det A_0|/(c_j \|A_{k}\|^n)$.
Gathering this with~\eqref{ostrovn}, we get
\[
|\zeta_1\dots\zeta_{n{k}}|\leq \frac{\binomial{nd}{n{k}}}{c_j}
 \frac{|\det A_0|}{\|A_{0}\|^n}\frac{\|A_{0}\|^n}{\|A_{k}\|^n} \enspace.
\]
Since $k=k_j$ is the abscissa of a vertex of the Newton polygon of 
the tropical polynomial $\trop p$, we get from
Proposition~\ref{prop:coefbnd} applied to $\trop p$ that
$\frac{\|A_0\|}{\|A_{k}\|}=\alpha_1\dots\alpha_{k}$,
hence the previous inequality shows Theorem~\ref{thm:newthem}.
\endproof

\section{Numerical examples illustrating the lower bounds}
\label{sec:examples}
In this section, we illustrate the bounds by numerical examples.
We shall see in particular
that the lower bounds
stated in Section~\ref{sec:Main_results} can be tight when the tropical
roots are well separated and when the input matrices are well conditioned,
according to Remark~\ref{rem-tight-cond}.
All the numerical results of the present paper
were obtained using Scilab 4.1.2.
Note that similar 
results are obtained using  Matlab, specially the polyeig function.

\subsection{Illustration of the lower bounds for the roots of scalar polynomials}
\

In the following examples, for each scalar polynomial,
using the notations of Section~\ref{sec:Main_results},
we shall show in tables or plots
the following quantities with respect
to $k=1,\ldots, d$ for comparison:
- the ratio $\frac{|\zeta_1\dots \zeta_{k}|}{\alpha_1\dots\alpha_k}$,
denoted \col{ratio};
- the lower bound of P\'olya given in~\eqref{eq:polyabound},
which coincides with the constant $L_k$ of~\eqref{gen_polya_fro},
denoted \col{p\'olya};
- the lower bound $L_k$ of~\eqref{nnz_bound_fro} in Theorem~\ref{thm:mainthmnnz}, denoted \col{fewnom};
- the lower bound 
$L_k^{*,\text{prox}}$ of~\eqref{L_k_j_fro_prox} in Corollary~\ref{cor-prox},
based on the separation between the tropical roots,
denoted \col{separ}.
Also, in tables, the absciss\ae\ of the vertices of the Newton polygon of 
$\trop p$ are indicated by the symbol~$^*$.

\begin{example} \rm 
Consider the following scalar polynomial
\begin{equation}
\label{eq:polyexp}
p(x)=1-\exp(1)x-\exp(6)x^2+\exp(4)x^4+\exp(9)x^8+\exp(5)x^{10}+\exp(1)x^{16}\enspace.
\end{equation}
The log-exp transformation of its tropical polynomial $\trop p$ is the
tropical polynomial $f$ of Example~\ref{ex1}. The graph
of $f$ and the associated Newton polygon were
shown in Figure~\ref{fig:newtongraph}.
Hence, the tropical roots of $\trop p$ are the exponentials
of $-3$, $-1/2$, and $1$, with multiplicities $2$, $6$ and $8$, respectively.
Table~\ref{tbl:scalarpoly} shows the ratios and lower bounds,
as explained at the beginning of the section.
\begin{table}[htbp]
{\scriptsize
\begin{center}
\begin{tabular}{|l|l|c|l|l|}
\hline
$k$ & \col{ratio} & \col{fewnom} & \col{p\'olya} & \col{separ} \\
\hline
1   &  0.93475  &  0.37796  &  0.5      &  0.5      \\  
2$^*$   &  1.00034  &  \myquote  &  0.3849   &  0.92102  \\  
3   &  0.98782  &  \myquote  &  0.32476  &  0.49664  \\  
4   &  0.97926  &  \myquote  &  0.28622  &  0.38360  \\  
5   &  0.98348  &  \myquote  &  0.25880  &  0.32403  \\  
6   &  0.98771  &  \myquote  &  0.23802  &  0.37508  \\  
7   &  0.99383  &  \myquote  &  0.22156  &  0.47570  \\  
8$^*$   &  1.00000  &  \myquote  &  0.20810  &  0.79696  \\  
9   &  0.98811  &  \myquote  &  0.19683  &  0.47570  \\  
10  &  0.97636  &  \myquote  &  0.18721  &  0.37508  \\  
11  &  0.96476  &  \myquote  &  0.17888  &  0.31917  \\  
12  &  0.95329  &  \myquote  &  0.17158  &  0.28622  \\  
13  &  0.96476  &  \myquote  &  0.16510  &  0.32476  \\  
14  &  0.97636  &  \myquote  &  0.15930  &  0.3849   \\  
15  &  0.98811  &  \myquote  &  0.15407  &  0.5      \\  
16$^*$  &  1        &  \myquote  &  0.14933  &  1        \\  
\hline
\end{tabular}
\end{center}
}
\caption{Comparison of the lower bounds for the scalar polynomial $p$ of~\eqref{eq:polyexp}. }
\label{tbl:scalarpoly}
\end{table}

\end{example}

\begin{example} \rm 
Consider now the following scalar polynomial
\begin{align}
p(x)=&1-\exp(3)x-\exp(6)x^2-\exp(7)x^4-\exp(8)x^6+\exp(9)x^8\nonumber \\
&+\exp(7)x^{10}-\exp(4)x^{13}-\exp(3)x^{14}+\exp(1)x^{16}\enspace.\label{eq:polyexp2}
\end{align}
The log-exp transformation of its tropical polynomial $\trop p$ is again the
tropical polynomial $f$ of Example~\ref{ex1}, but now there are some points of
the graph of the map $k\mapsto \log |p_k|$, where 
$p_k$ is the $k$th coefficient of $p$, that are on the edges of 
its Newton polygon, that is its concave hull.
Table~\ref{tbl:scalarpoly2} shows the ratios and lower bounds for $p$.
\begin{table}[htbp]
{\scriptsize
\begin{center}
\begin{tabular}{|l|l|c|l|l|}
\hline
$k$ & \col{ratio} & \col{fewnom} & \col{p\'olya} & \col{separ} \\
\hline
1   &  0.61759  &  0.31623  &  0.5      &  0.5      \\  
2$^*$   &  0.98684  &  \myquote  &  0.3849   &  0.92102  \\  
3   &  0.84948  &  \myquote  &  0.32476  &  0.49664  \\  
4   &  0.73124  &  \myquote  &  0.28622  &  0.38360  \\  
5   &  0.63280  &  \myquote  &  0.25880  &  0.32403  \\  
6   &  0.54760  &  \myquote  &  0.23802  &  0.37508  \\  
7   &  0.72118  &  \myquote  &  0.22156  &  0.47570  \\  
8$^*$   &  0.95684  &  \myquote  &  0.20810  &  0.79696  \\  
9   &  0.76959  &  \myquote  &  0.19683  &  0.47570  \\  
10  &  0.61899  &  \myquote  &  0.18721  &  0.37508  \\  
11  &  0.57724  &  \myquote  &  0.17888  &  0.31917  \\  
12  &  0.53832  &  \myquote  &  0.17158  &  0.28622  \\  
13  &  0.62117  &  \myquote  &  0.16510  &  0.32476  \\  
14  &  0.71677  &  \myquote  &  0.15930  &  0.3849   \\  
15  &  0.84662  &  \myquote  &  0.15407  &  0.5      \\  
16$^*$  &  1        &  \myquote  &  0.14933  &  1        \\  
\hline
\end{tabular}
\end{center}
}
\caption{Comparison of the lower bounds for the scalar polynomial $p$ of~\eqref{eq:polyexp2}.}%
\label{tbl:scalarpoly2}
\end{table}
In Figure~\ref{fig_exemplescalar}, we plot in the same graph 
the values of the ratios and lower bounds for the 
 scalar polynomials $p$ of~\eqref{eq:polyexp} and~\eqref{eq:polyexp2}.
This indicates that the lower bound $L_k$ of Corollary~\ref{cor-prox},
based on the separation between tropical roots,
may become tighter when the polynomial has non zero coefficients 
between the vertices of the Newton polygon.
\begin{figure}[htbp]
\centering
\includegraphics[scale=0.4]{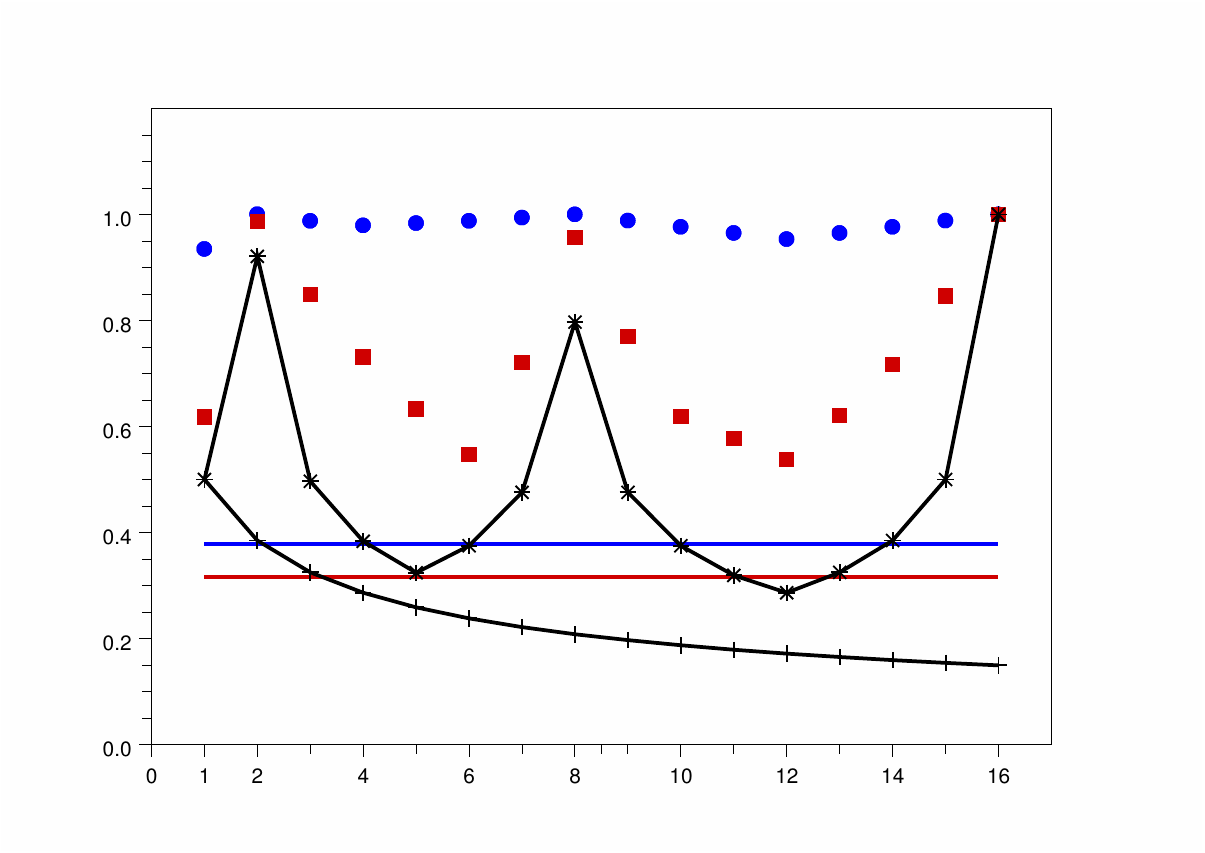}
\caption{Plot of \col{ratio} as a function of $k$ for the scalar polynomials of~\eqref{eq:polyexp}
(blue circles) and~\eqref{eq:polyexp2} (red squares),
together with the lower bound \col{fewnom} (blue and red lines resp.),
\col{p\'olya} (plus signs) and \col{separ} (stars signs).}
\label{fig_exemplescalar}
\end{figure}
\end{example}

\subsection{Illustration of lower bounds for the eigenvalues of matrix polynomials}
\label{sec:numexp}
Note that to compute the eigenvalues of the following examples 
we had to use the tropical scaling algorithm that we recall in the next section.

In the following examples, we are comparing the value of the ratio 
$\frac{|\zeta_1\dots \zeta_{nk}|}{(\alpha_1\dots\alpha_k)^n}$,
$k=1,\ldots, d$,
with the lower bounds~\eqref{i_edge_bound} 
or~\eqref{i_edge_bound_n} of Section~\ref{sec:Main_results} 
for some matrix polynomials $P$.
In view of Remark~\ref{comparenorms}, we computed only
the lower bounds obtained under~\ref{assump2} for
the normalized Frobenius norm $\|\cdot\|_\nf$,
and the lower bounds obtained under~\ref{assump1} for 
the normalized Schatten $1$-norm (or normalized trace norm) 
$\|\cdot\|_{*1}$ (which satisfies~\ref{assump1} but not \ref{assump2}, 
see Property~\ref{prop-schatten}).

Then, using the notations of Section~\ref{sec:Main_results},
we shall show in tables or plots the following quantities with respect
to $k=1,\ldots, d$ for comparison:
- the ratio $\frac{|\zeta_1\dots \zeta_{nk}|}{(\alpha_1\dots\alpha_k)^n}$,
denoted \col{ratio};
- the universal lower bound
generalizing the lower bound of P\'olya to the matrix case, 
given by $c L_k^n$ (see~\eqref{i_edge_bound_n})
with $L_k$  of~\eqref{gen_polya} or~\eqref{gen_polya_fro},
denoted \col{g-p\'olya};
- the lower bound involving the number of monomials, given by
$c L_k^n$ with  $L_k$ of~\eqref{nnz_bound} or~\eqref{nnz_bound_fro},
denoted \col{fewnom};
- the lower bound based on the 
separation between the tropical roots, given by
$c (L_k^{\text{prox}})^n$ or $c (L_k^{*,\text{prox}})^n$ with 
$L_k^{\text{prox}}$ and $L_k^{*,\text{prox}}$ of~\eqref{L_k_j_prox}
and~\eqref{L_k_j_fro_prox} respectively,
denoted \col{separ}.
Also, in tables, the absciss\ae\ of the vertices of the Newton polygon of 
$\trop p$ are indicated by the symbol~$^*$.

\begin{example}
\label{matx_exp_1} \rm 
Consider the following matrix polynomial
\[
P_1(\lambda)=10^{-7}U_0+10^2\lambda^2U_2+10^{7}\lambda^4U_4+10\lambda^7U_7+10^{-8}\lambda^9U_9\enspace,
\]
where all the matrices $U_j$, $j\in \{0,2,4,7,9\}$, are unitary of dimension $3$.  Here, we shall only consider the normalized Frobenius norm since
for a unitary matrix $U$, we have $\|U\|_\nf=\|U\|_{*1}=1$,
so the tropical polynomial associated with $P_1$  are the same for both norms,
and the lower bounds under~\ref{assump1} are the same, 
and are thus weaker than the lower bounds under~\ref{assump2}
for the normalized Frobenius norm.
The Newton polygon of the tropical polynomial corresponding to $P_1$
(for the normalized Frobenius norm) is shown in Figure~\ref{new_exp_mtx_1},
\begin{figure}[htbp]
\centering
\includegraphics[scale=0.4]{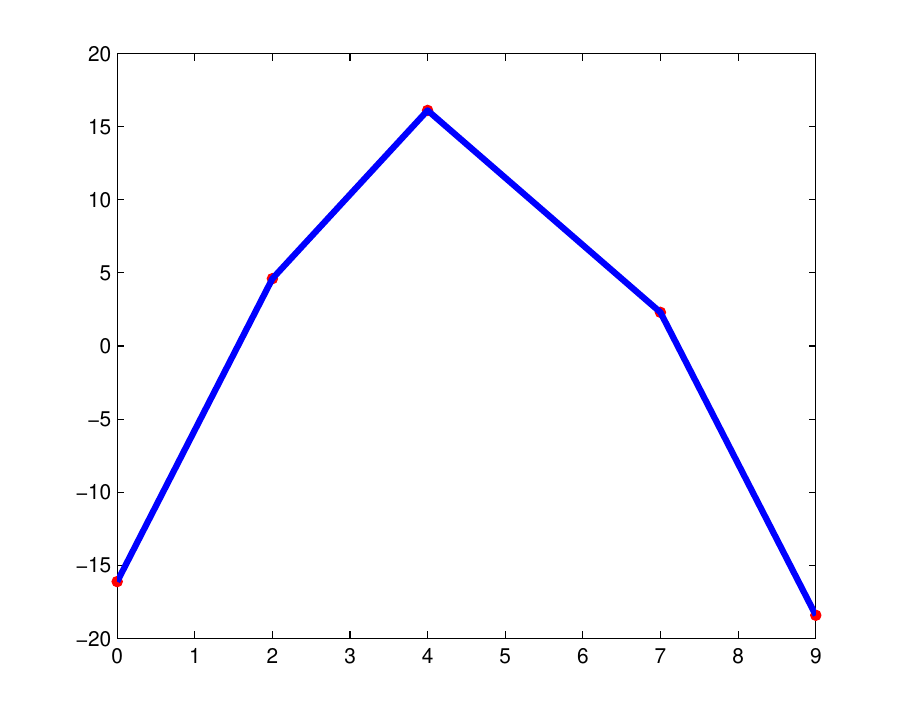}
\caption{The Newton polygon of the tropical polynomial corresponding to $P_1$.}
\label{new_exp_mtx_1}
\end{figure}
and its tropical roots are $10^{-9/2},\; 10^{-5/2},\; 10^2$, and $10^{9/2}$ with
multiplicities $2,\; 2,\; 3$, and $2$, respectively.
Figure~\ref{tbl_mtx1} shows the lower bounds for all values of $1\leq k \leq 9$,
and the maximum of the ratios for a sample of $1000$ random choices
of the unitary matrices $U_j$.
The ratios slightly change for different random choices of 
the unitary matrices $U_j$, but with a difference smaller than $6\, 10^{-4}$.
Note that when $1\leq k\leq 9$ increases, the generalized lower 
bound of P\'olya decreases,
and should be replaced by its symmetrized version with $k$ replaced by
$d-k$.
However the bounds using the separation between tropical roots 
are the best ones, and they are tight at the vertices of the Newton polygon
of the tropical polynomial associated with $P_1$.
 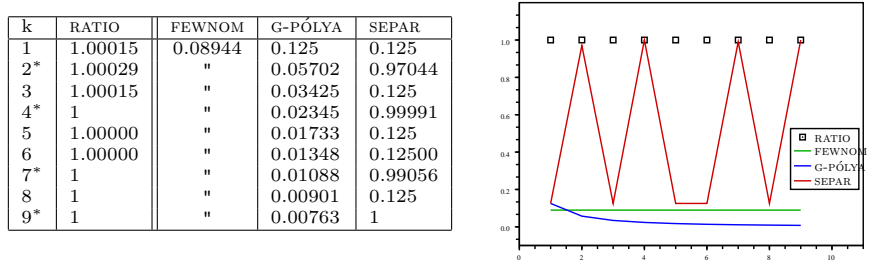
\begin{figure}[htbp]%
\begin{center}
\begin{minipage}[l]{0.51\textwidth}{\scriptsize
\hfill 
\begin{tabular}{|l|l||c|l|l|}
\hline
k & \col{ratio} & \col{fewnom} &  \col{g-p\'olya} & \col{separ} \\
\hline
1  &  1.00015  &  0.08944  &  0.125    &  0.125    \\  
2$^*$  &  1.00029  &  \myquote  &  0.05702  &  0.97044  \\  
3  &  1.00015  &  \myquote  &  0.03425  &  0.125    \\  
4$^*$  &  1        &  \myquote  &  0.02345  &  0.99991  \\  
5  &  1.00000  &  \myquote  &  0.01733  &  0.125    \\  
6  &  1.00000  &  \myquote  &  0.01348  &  0.12500  \\  
7$^*$  &  1        &  \myquote  &  0.01088  &  0.99056  \\  
8  &  1        &  \myquote  &  0.00901  &  0.125    \\  
9$^*$  &  1        &  \myquote  &  0.00763  &  1        \\  
\hline
\end{tabular}}
\end{minipage}
\begin{minipage}[r]{0.47\textwidth}
\hfill
\begin{picture}(0,0)%
\includegraphics{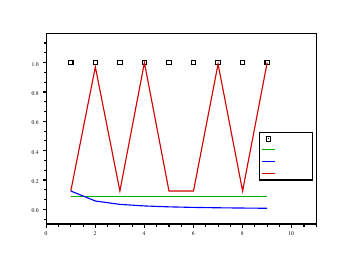}%
\end{picture}%
\setlength{\unitlength}{1184sp}%
\begingroup\makeatletter\ifx\SetFigFont\undefined%
\gdef\SetFigFont#1#2#3#4#5{%
  \reset@font\fontsize{#1}{#2pt}%
  \fontfamily{#3}\fontseries{#4}\fontshape{#5}%
  \selectfont}%
\fi\endgroup%
\begin{picture}(9645,6829)(-32,-5957)
\put(7383,-2920){\makebox(0,0)[lb]{\smash{{\SetFigFont{5}{6.0}{\rmdefault}{\mddefault}{\updefault}{\color[rgb]{0.000,0.000,0.000} \col{ratio}}%
}}}}
\put(7383,-3208){\makebox(0,0)[lb]{\smash{{\SetFigFont{5}{6.0}{\rmdefault}{\mddefault}{\updefault}{\color[rgb]{0.000,0.000,0.000} \col{fewnom}}%
}}}}
\put(7383,-3548){\makebox(0,0)[lb]{\smash{{\SetFigFont{5}{6.0}{\rmdefault}{\mddefault}{\updefault}{\color[rgb]{0.000,0.000,0.000} \col{g-p\'olya}}%
}}}}
\put(7383,-3836){\makebox(0,0)[lb]{\smash{{\SetFigFont{5}{6.0}{\rmdefault}{\mddefault}{\updefault}{\color[rgb]{0.000,0.000,0.000} \col{separ}}%
}}}}
\end{picture}%
\end{minipage}
\end{center}
\vspace{-1em}
\caption{Comparison of the bounds for $P_1$ using the normalized
Frobenius norm.}
\label{tbl_mtx1}
\end{figure}
\end{example}

\begin{example} \rm
\label{matx_exp_2}
In the following example, we perturb the previous polynomial by adding
coefficients to  $P_1$, in such a way that the Newton polygon 
of the tropical polynomial associated with $P_1$ does not change,
either for the normalized Frobenius norm or at least for
the normalized Schatten $1$-norm.
Here we fix the unitary matrices to be either the identity matrix or
its opposite, and we consider the following matrices:
\[
B=b \begin{pmatrix} 10^{-15} & 1 & 1\\ 0 & 1 & 1\\ 1 & 2 & 3 \end{pmatrix},
\quad C=\frac{\sqrt{3}}{14}\begin{pmatrix} 1 & 2 & 3\\ 2 & 4 & 6\\ 3 & 6 & 9 \end{pmatrix}\enspace
\]
where $b$ is chosen so that $\|B\|_\nf=1$, and $C$ is also such that
$\|C\|_\nf$=1.
The matrix $C$ is singular, since it has rank $1$,
and it satisfies $\|C\|_{*1}=1/\sqrt{3}$.
The matrix $B$ is ill-conditioned, 
so that $\|B\|_{*1}<\|B\|_{\nf}=1$ too.

We shall consider the two polynomials:
\begin{align*}
P_2(\lambda)=&10^{-7}U_0+10^{-5/2}\lambda C_1+10^2\lambda^2U_2+10^{9/2}\lambda^3C_1+10^{7}\lambda^4U_4+10^5\lambda^5 C_1\\
&+10^3\lambda^6B_1+10\lambda^7U_7+10^{-7/2}\lambda^8 B_1+10^{-8}\lambda^9U_9\\
P_3(\lambda)=&10^{-7}U_0+10^{-5/2}\lambda C_2+10^2\lambda^2U_2+10^{9/2}\lambda^3C_2+10^{7}\lambda^4U_4+10^5\lambda^5 C_2\\
&+10^3\lambda^6B_2+10\lambda^7U_7+10^{-7/2}\lambda^8 B_2+10^{-8}\lambda^9U_9\enspace,
\end{align*}
with $U_0=U_4=U_7=-I$, $U_2=U_9=I$, $B_1=B$, $C_1=C$,
$B_2=\|B\|_{*1}^{-1} B$, $C_2=\sqrt{3} C$.
Since for any unitary matrix $U$ with dimension $3$,
 $\|U\|_{*1}=\|U\|_\nf=1$, 
and since for any complex matrix $A$ of dimension $3$,
we have $\|A\|_{*1}\leq \|A\|_\nf$,
the tropical polynomial associated with $P_2$ for either the normalized Frobenius
norm or the Schatten $1$-norm coincides with the one associated with $P_1$.
Hence, as for $P_1$, the bounds for $P_2$ with the normalized
Schatten $1$-norm are necessarily weaker than the ones 
with the normalized Frobenius norm.
Moreover, all the bounds presented in 
Table~\ref{tbl_mtx1} for $P_1$ remain the same for $P_2$ 
except the bounds based on the number of nonzero coefficients.
We present all these bounds in Figure~\ref{tbl_mtx5}
together with the new ratios. Since these ratios are different from
the ones of Table~\ref{tbl_mtx1}, one can see that the lower bounds
based on the separation between tropical roots may be tight.

\begin{figure}[htbp] %
\begin{center}
\begin{minipage}[c]{0.51\textwidth}{\scriptsize
\hfill 
\begin{tabular}{|l|l||c|l|l|}
\hline
k & \col{ratio} & \col{fewnom} &  \col{g-p\'olya} & \col{separ} \\
\hline
1  &  0.45628  &  0.03162  &  0.125    &  0.125    \\  
2$^*$  &  1.03298  &  \myquote  &  0.05702  &  0.97044  \\  
3  &  0.45627  &  \myquote  &  0.03425  &  0.125    \\  
4$^*$  &  1.00009  &  \myquote  &  0.02345  &  0.99991  \\  
5  &  0.39258  &  \myquote  &  0.01733  &  0.125    \\  
6  &  0.39258  &  \myquote  &  0.01348  &  0.12500  \\  
7$^*$  &  1.00867  &  \myquote  &  0.01088  &  0.99056  \\  
8  &  0.47428  &  \myquote  &  0.00901  &  0.125    \\  
9$^*$  &  1        &  \myquote  &  0.00763  &  1        \\  
\hline
\end{tabular}}
\end{minipage}
\begin{minipage}[c]{0.47\textwidth}
\hfill
\begin{picture}(0,0)%
\includegraphics{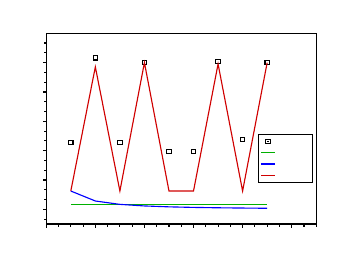}%
\end{picture}%
\setlength{\unitlength}{1184sp}%
\begingroup\makeatletter\ifx\SetFigFont\undefined%
\gdef\SetFigFont#1#2#3#4#5{%
  \reset@font\fontsize{#1}{#2pt}%
  \fontfamily{#3}\fontseries{#4}\fontshape{#5}%
  \selectfont}%
\fi\endgroup%
\begin{picture}(9645,6829)(-32,-5957)
\put(1156,-5410){\makebox(0,0)[lb]{\smash{{\SetFigFont{5}{6.0}{\sfdefault}{\mddefault}{\updefault}{\color[rgb]{0.000,0.000,0.000}0}%
}}}}
\put(2465,-5410){\makebox(0,0)[lb]{\smash{{\SetFigFont{5}{6.0}{\sfdefault}{\mddefault}{\updefault}{\color[rgb]{0.000,0.000,0.000}2}%
}}}}
\put(3774,-5410){\makebox(0,0)[lb]{\smash{{\SetFigFont{5}{6.0}{\sfdefault}{\mddefault}{\updefault}{\color[rgb]{0.000,0.000,0.000}4}%
}}}}
\put(5083,-5410){\makebox(0,0)[lb]{\smash{{\SetFigFont{5}{6.0}{\sfdefault}{\mddefault}{\updefault}{\color[rgb]{0.000,0.000,0.000}6}%
}}}}
\put(6392,-5410){\makebox(0,0)[lb]{\smash{{\SetFigFont{5}{6.0}{\sfdefault}{\mddefault}{\updefault}{\color[rgb]{0.000,0.000,0.000}8}%
}}}}
\put(7655,-5410){\makebox(0,0)[lb]{\smash{{\SetFigFont{5}{6.0}{\sfdefault}{\mddefault}{\updefault}{\color[rgb]{0.000,0.000,0.000}10}%
}}}}
\put(811,-4802){\makebox(0,0)[lb]{\smash{{\SetFigFont{5}{6.0}{\sfdefault}{\mddefault}{\updefault}{\color[rgb]{0.000,0.000,0.000}0.0}%
}}}}
\put(811,-4019){\makebox(0,0)[lb]{\smash{{\SetFigFont{5}{6.0}{\sfdefault}{\mddefault}{\updefault}{\color[rgb]{0.000,0.000,0.000}0.2}%
}}}}
\put(811,-3236){\makebox(0,0)[lb]{\smash{{\SetFigFont{5}{6.0}{\sfdefault}{\mddefault}{\updefault}{\color[rgb]{0.000,0.000,0.000}0.4}%
}}}}
\put(811,-2453){\makebox(0,0)[lb]{\smash{{\SetFigFont{5}{6.0}{\sfdefault}{\mddefault}{\updefault}{\color[rgb]{0.000,0.000,0.000}0.6}%
}}}}
\put(811,-1671){\makebox(0,0)[lb]{\smash{{\SetFigFont{5}{6.0}{\sfdefault}{\mddefault}{\updefault}{\color[rgb]{0.000,0.000,0.000}0.8}%
}}}}
\put(811,-888){\makebox(0,0)[lb]{\smash{{\SetFigFont{5}{6.0}{\sfdefault}{\mddefault}{\updefault}{\color[rgb]{0.000,0.000,0.000}1.0}%
}}}}
\put(7367,-2988){\makebox(0,0)[lb]{\smash{{\SetFigFont{5}{6.0}{\sfdefault}{\mddefault}{\updefault}{\color[rgb]{0.000,0.000,0.000} \col{ratio}}%
}}}}
\put(7367,-3276){\makebox(0,0)[lb]{\smash{{\SetFigFont{5}{6.0}{\sfdefault}{\mddefault}{\updefault}{\color[rgb]{0.000,0.000,0.000} \col{fewnom}}%
}}}}
\put(7367,-3616){\makebox(0,0)[lb]{\smash{{\SetFigFont{5}{6.0}{\sfdefault}{\mddefault}{\updefault}{\color[rgb]{0.000,0.000,0.000} \col{g-p\'olya}}%
}}}}
\put(7367,-3904){\makebox(0,0)[lb]{\smash{{\SetFigFont{5}{6.0}{\sfdefault}{\mddefault}{\updefault}{\color[rgb]{0.000,0.000,0.000} \col{separ}}%
}}}}
\end{picture}%
\end{minipage}
\end{center}
\vspace{-1em}
\caption{Comparison of the bounds for $P_2$ using the normalized
Frobenius norm.}
\label{tbl_mtx5}
\end{figure}
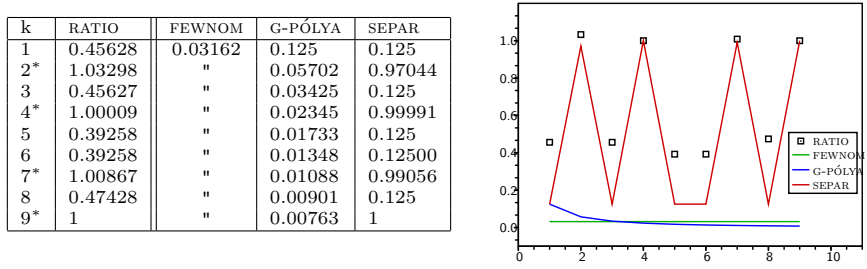

Let us consider now the polynomial $P_3$.
The normalized Schatten $1$-norm of the coefficients of $P_3$ coincide
with the normalized Frobenius norm of the coefficients of $P_2$, hence
the tropical polynomial associated with $P_3$ for the normalized
Schatten $1$-norm coincides with the one associated with $P_1$
for the normalized Frobenius norm, so that the tropical roots are
the same.
However, since the Schatten $1$-norm satisfies~\ref{assump1}
but not~\ref{assump2}, we can only get the bounds based on~\ref{assump1}
which are necessarily lower than the one presented above for $P_1$
with the normalized Frobenius norm.
Finally, the tropical polynomial associated
with $P_3$ for the normalized Frobenius norm differs from the one
associated with $P_1$, so that the tropical roots and the bounds differ too.
All tropical roots of this new tropical polynomial
have multiplicity $1$, which means that
all indices are absciss\ae\ of vertices of its Newton polygon.
We present in Figure~\ref{tbl_mtx6}, the ratios and the lower bounds for 
both normalized Frobenius and  Schatten $1$-norms.
Here the ratios and lower bounds for the normalized Schatten $1$-norm
 are lower than the ones for the normalized Frobenius norm.
However, one can see that the results for the  normalized Frobenius norm
are still better than the ones for the normalized Schatten $1$-norm.

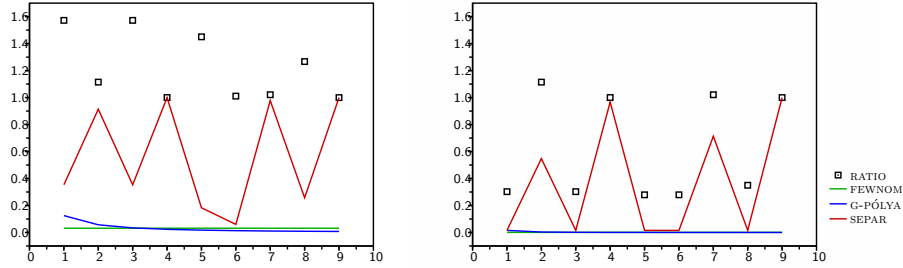
\begin{figure}[htbp] %
\centering
{\scriptsize
\begin{tabular}{|l||l||c|l|l||l||c|l|l|}
\hline
&\multicolumn{4}{|c||}{norm.\ Frobenius norm} &\multicolumn{4}{|c|}{norm.\ Schatten $1$-norm}\\
\hline
k & \col{ratio} & \col{fewnom} &  \col{g-p\'olya} & \col{separ} 
& \col{ratio} & \col{fewnom} &  \col{g-p\'olya} & \col{separ} \\
\hline
1  &  1.57223  &  0.03162  &  0.125    &  0.35355  &  0.30258  &  0.001  &  0.01563  &  0.01563  \\  
2$^*$  &  1.11466  &  \myquote  &  0.05702  &  0.91391  &  1.11466  &  \myquote  &  0.00325  &  0.54771  \\  
3  &  1.57223  &  \myquote  &  0.03425  &  0.35355  &  0.30258  &  \myquote  &  0.00117  &  0.01562  \\  
4$^*$   &  1.00028  &  \myquote  &  0.02345  &  0.99972  &  1.00028  &  \myquote  &  0.00055  &  0.96682  \\  
5  &  1.45041  &  \myquote  &  0.01733  &  0.18345  &  0.27913  &  \myquote  &  0.00030  &  0.01562  \\  
6  &  1.01082  &  \myquote  &  0.01348  &  0.05989  &  0.27913  &  \myquote  &  0.00018  &  0.01533  \\  
7$^*$   &  1.02104  &  \myquote  &  0.01088  &  0.97788  &  1.02104  &  \myquote  &  0.00012  &  0.71337  \\  
8  &  1.26686  &  \myquote  &  0.00901  &  0.25721  &  0.34983  &  \myquote  &  0.00008  &  0.01563  \\  
9$^*$   &  1        &  \myquote  &  0.00763  &  1        &  1        &  \myquote  &  0.00006  &  1        \\  
\hline
\end{tabular}}

\begin{picture}(0,0)%
\includegraphics{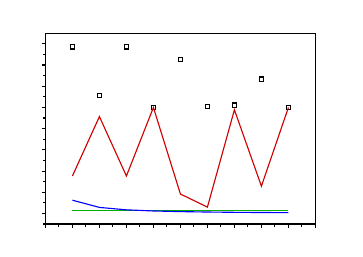}%
\end{picture}%
\setlength{\unitlength}{1184sp}%
\begingroup\makeatletter\ifx\SetFigFont\undefined%
\gdef\SetFigFont#1#2#3#4#5{%
  \reset@font\fontsize{#1}{#2pt}%
  \fontfamily{#3}\fontseries{#4}\fontshape{#5}%
  \selectfont}%
\fi\endgroup%
\begin{picture}(9624,6808)(-11,-5957)
\put(1156,-5410){\makebox(0,0)[lb]{\smash{{\SetFigFont{5}{6.0}{\sfdefault}{\mddefault}{\updefault}{\color[rgb]{0.000,0.000,0.000}0}%
}}}}
\put(1876,-5410){\makebox(0,0)[lb]{\smash{{\SetFigFont{5}{6.0}{\sfdefault}{\mddefault}{\updefault}{\color[rgb]{0.000,0.000,0.000}1}%
}}}}
\put(2596,-5410){\makebox(0,0)[lb]{\smash{{\SetFigFont{5}{6.0}{\sfdefault}{\mddefault}{\updefault}{\color[rgb]{0.000,0.000,0.000}2}%
}}}}
\put(3316,-5410){\makebox(0,0)[lb]{\smash{{\SetFigFont{5}{6.0}{\sfdefault}{\mddefault}{\updefault}{\color[rgb]{0.000,0.000,0.000}3}%
}}}}
\put(4036,-5410){\makebox(0,0)[lb]{\smash{{\SetFigFont{5}{6.0}{\sfdefault}{\mddefault}{\updefault}{\color[rgb]{0.000,0.000,0.000}4}%
}}}}
\put(4756,-5410){\makebox(0,0)[lb]{\smash{{\SetFigFont{5}{6.0}{\sfdefault}{\mddefault}{\updefault}{\color[rgb]{0.000,0.000,0.000}5}%
}}}}
\put(5476,-5410){\makebox(0,0)[lb]{\smash{{\SetFigFont{5}{6.0}{\sfdefault}{\mddefault}{\updefault}{\color[rgb]{0.000,0.000,0.000}6}%
}}}}
\put(6196,-5410){\makebox(0,0)[lb]{\smash{{\SetFigFont{5}{6.0}{\sfdefault}{\mddefault}{\updefault}{\color[rgb]{0.000,0.000,0.000}7}%
}}}}
\put(6916,-5410){\makebox(0,0)[lb]{\smash{{\SetFigFont{5}{6.0}{\sfdefault}{\mddefault}{\updefault}{\color[rgb]{0.000,0.000,0.000}8}%
}}}}
\put(7636,-5410){\makebox(0,0)[lb]{\smash{{\SetFigFont{5}{6.0}{\sfdefault}{\mddefault}{\updefault}{\color[rgb]{0.000,0.000,0.000}9}%
}}}}
\put(8310,-5410){\makebox(0,0)[lb]{\smash{{\SetFigFont{5}{6.0}{\sfdefault}{\mddefault}{\updefault}{\color[rgb]{0.000,0.000,0.000}10}%
}}}}
\put(811,-4910){\makebox(0,0)[lb]{\smash{{\SetFigFont{5}{6.0}{\sfdefault}{\mddefault}{\updefault}{\color[rgb]{0.000,0.000,0.000}0.0}%
}}}}
\put(811,-4345){\makebox(0,0)[lb]{\smash{{\SetFigFont{5}{6.0}{\sfdefault}{\mddefault}{\updefault}{\color[rgb]{0.000,0.000,0.000}0.2}%
}}}}
\put(811,-3780){\makebox(0,0)[lb]{\smash{{\SetFigFont{5}{6.0}{\sfdefault}{\mddefault}{\updefault}{\color[rgb]{0.000,0.000,0.000}0.4}%
}}}}
\put(811,-3214){\makebox(0,0)[lb]{\smash{{\SetFigFont{5}{6.0}{\sfdefault}{\mddefault}{\updefault}{\color[rgb]{0.000,0.000,0.000}0.6}%
}}}}
\put(811,-2649){\makebox(0,0)[lb]{\smash{{\SetFigFont{5}{6.0}{\sfdefault}{\mddefault}{\updefault}{\color[rgb]{0.000,0.000,0.000}0.8}%
}}}}
\put(811,-2084){\makebox(0,0)[lb]{\smash{{\SetFigFont{5}{6.0}{\sfdefault}{\mddefault}{\updefault}{\color[rgb]{0.000,0.000,0.000}1.0}%
}}}}
\put(811,-1518){\makebox(0,0)[lb]{\smash{{\SetFigFont{5}{6.0}{\sfdefault}{\mddefault}{\updefault}{\color[rgb]{0.000,0.000,0.000}1.2}%
}}}}
\put(811,-953){\makebox(0,0)[lb]{\smash{{\SetFigFont{5}{6.0}{\sfdefault}{\mddefault}{\updefault}{\color[rgb]{0.000,0.000,0.000}1.4}%
}}}}
\put(811,-388){\makebox(0,0)[lb]{\smash{{\SetFigFont{5}{6.0}{\sfdefault}{\mddefault}{\updefault}{\color[rgb]{0.000,0.000,0.000}1.6}%
}}}}
\end{picture}%
\hspace{-1em}
\begin{picture}(0,0)%
\includegraphics{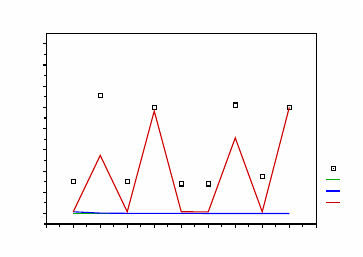}%
\end{picture}%
\setlength{\unitlength}{1184sp}%
\begingroup\makeatletter\ifx\SetFigFont\undefined%
\gdef\SetFigFont#1#2#3#4#5{%
  \reset@font\fontsize{#1}{#2pt}%
  \fontfamily{#3}\fontseries{#4}\fontshape{#5}%
  \selectfont}%
\fi\endgroup%
\begin{picture}(9645,6829)(-32,-5957)
\put(1156,-5410){\makebox(0,0)[lb]{\smash{{\SetFigFont{5}{6.0}{\sfdefault}{\mddefault}{\updefault}{\color[rgb]{0.000,0.000,0.000}0}%
}}}}
\put(1876,-5410){\makebox(0,0)[lb]{\smash{{\SetFigFont{5}{6.0}{\sfdefault}{\mddefault}{\updefault}{\color[rgb]{0.000,0.000,0.000}1}%
}}}}
\put(2596,-5410){\makebox(0,0)[lb]{\smash{{\SetFigFont{5}{6.0}{\sfdefault}{\mddefault}{\updefault}{\color[rgb]{0.000,0.000,0.000}2}%
}}}}
\put(3316,-5410){\makebox(0,0)[lb]{\smash{{\SetFigFont{5}{6.0}{\sfdefault}{\mddefault}{\updefault}{\color[rgb]{0.000,0.000,0.000}3}%
}}}}
\put(4036,-5410){\makebox(0,0)[lb]{\smash{{\SetFigFont{5}{6.0}{\sfdefault}{\mddefault}{\updefault}{\color[rgb]{0.000,0.000,0.000}4}%
}}}}
\put(4756,-5410){\makebox(0,0)[lb]{\smash{{\SetFigFont{5}{6.0}{\sfdefault}{\mddefault}{\updefault}{\color[rgb]{0.000,0.000,0.000}5}%
}}}}
\put(5476,-5410){\makebox(0,0)[lb]{\smash{{\SetFigFont{5}{6.0}{\sfdefault}{\mddefault}{\updefault}{\color[rgb]{0.000,0.000,0.000}6}%
}}}}
\put(6196,-5410){\makebox(0,0)[lb]{\smash{{\SetFigFont{5}{6.0}{\sfdefault}{\mddefault}{\updefault}{\color[rgb]{0.000,0.000,0.000}7}%
}}}}
\put(6916,-5410){\makebox(0,0)[lb]{\smash{{\SetFigFont{5}{6.0}{\sfdefault}{\mddefault}{\updefault}{\color[rgb]{0.000,0.000,0.000}8}%
}}}}
\put(7636,-5410){\makebox(0,0)[lb]{\smash{{\SetFigFont{5}{6.0}{\sfdefault}{\mddefault}{\updefault}{\color[rgb]{0.000,0.000,0.000}9}%
}}}}
\put(8310,-5410){\makebox(0,0)[lb]{\smash{{\SetFigFont{5}{6.0}{\sfdefault}{\mddefault}{\updefault}{\color[rgb]{0.000,0.000,0.000}10}%
}}}}
\put(811,-4910){\makebox(0,0)[lb]{\smash{{\SetFigFont{5}{6.0}{\sfdefault}{\mddefault}{\updefault}{\color[rgb]{0.000,0.000,0.000}0.0}%
}}}}
\put(811,-4345){\makebox(0,0)[lb]{\smash{{\SetFigFont{5}{6.0}{\sfdefault}{\mddefault}{\updefault}{\color[rgb]{0.000,0.000,0.000}0.2}%
}}}}
\put(811,-3780){\makebox(0,0)[lb]{\smash{{\SetFigFont{5}{6.0}{\sfdefault}{\mddefault}{\updefault}{\color[rgb]{0.000,0.000,0.000}0.4}%
}}}}
\put(811,-3214){\makebox(0,0)[lb]{\smash{{\SetFigFont{5}{6.0}{\sfdefault}{\mddefault}{\updefault}{\color[rgb]{0.000,0.000,0.000}0.6}%
}}}}
\put(811,-2649){\makebox(0,0)[lb]{\smash{{\SetFigFont{5}{6.0}{\sfdefault}{\mddefault}{\updefault}{\color[rgb]{0.000,0.000,0.000}0.8}%
}}}}
\put(811,-2084){\makebox(0,0)[lb]{\smash{{\SetFigFont{5}{6.0}{\sfdefault}{\mddefault}{\updefault}{\color[rgb]{0.000,0.000,0.000}1.0}%
}}}}
\put(811,-1518){\makebox(0,0)[lb]{\smash{{\SetFigFont{5}{6.0}{\sfdefault}{\mddefault}{\updefault}{\color[rgb]{0.000,0.000,0.000}1.2}%
}}}}
\put(811,-953){\makebox(0,0)[lb]{\smash{{\SetFigFont{5}{6.0}{\sfdefault}{\mddefault}{\updefault}{\color[rgb]{0.000,0.000,0.000}1.4}%
}}}}
\put(811,-388){\makebox(0,0)[lb]{\smash{{\SetFigFont{5}{6.0}{\sfdefault}{\mddefault}{\updefault}{\color[rgb]{0.000,0.000,0.000}1.6}%
}}}}
\put(9106,-3706){\makebox(0,0)[lb]{\smash{{\SetFigFont{5}{6.0}{\sfdefault}{\mddefault}{\updefault}{\color[rgb]{0.000,0.000,0.000} \col{ratio}}%
}}}}
\put(9106,-3994){\makebox(0,0)[lb]{\smash{{\SetFigFont{5}{6.0}{\sfdefault}{\mddefault}{\updefault}{\color[rgb]{0.000,0.000,0.000} \col{fewnom}}%
}}}}
\put(9106,-4334){\makebox(0,0)[lb]{\smash{{\SetFigFont{5}{6.0}{\sfdefault}{\mddefault}{\updefault}{\color[rgb]{0.000,0.000,0.000} \col{g-p\'olya}}%
}}}}
\put(9106,-4622){\makebox(0,0)[lb]{\smash{{\SetFigFont{5}{6.0}{\sfdefault}{\mddefault}{\updefault}{\color[rgb]{0.000,0.000,0.000} \col{separ}}%
}}}}
\end{picture}%
\caption{Comparison of the bounds for $P_3$ using the normalized
Frobenius (on left) and Schatten $1$-norms (on right).
The symbol~$^*$ corresponds to 
the absciss\ae\ of the vertices of the Newton polygon of the tropical polynomial
associated with the normalized Frobenius norm.}
\label{tbl_mtx6}
\end{figure}

\end{example}

\begin{example}
\label{matx_exp_3}
\rm
In order to obtain a polynomial for which the lower bounds for the normalized 
Schatten $1$-norm are better than the ones 
for the normalized Frobenius norm, one need to take matrices $A_i$ such that
$\|A_i\|_\nf/\|A_i\|_{*1}$ is as large as possible. This means that
the dimension $n$ is large and that some of the singular values of $A_i$
are zero, or at least that the ratio between
the maximal and minimal singular values is large, which implies that $A_i$ is 
singular or at least has a large condition number.
Moreover some parameters of the lower bounds need to be 
smaller than $\|A_i\|_\nf/\|A_i\|_{*1}$, which is possible for instance for
the lower bounds based on the separation between tropical roots, 
when these tropical roots have a small multiplicity and 
are well separated.
We shall show here such examples.

We consider a matrix polynomial with degree $3$ and dimension
$n=10$:
\[
P_4(\lambda)=A_0+A_1\lambda+A_2\lambda^2+A_3 \lambda^3 \enspace,
\]
where all the matrices $A_i,i=0,\ldots,3$ are diagonal matrices:
$A_i=\diag(d_i)$ with $d_i\in\C^n$.
We choose $A_0=A_3=I$, 
so that $c=1$, $\alpha_1\alpha_2\alpha_3=1$
(for norms such that $\|I\|=1$) and  for $k=3$ the ratio
$\frac{|\zeta_1\dots \zeta_{nk}|}{(\alpha_1\dots\alpha_k)^n}$
is equal to $1$, which is equal to the lower bound based on separation
of tropical roots.
We choose the diagonal elements of the matrices $A_1,A_2$ as follows:
\[
d_1:= \frac{a_1}{1+b/n} (1, \ldots, 1, 1+b)\enspace,\quad 
d_2:= \frac{a_2}{1+b/n} (1+b, 1, \ldots, 1)\enspace,
\]
with $a_1,a_2,b> 0$.
Then 
\begin{gather*} \|A_1\|_{*1}= \frac{\|d_1\|_1}{n}=a_1 \enspace,\quad
\|A_2\|_{*1}= \frac{\|d_2\|_1}{n}=a_2 \enspace,\\
\|A_1\|_\nf=\frac{\|d_1\|_2}{\sqrt{n}}=a_1 f_b\enspace,\quad
\|A_2\|_\nf=\frac{\|d_2\|_2}{\sqrt{n}}=a_2 f_b \enspace,
\end{gather*}
with
\[ f_b:=\frac{\sqrt{1+\frac{2b+b^2}{n}}}{1+b/n}>1\enspace.\]
The constant $f_b$ is large when $b$ is large, that is when the 
condition numbers of $A_1$ and $A_2$ are large.

When $a_1=a_2>1$, the tropical polynomial associated with $P_4$
for the normalized Schatten $1$-norm (resp.\ the normalized Frobenius norm)
has $3$ different tropical roots equal to $1/a_1$, $1$ and $a_1$
(resp.\ $1/(a_1 f_b)$, $1$ and $a_1 f_b$).
We present in Tables~\ref{tbl_mtx7} and~\ref{tbl_mtx8}, 
the ratios and the lower bounds for 
both normalized Frobenius and  Schatten $1$-norms,
when $a_1=10^4$, and  $b=2$ and $b=10$ respectively.
When $b=10$, the matrices have a large condition number, so that the
ratios $\frac{|\zeta_1\dots \zeta_{nk}|}{(\alpha_1\dots\alpha_k)^n}$
for $k=1$ or $2$ are very large. When either $b=2$ or $b=10$,
we see that the lower bounds based on the separation between the tropical
roots are the best ones, and that they are nearer the ratios
($\frac{|\zeta_1\dots \zeta_{nk}|}{(\alpha_1\dots\alpha_k)^n}$) in the
case of the Schatten $1$-norm than in the case of the Frobenius norm,
although they are far in the case where $b=10$.

\begin{table}[htp]
\centering
{\scriptsize
\begin{tabular}{|l||l||c|l|l||l||c|l|l|}
\hline
&\multicolumn{4}{|c||}{norm.\ Frobenius norm} &\multicolumn{4}{|c|}{norm.\ Schatten $1$-norm}\\
\hline
k & \col{ratio} & \col{fewnom} &  \col{g-p\'olya} & \col{separ} 
& \col{ratio} & \col{fewnom} &  \col{g-p\'olya} & \col{separ} \\
\hline
1$^*$  &  6.30697  &  0.00098  &  0.00098  &  0.99911  &  2.06667  &  9.5D-07  &  9.5D-07  &  0.81873  \\  
2$^*$  &  6.30697  &  0.00098  &  0.00007  &  0.99911  &  2.06667  &  9.5D-07  &  5.1D-09  &  0.81873  \\  
3$^*$  &  1        &  0.00098  &  0.00001  &  1        &  1        &  9.5D-07  &  1.7D-10  &  1        \\  
\hline
\end{tabular}
\caption{Comparison of the bounds for $P_4$ with $a_1=a_2=10^4$ and $b=2$.}
\label{tbl_mtx7}
}
\end{table}
\begin{table}[htp]
\centering
{\scriptsize
\begin{tabular}{|l||l||c|l|l||l||c|l|l|}
\hline
&\multicolumn{4}{|c||}{norm.\ Frobenius norm} &\multicolumn{4}{|c|}{norm.\ Schatten $1$-norm}\\
\hline
k & \col{ratio} & \col{fewnom} &  \col{g-p\'olya} & \col{separ} 
& \col{ratio} & \col{fewnom} &  \col{g-p\'olya} & \col{separ} \\
\hline
1$^*$  &  33882.7  &  0.00098  &  0.00098  &  0.99945  &  93.4462  &  9.5D-07  &  9.5D-07  &  0.81873  \\  
2$^*$  &  33882.7  &  0.00098  &  0.00007  &  0.99945  &  93.4462  &  9.5D-07  &  5.1D-09  &  0.81873  \\  
3$^*$  &  1        &  0.00098  &  0.00001  &  1        &  1        &  9.5D-07  &  1.7D-10  &  1        \\  
\hline
\end{tabular}
\caption{Comparison of the bounds for $P_4$ with $a_1=a_2=10^4$ and $b=10$.}
\label{tbl_mtx8}
}
\end{table}

\begin{table}[htp]
\centering
{\scriptsize
\begin{tabular}{|l||l||c|l|l||l||c|l|l|}
\hline
&\multicolumn{4}{|c||}{norm.\ Frobenius norm} &\multicolumn{4}{|c|}{norm.\ Schatten $1$-norm}\\
\hline
k & \col{ratio} & \col{fewnom} &  \col{g-p\'olya} & \col{separ} 
& \col{ratio} & \col{fewnom} &  \col{g-p\'olya} & \col{separ} \\
\hline
 1$^*$  &  1.1D+21  &  0.00098  &  0.00098  &  0.99968  &  1.2D+16  &  9.5D-07  &  9.5D-07  &  0.81873  \\
2$^*$  &  1.1D+21  &  0.00098  &  0.00007  &  0.99968  &  1.2D+16  &  9.5D-07  &  5.1D-09  &  0.81873  \\
3$^*$  &  1        &  0.00098  &  0.00001  &  1        &  1        &  9.5D-07  &  1.7D-10  &  1        \\ 
\hline
\end{tabular}
\caption{Comparison of the bounds for $P_4$ with $a_1=a_2=10^4$ and $b=10^3$.}
\label{tbl_p4_1000}
}
\end{table}

When $a_1>1$ and $a_2=\sqrt{a_1/f_b}$ the tropical polynomial associated with $P_4$
for the normalized Schatten $1$-norm (resp.\ the normalized Frobenius norm)
has only $2$ different tropical roots equal to $1/a_1$ and
$\sqrt{a_1}$ (resp.\ $1/(a_1 f_b)$, $1$ and $\sqrt{a_1 f_b}$)
with respective multiplicities $1$ and $2$.
We present in Tables~\ref{tbl_mtx9} and~\ref{tbl_mtx10}, 
the ratios and the lower bounds for 
both normalized Frobenius and  Schatten $1$-norms,
when $a_1=10^4$, and  $b=2$ and $b=10$ respectively.
When $b=10$, the matrices have a large condition number, so that for $k=1$ or $2$ the
distances (ratios) between the ratio
$\frac{|\zeta_1\dots \zeta_{nk}|}{(\alpha_1\dots\alpha_k)^n}$ 
and its lower bounds are very large, although for $k=2$ this ratio
is smaller than in the case where $a_1=a_2=10^4$ above.
When either $b=2$ or $b=10$,
we see that the lower bounds based on the separation between the tropical
roots are the best ones.
For $k=2$ which is the abscissa of a vertex of the Newton polygon,
the lower bound obtained in the
case of the Schatten $1$-norm is nearer the ratio
$\frac{|\zeta_1\dots \zeta_{nk}|}{(\alpha_1\dots\alpha_k)^n}$ 
than the one obtained in the case of the Frobenius norm,
whereas for $k=3$ the contrary holds.

\begin{table}[htp]
\centering
{\scriptsize
\begin{tabular}{|l||l||c|l|l||l||c|l|l|}
\hline
&\multicolumn{4}{|c||}{norm.\ Frobenius norm} &\multicolumn{4}{|c|}{norm.\ Schatten $1$-norm}\\
\hline
k & \col{ratio} & \col{fewnom} &  \col{g-p\'olya} & \col{separ} 
& \col{ratio} & \col{fewnom} &  \col{g-p\'olya} & \col{separ} \\
\hline
1$^*$  &  6.29864  &  0.00098  &  0.00098  &  0.99999  &  2.06394  &  9.5D-07  &  9.5D-07  &  0.98020  \\  
2  &  0.68418  &  0.00098  &  0.00007  &  0.00098  &  0.39165  &  9.5D-07  &  5.1D-09  &  9.5D-07  \\  
3$^*$  &  1        &  0.00098  &  0.00001  &  1        &  1        &  9.5D-07  &  1.7D-10  &  1        \\  
\hline
\end{tabular}
\caption{Comparison of the bounds for $P_4$ with $a_1=10^4$,
$a_2=\sqrt{a_1/f_b}$, and $b=2$.}
\label{tbl_mtx9}
}
\end{table}
\begin{table}[htp]
\centering
{\scriptsize
\begin{tabular}{|l||l||c|l|l||l||c|l|l|}
\hline
&\multicolumn{4}{|c||}{norm.\ Frobenius norm} &\multicolumn{4}{|c|}{norm.\ Schatten $1$-norm}\\
\hline
k & \col{ratio} & \col{fewnom} &  \col{g-p\'olya} & \col{separ} 
& \col{ratio} & \col{fewnom} &  \col{g-p\'olya} & \col{separ} \\
\hline
1$^*$  &  33754.9  &  0.00098  &  0.00098  &  1.00000  &  93.0935  &  9.5D-07  &  9.5D-07  &  0.98020  \\  
2  &  9.86557  &  0.00098  &  0.00007  &  0.00098  &  0.5181   &  9.5D-07  &  5.1D-09  &  9.5D-07  \\  
3$^*$  &  1        &  0.00098  &  0.00001  &  1        &  1        &  9.5D-07  &  1.7D-10  &  1        \\  
\hline
\end{tabular}
\caption{Comparison of the bounds for $P_4$ with $a_1=10^4$,
$a_2=\sqrt{a_1/f_b}$, and $b=10$.}
\label{tbl_mtx10}
}
\end{table}

\end{example}

\section{The tropical scaling and its relation with the bounds}\label{sec-scaling}
In the examples of the previous section, 
we considered matrix polynomials the coefficients of which 
have different order of magnitude. Hence,
their eigenvalues need to be computed 
by using some scaling, otherwise the relative error of
their approximation may become very large.
We used here the ``tropical scaling'' introduced 
in~\cite{posta09}, that we next recall. We shall see
that in most of the examples of Section~\ref{sec:examples},
it improves the relative error,
by comparison with a usual method, and that its efficiency is
related with the closeness of the ratios 
$\frac{|\zeta_1\dots \zeta_{nk}|}{(\alpha_1\dots\alpha_k)^n}$ to $1$.

\subsection{Tropical scaling}
\label{sec:tropscal}
Consider the matrix polynomial 
$P(\lambda)=\sum_{j=0}^d A_j\lambda^j$ with degree $d$ 
(as in~\eqref{matrixpoly}).
Denote $\trop p(x)=\max_{0\leq j\leq d} \|A_j\| x^j$
the tropical polynomial associated to $P$
(as in~\eqref{eq:tropdefini}) and
$\alpha_1\leq \dots\leq \alpha_d$ the tropical roots of $\trop p$.
Let also $k_0=0,k_1,\dots,k_q=d$ be the sequence of absciss\ae\ of the vertices
of the Newton polygon of $\trop p(x)$, as shown in Figure~\ref{newtonpolygon}.

It has been observed numerically in~\cite{posta09}, and proved
in the special case of quadratic pencils, that 
when the coefficient matrices are well conditioned, 
the tropical roots are good approximations of the moduli of the eigenvalues.
Corollary~\ref{cor-bounds} and Remark~\ref{rem-tight-cond} 
prove the same result for general matrix polynomials of any degree.
So we expect to get $(k_i-k_{i-1})n$ eigenvalues of $P$ with a modulus of order $\alpha_{k_i}$, for $i=1,\ldots, q$.

Given the order of magnitude $\alpha$ of an eigenvalue $\zeta$ of $P$,
one can apply the following scaling with respect to $\alpha$:
$\lambda=\alpha \mu$ and 
\begin{equation}
\label{eq:scalpoly}
\widetilde{P}(\mu):=(\trop p(\alpha))^{-1}P(\alpha \mu)=\sum_{j=0}^d\widetilde{A}_j\mu^j\enspace.
\end{equation}
The eigenvalue $\zeta$ of $P$ corresponds to the eigenvalue
 $\xi=\alpha^{-1}\zeta$ of $\widetilde{P}$.
If $\zeta$ is of order of magnitude $\alpha$, we get that 
$\xi$ is of order $1$, so that we expect that the numerical computation
of $\xi$ will behaves well, and 
better than the numerical computation of $\zeta$.

The tropical scaling algorithm, already proposed in~\cite{posta09},
consists in doing 
a different scaling for every expected order of magnitude
of the eigenvalues as follows.
Let us call \NEW{C{\&}QZ}, the algorithm consisting in computing
the eigenvalues (and eigenvectors) of the matrix polynomial $P$ by applying
QZ algorithm after a companion linearization $L$ of $P$.
For every tropical 
root $\alpha=\alpha_{k_i}$, $i=1,\ldots, q$, 
we compute the  eigenvalues of the scaled matrix polynomial
$\widetilde{P}$ in~\eqref{eq:scalpoly} using C{\&}QZ algorithm,
and we multiply them by $\alpha_{k_i}$.
We get $nd$ candidate eigenvalues of $P$.
Then, we keep among these candidate eigenvalues only the $n(k_i-k_{i-1})$ ones
having the closest
modulus to $\alpha_{k_i}$, and we say that these
candidate eigenvalues are {\em assigned} to $\alpha_{k_i}$. 
The other eigenvalues returned
by the C{\&}QZ algorithm are discarded. 
Note that we measure the distance between two moduli 
by the absolute value of the difference of the log of these moduli.
Finally, the whole spectrum is obtained as the union of the different groups 
of $n(k_i-k_{i-1})$ candidate eigenvalues assigned to the different tropical roots, for $1\leq i\leq q$.
If the moduli of the values of some groups interleave, meaning that
for $i<j$, a candidate eigenvalue $\lambda$ assigned to a tropical root
$\alpha_{k_i}$
is such that $|\lambda|>|\lambda'|$ for some candidate eigenvalue $\lambda'$
assigned to a tropical root $\alpha_{k_j}$, the algorithm returns
``fail'' (the assignment is inconsistent).

\subsection{Condition numbers and backward errors with or without tropical scalings}
\label{sec-scaling2}

Let us now give some new insight of the interest of the tropical scaling,
and its relation with the bounds of Sections~\ref{sec:Main_results}
and~\ref{sec-upperbound}.

Let us consider 
the condition numbers and backward errors of eigenvalues defined 
with respect to the relative errors (and left and right eigenvectors)
as in~\cite[Theorems 4 and 5]{tiss00}.
Recall that the accuracy of the computation of an eigenvalue can be evaluated 
by the product of its condition number and of its backward error.
In our case, this accuracy can be computed
either by using the condition number $\kappa_P$ 
and the backward error $\eta_P$ 
with respect to $P$ or the condition number $\kappa_L$ and  the 
backward error $\eta_L$ with respect to $L$. 

Since the above scaling does not affect the condition number of the eigenvalue
with respect to the matrix polynomial~\cite{hamarling},
and since QZ algorithm is backward stable, 
the accuracy of C{\&}QZ on a given eigenvalue will be improved after scaling
as soon as $\eta_P/\eta_L$ decreases or $\kappa_L/\kappa_P$  decreases
(here $\eta_P$, $\kappa_P$, $\eta_L$, $\kappa_L$  are
respectively the backward error and the condition number
of the eigenvalue with respect to the matrix polynomial and with 
respect to its linearization).
The last property writes
 $\kappa_{\widetilde{L}}/\kappa_{\widetilde{P}}<\kappa_L/\kappa_P$,
where $\kappa_{\widetilde{P}}$ is the condition number
of the eigenvalue $\xi=\alpha^{-1}\zeta$
with respect to  $\widetilde{P}$, and $\kappa_{\widetilde{L}}$
is the same with respect to the linearization of $\widetilde{L}$ 
of $\widetilde{P}$.

The results of~\cite{high2006} and~\cite{hamarling} show that,
 for quadratic pencils, $\kappa_L/\kappa_P$  is lower bounded and that 
under some conditions, it is
also upper bounded  when using the
scaling of~\cite{vandooren00} and the tropical scaling respectively.
This means that in that cases these scalings are almost minimizing the
ratio $\kappa_L/\kappa_P$ (among all scalings)
hence are almost maximizing the efficiency of C{\&}QZ.
These results have been generalized and improved 
in~\cite{marchesinitisseur} and~\cite{shti14},
for the companion linearization and several
linearizations of matrix polynomials  respectively.
In particular, using the property that 
the scaled polynomial $\widetilde{P}$ of~\eqref{eq:scalpoly}
satisfies $\max_{0\leq j\leq d}\|\widetilde{A}_j\|=1$,
it is shown that the following holds (for the companion linearization):
\[ \frac{1}{\sqrt{d}} \leq \kappa_L/\kappa_P \enspace,
\quad \kappa_{\widetilde{L}}/\kappa_{\widetilde{P}}\leq C_d 
\max(|\xi|,1/|\xi|)^{d} \enspace,\]
where $C_d$ is a constant depending only on $d$.
From these two inequalities, we see that the above scaling is almost maximizing 
the efficiency of C{\&}QZ as soon as $\xi=\alpha^{-1}\zeta$
is close to $1$ or equivalently $\alpha$ is a good order of magnitude 
of $\zeta$.
Hence the tropical scaling algorithm described in previous section
leads to almost the
best efficiency of C{\&}QZ as soon as 
all the ratios $\zeta_{n(k-1)+l}/\alpha_k$ with $k=1,\ldots, d$ and
$\ell=1,\ldots, n$ are close to $1$.
Several numerical examples are also given in~\cite{marchesinitisseur} 
and~\cite{shti14}.

\subsection{Condition numbers and backward errors in examples}
\label{sec:examples-cond}

In the following we compare
on the matrix polynomials of Section~\ref{sec:numexp},
the numerical accuracy of the tropical scaling algorithm
to the one of the standard C{\&}QZ algorithm.
As said before this accuracy can be measured either using 
$\kappa_L$  or $\eta_P$, or
the ratios $\kappa_L/\kappa_P$  or $\eta_P/\eta_L$.
Note that the condition numbers are computed using the 
result of the computation of the eigenvalues and eigenvectors,
so only approximate condition numbers will be obtained.
Actually, without scaling some of the eigenvalues will be aberrant.
Moreover, the numerical values of
$\kappa_P$ and $\eta_L$ may change after scaling, although they
should not change theoretically.

In Figure~\ref{cond_ex61}, we plot the condition numbers and backward errors 
for the polynomial $P_1$ of Example~\ref{matx_exp_2}.
The number of eigenvalues is 27, but they moduli are close to
the $4$ tropical eigenvalues as shown in Figure~\ref{tbl_mtx1}.
So, we see only $4$ points in each graph corresponding to the
computations using the tropical scaling.
However, in the graphs  corresponding to the
computations without scaling, we see an additional point 
with approximate modulus $10^{-3}$ corresponding
to aberrant values (there are no eigenvalues with this order 
of magnitude).
We see also that the condition number of the linearization
and the backward error of the matrix polynomial are much improved
by using the tropical scaling, in particular for the smallest eigenvalues.

Similar conclusions are obtained for the polynomials $P_2$ and $P_3$ 
of Example~\ref{matx_exp_2}.

\begin{figure}[htbp]
\begin{center}
\includegraphics[scale=0.37]{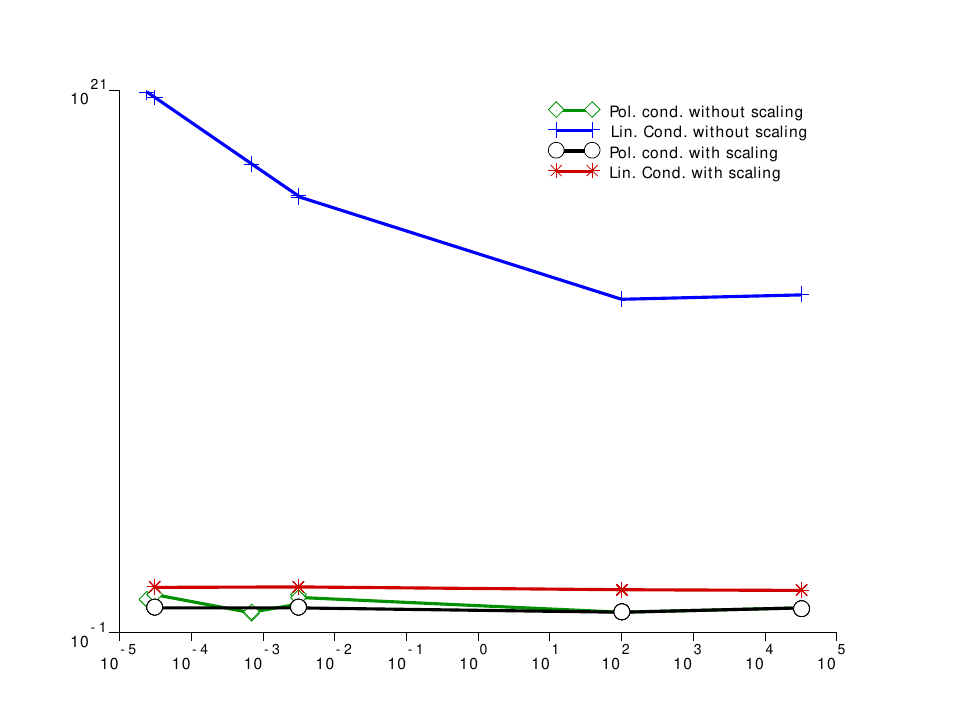}%
\hfill
\includegraphics[scale=0.37]{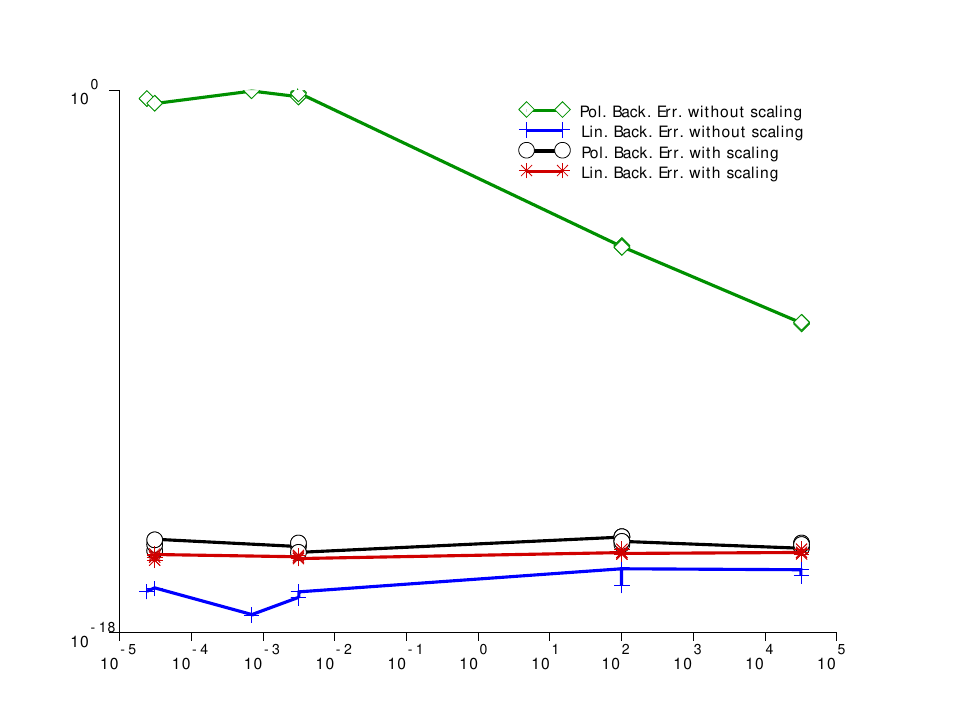}%
\end{center}
\caption{Eigenvalues condition numbers and backward errors for $P_1$ of Example~\ref{matx_exp_1}}
\label{cond_ex61}
\end{figure}

\begin{figure}[htbp]
\begin{center}
\includegraphics[scale=0.37]{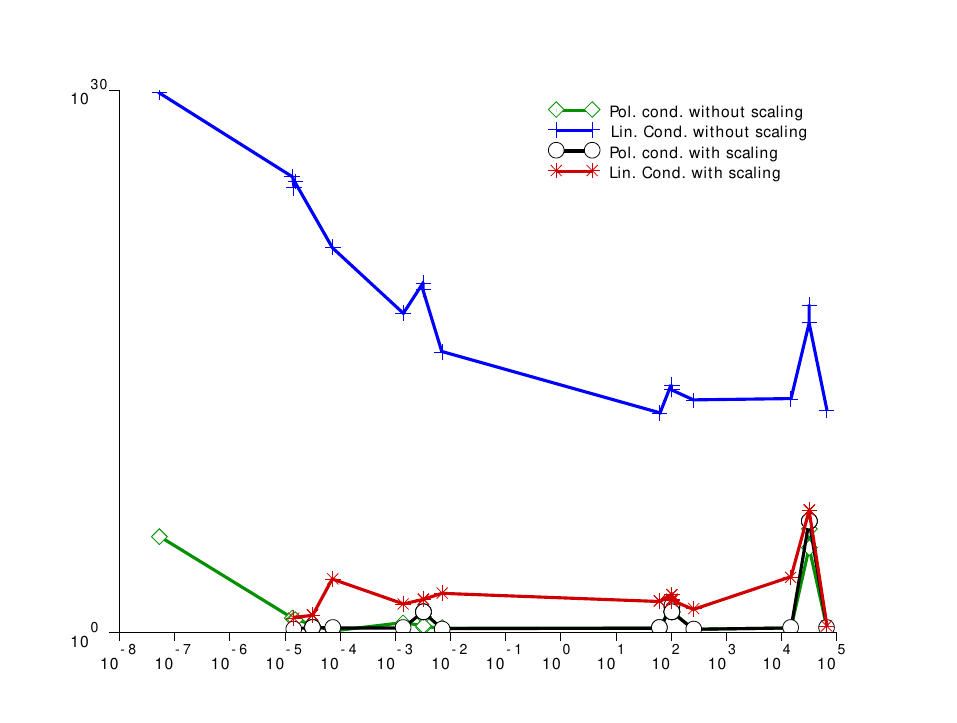}%
\hfill
\includegraphics[scale=0.37]{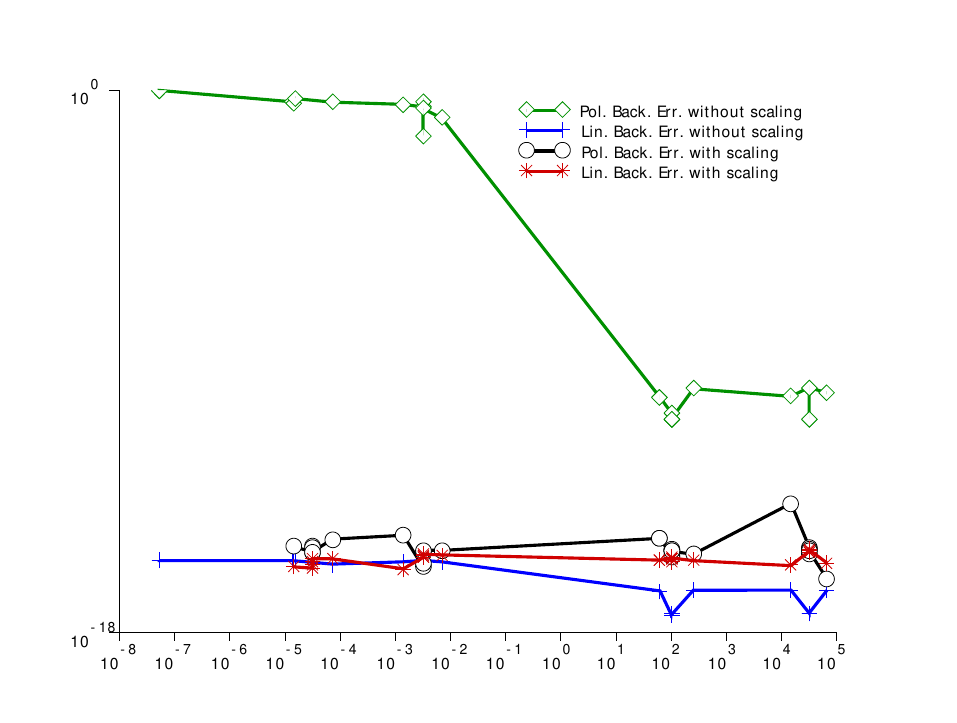}%
\end{center}
\caption{Eigenvalues condition numbers and backward errors for $P_2$
of Example~\ref{matx_exp_2}}
\label{cond_ex62P2}
\end{figure}

\begin{figure}[htbp]
\begin{center}
\includegraphics[scale=0.37]{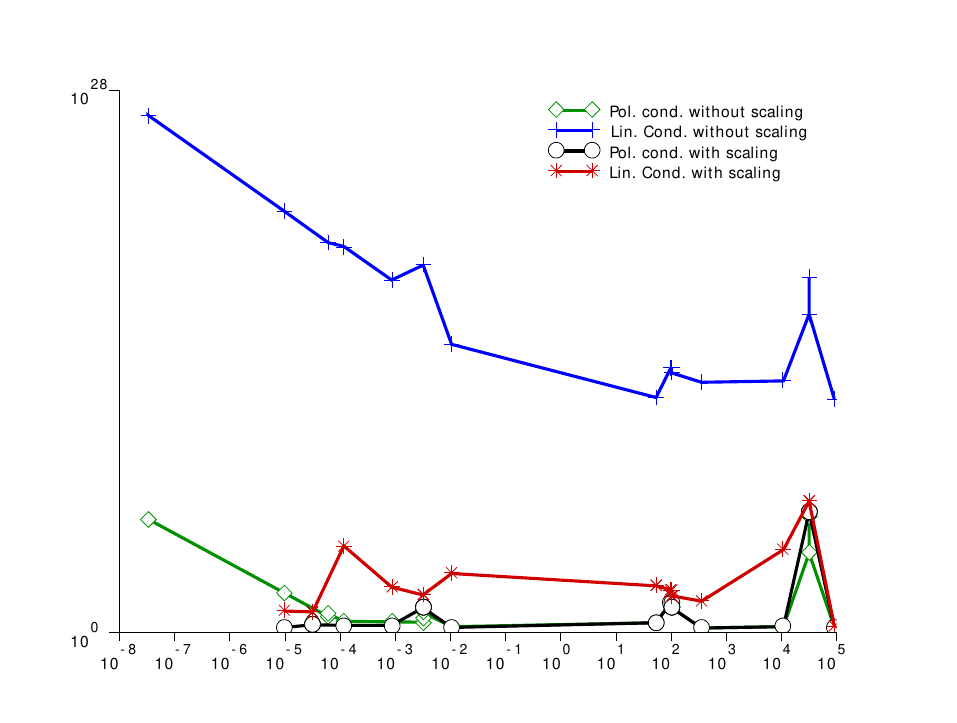}%
\hfill
\includegraphics[scale=0.37]{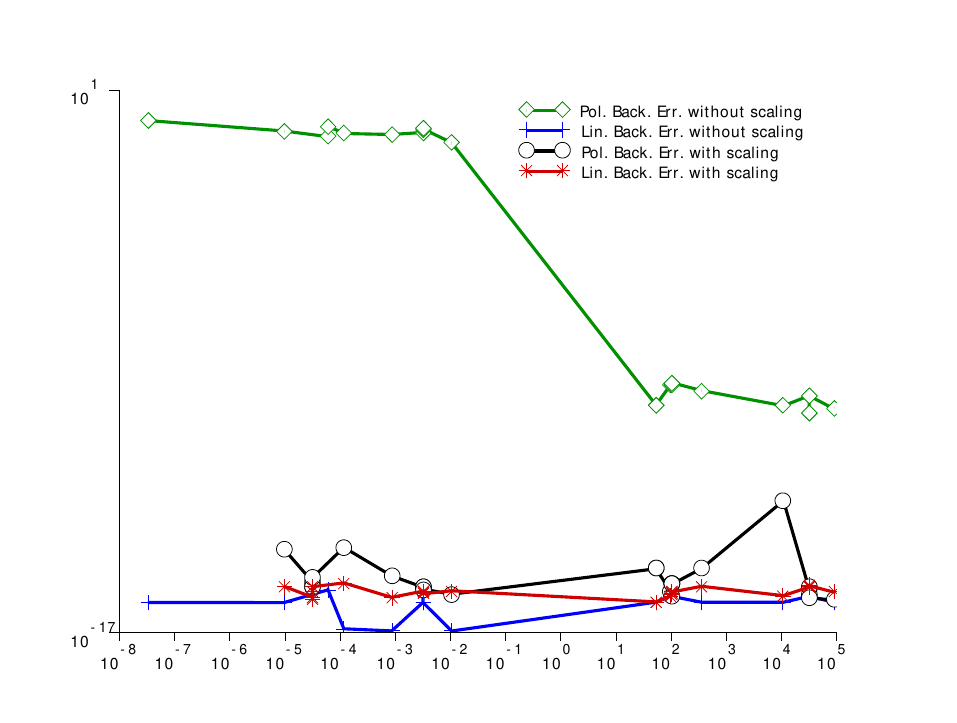}%
\end{center}
\caption{Eigenvalues condition numbers and backward errors for $P_3$
of Example~\ref{matx_exp_2}
without scaling or with
scaling with respect to Schatten $1$-norm}
\label{cond_ex62P3}
\end{figure}

We finally consider the polynomial $P_4$ of Example~\ref{matx_exp_3}
with $a_1=a_2=10^4$.
In this case, we can compute the eigenvalues using the diagonal
structure of the matrices, for verification.
We saw that the tightness of the lower bound decreases when $b$ increases.
The efficiency of the tropical scaling also decreases with $b$.
Compare the results for $b=10$ and $b=100$
in Figures~\ref{cond_ex63_p4_10} and \ref{cond_ex63_p4_100}, respectively.
When $b=10^3$, the tropical scaling fails, leading to an error 
for the logarithm of the modulus greater than $1.99$.

\begin{figure}[htbp]
\begin{center}
\includegraphics[scale=0.37]{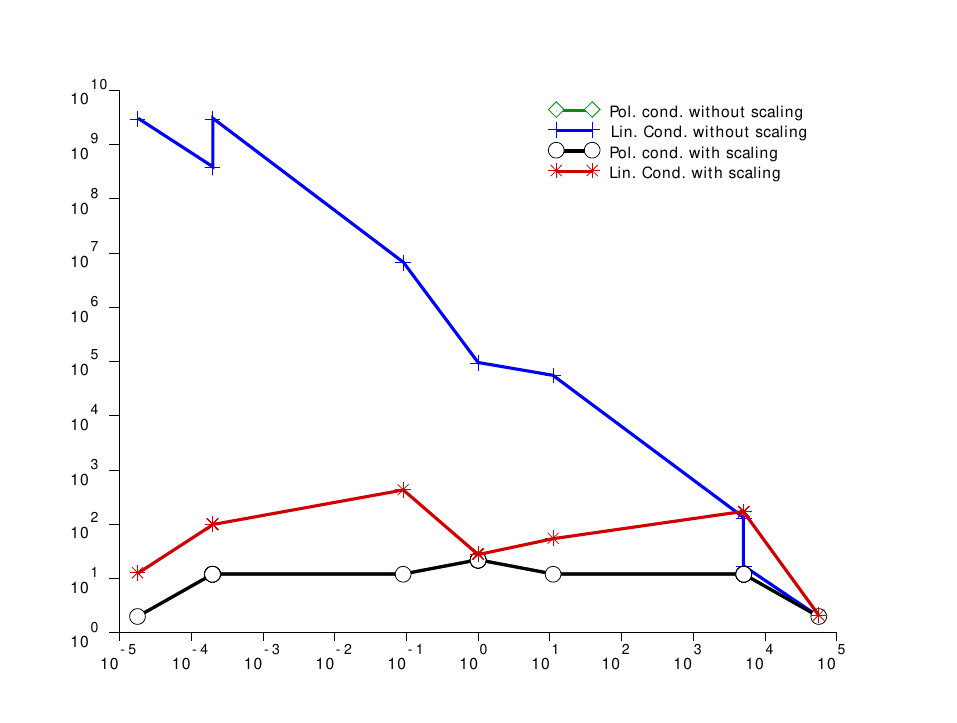}%
\hfill
\includegraphics[scale=0.37]{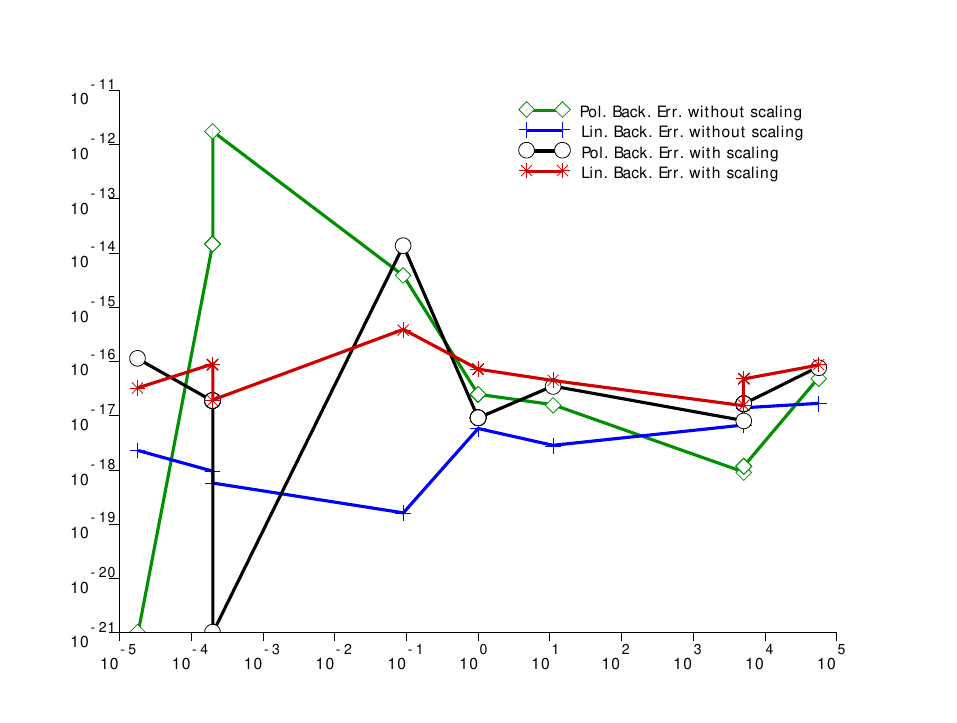}%
\end{center}
\caption{Eigenvalues condition numbers and backward errors for $P_4$
with $b=10$, without scaling or with
scaling with respect to  Frobenius norm}
\label{cond_ex63_p4_10}
\end{figure}

\begin{figure}[htbp]
\begin{center}
\includegraphics[scale=0.37]{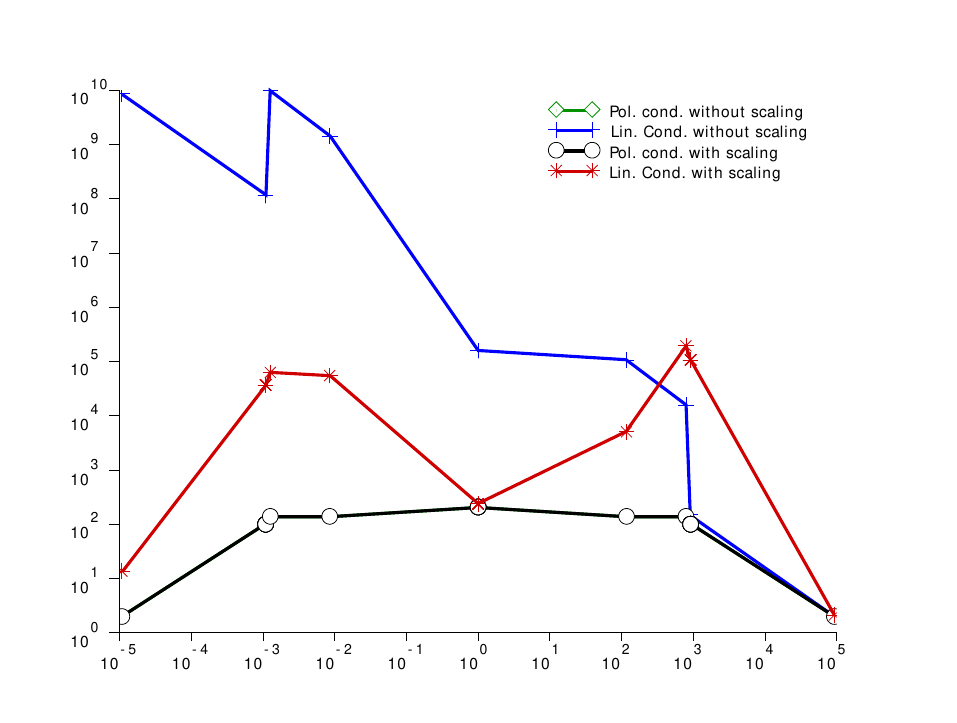}%
\hfill
\includegraphics[scale=0.37]{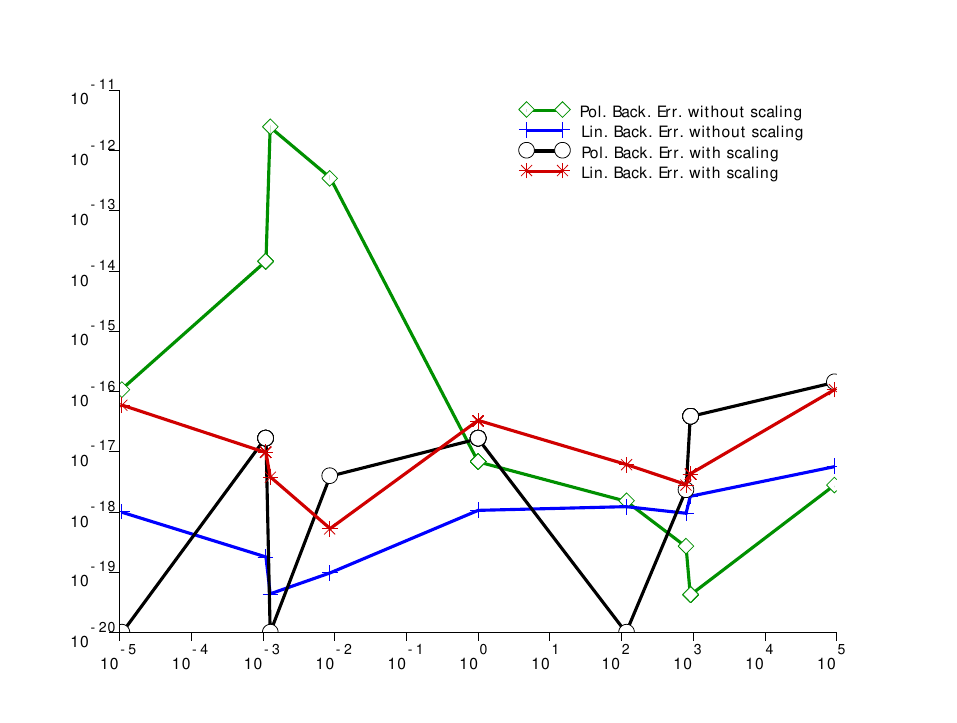}%
\end{center}
\caption{Eigenvalues condition numbers and backward errors for $P_4$
with $b=100$, without scaling or with
scaling with respect to  Frobenius norm}
\label{cond_ex63_p4_100}
\end{figure}

\bibliographystyle{plain}
\bibliography{references}

\end{document}